\documentclass[11pt,a4paper,reqno]{amsart}
\usepackage[utf8]{inputenc}
\usepackage[T1]{fontenc}
\usepackage{lmodern}
\usepackage{amsmath,amssymb,amsthm}
\usepackage{commath,mathtools}
\usepackage{xcolor}
\usepackage{mathrsfs,dsfont}
\usepackage[right,nolabel]{showlabels}
\usepackage{enumitem}
\usepackage{chngcntr}
\usepackage{marginnote}
\usepackage{doi,hyperref}

\newcommand{\wt}[1]{\widetilde{#1}}
\newcommand{\cB}{\mathcal{B}}
\newcommand{\cW}{\mathcal{W}}

\newcommand{\cQ}{\mathcal{Q}}
\newcommand{\cY}{\mathcal{Y}}
\newcommand{\cI}{\mathcal{I}}
\newcommand{\cF}{\mathcal{F}}
\newcommand{\cS}{\mathcal{S}}
\newcommand{\cM}{\mathcal{M}}

\newcommand{\cP}{\mathcal{P}}
\newcommand{\cG}{\mathcal{G}}
\newcommand{\ELL}{\ell}

\newcommand{\DD}{\mathcal{D}}
\newcommand{\bD}{\mathbb{D}}

\newcommand{\R}{\mathbb{R}}
\newcommand{\Z}{\mathbb{Z}}

\newcommand{\T}{\mathcal{T}}
\newcommand{\E}{\mathcal{E}}

\newcommand{\ve}{\varepsilon}

\newcommand{\Roots}{\mathsf{Roots}}
\newcommand{\Bad}{\mathsf{Bad}}
\newcommand{\Br}{\mathsf{Branch}}
\newcommand{\Peer}{\mathsf{Peer}}
\newcommand{\Ch}{\mathsf{Ch}}
\newcommand{\Stop}{\mathsf{Stop}}
\newcommand{\Tree}{\mathsf{Tree}}
\newcommand{\Top}{\mathsf{Top}}
\newcommand{\Good}{\mathsf{Good}}
\newcommand{\Next}{\mathsf{Next}}

\newcommand{\Div}{\mathsf{Sh}}
\newcommand{\Fin}{\mathsf{Fin}}
\newcommand{\Evil}{\mathsf{Evil}}
\newcommand{\Near}{\mathsf{Near}}
\newcommand{\End}{\mathsf{End}}

\newcommand{\one}{\mathds{1}}

\newcommand{\Bi}{\mathsf{B}_{int}}
\newcommand{\Be}{\mathsf{B}_{ext}}
\newcommand{\Sh}{\mathsf{Sh}}

\newcommand\restr[2]{{% we make the whole thing an ordinary symbol
		\left.\kern-\nulldelimiterspace % automatically resize the bar with \right
		#1 % the function
		\right|_{#2} % this is the delimiter
}}

\newcommand{\HH}{\mathcal{H}^1}
\newcommand{\TT}{\mathbb{T}}

\DeclareMathOperator{\supp}{supp}

\DeclareMathOperator{\lip}{Lip}
\DeclareMathOperator{\spn}{span}
\DeclareMathOperator{\diam}{diam}
\DeclareMathOperator{\dist}{dist}
\DeclareMathOperator{\Fav}{Fav}

\newtheorem{theorem}{Theorem}
\newtheorem*{theorem*}{Theorem}
\newtheorem{lemma}[theorem]{Lemma}
\newtheorem{cor}[theorem]{Corollary}
\newtheorem{prop}[theorem]{Proposition}
\newtheorem{claim}[theorem]{Claim}
\theoremstyle{remark}
\newtheorem{definition}[theorem]{Definition}
\newtheorem*{definition*}{Definition}
\newtheorem{question}[theorem]{Question}
\newtheorem{problem}[theorem]{Problem}
\newtheorem{remark}[theorem]{Remark}

\numberwithin{equation}{section}
\numberwithin{theorem}{section}
\counterwithin{figure}{section}

\title[Favard length and quantitative rectifiability]{Favard length and quantitative rectifiability}

\author{Damian D{\k a}browski}
\address{Department of Mathematics and Statistics\\ University of Jyv\"askyl\"a,
	P.O. Box 35 (MaD)\\
	FI-40014 University of Jyv\"askyl\"a\\
	Finland} 

\email{\href{mailto:damian.m.dabrowski@jyu.fi}{damian.m.dabrowski@jyu.fi}}

\thanks{D.D. is supported by the Research Council of Finland postdoctoral grant \textit{Quantitative rectifiability and harmonic measure beyond the Ahlfors-David-regular setting}, grant No. 347123.}

\keywords{Favard length, Besicovitch projection theorem, Buffon's needle, Lipschitz graph, quantitative rectifiability, Vitushkin's conjecture}

\begin{document}
	\maketitle

	\begin{abstract}
		The Favard length of a Borel set $E\subset\R^2$ is the average length of its orthogonal projections. We prove that if $E$ is Ahlfors 1-regular and it has large Favard length, then it contains a big piece of a Lipschitz graph. This gives a quantitative version of the Besicovitch projection theorem. As a corollary, we answer questions of David and Semmes and of Peres and Solomyak. We also make progress on Vitushkin's conjecture.
	\end{abstract}
	\tableofcontents
	\section{Introduction}
	{Favard length} of a Borel set $E\subset\R^2$ is the average length of its orthogonal projections. More precisely,
	\begin{equation*}
	\Fav(E) = \int_0^1 \HH(\pi_\theta(E))\, d\theta.
	\end{equation*}
	This quantity is also referred to as \emph{Buffon's needle probability}.
	
	The Besicovitch projection theorem \cite{besicovitch1939on}, which is one of the foundational results of geometric measure theory, asserts that if a Borel set $E\subset\R^2$ satisfies $0<\HH(E)<\infty$ and $\Fav(E)>0$, then there exists a Lipschitz graph $\Gamma$ such that
	\begin{equation*}
	\HH(E\cap \Gamma)>0.
	\end{equation*}
	It is a purely qualitative result: it gives no estimates on $\lip(\Gamma)$ or $\HH(E\cap \Gamma)$. 
	
	The \emph{Favard length problem} consists of quantifying Besicovitch's theorem. Different versions of the problem have been proposed by many authors, see e.g. \cite[Section 7.3]{peres2002howlikely}, \cite[Problem 4]{mattila2004hausdorff}, \cite[p. 857]{david1993quantitative}, \cite[p. 29]{david1993analysis}, \cite[Question 1]{orponen2020plenty}. Our main theorem solves the Favard length problem in the class of Ahlfors regular sets.
	\begin{definition*}
		We say that $E\subset\R^2$ is \emph{Ahlfors regular with constant} $A\ge 1$ if $E$ is closed and for all $x\in E,$ $0<r<\diam(E),$ we have
		\begin{equation*}
		A^{-1}r\le \HH(E\cap B(x,r))\le Ar.
		\end{equation*}
		If the constant $A$ does not matter, we simply say that $E$ is Ahlfors regular.
	\end{definition*}
	
	\begin{theorem}\label{thm:main}
		Let $0<\kappa<1$ and $A\ge 1$. Suppose that $E\subset \R^2$ is Ahlfors regular with constant $A$, and that $\Fav(E)\ge\kappa\HH(E)$.
		
		Then, there exists a Lipschitz graph $\Gamma$ with $\lip(\Gamma)\lesssim_{\kappa,A}1$ such that
		\begin{equation*}
			\HH(E\cap \Gamma)\gtrsim_{\kappa,A} \HH(E).
		\end{equation*}
	\end{theorem}
	\begin{remark}
		We get explicit estimates for $\lip(\Gamma)$ and $\HH(E\cap \Gamma)$ in terms of $A$ and $\kappa$, see \eqref{eq:mainthmLib} and \eqref{eq:mainthmSize}.
	\end{remark}
	\begin{remark}
		It is not difficult to reduce the proof of \thmref{thm:main} to the case of finite unions of parallel segments, see Appendix \ref{app:paral}. Consequently, we do not need to use the classical Besicovitch projection theorem in our proof.
	\end{remark}

	As a corollary, we get positive answers to a question of David and Semmes \cite[p. 29]{david1993analysis}, a question of Orponen \cite[Question 1]{orponen2020plenty}, and a problem posed by Peres and Solomyak \cite[Question 7.5]{peres2002howlikely} and Mattila \cite[Problem 4]{mattila2004hausdorff}. Finally, we make progress on Vitushkin's conjecture by showing that sets with uniformly large Favard length have large analytic capacity. We describe these corollaries in the following subsections.
	\begin{remark}
		All the problems mentioned in this introduction have natural counterparts for $n$-dimensional projections in $\R^d$. It seems plausible the proof of \thmref{thm:main} can be generalized to that setting, and we hope to address this in a later version of this preprint.
	\end{remark}
	
	\subsection{Quantitative rectifiability}
	The field of quantitative rectifiability was born in the 1990s after the breakthrough works of Jones \cite{jones1991traveling} and David and Semmes \cite{david1991singular,david1993analysis}, motivated by the study of singular integrals and analytic capacity. The central objects of interest in this area are \emph{uniformly rectifiable sets}, which are quantitative counterparts of the usual rectifiable sets. An important subclass of uniformly rectifiable sets are the sets with {big pieces of Lipschitz graphs}.
	\begin{definition*}
		A set $E\subset\R^2$ has \emph{big pieces of Lipschitz graphs} (BPLG) if $E$ is Ahlfors regular and there exist constants $L>0$ and $\theta\in (0,1)$ such that for every $x\in E$ and $0<r<\diam(E)$ there exists a Lipschitz graph $\Gamma_{x,r}$ such that $\lip(\Gamma_{x,r})\le L$ and
		\begin{equation*}
		\HH(E\cap B(x,r)\cap\Gamma_{x,r})\ge \theta r.
		\end{equation*}
	\end{definition*}
	In \cite[p. 29]{david1993analysis} David and Semmes asked the following.	
	\begin{question}[David-Semmes]\label{q:DS}
		Suppose that $E\subset\R^2$ is an Ahlfors regular set with \emph{uniformly large Favard length} (ULFL), i.e. it satisfies $\Fav(E\cap B(x,r))\gtrsim r$ for all $x\in E$ and $0<r<\diam(E)$. Is $E$ uniformly rectifiable?
	\end{question} 
	A positive answer follows immediately from \thmref{thm:main}.
	\begin{cor}\label{cor:ULFL}
		Suppose that $E\subset \R^2$ is Ahlfors regular and satisfies ULFL. Then, $E$ has BPLG.
	\end{cor}
	\begin{remark}
		It is easy to check that the converse implication also holds: sets with BPLG satisfy ULFL, simply because large subsets of Lipschitz graphs have many big projections, see \cite[Proposition 1.4]{martikainen2018characterising}. Thus, ULFL is equivalent to BPLG in the class of Ahlfors regular sets.
	\end{remark}
	\begin{proof}
		Let $x\in E$ and $0<r<\diam(E)$. By \cite[Lemma 2.2]{bortz2022coronizations} there exists a compact set $E_{x,r}$ which is Ahlfors regular with constant $\lesssim 1$, and such that $E\cap B(x,r/3)\subset E_{x,r}\subset E\cap B(x,r)$. Since
		\begin{equation*}
		\Fav(E_{x,r})\ge \Fav(E\cap B(x,r/3))\gtrsim r\gtrsim \HH(E_{x,r}),
		\end{equation*}
		it follows from \thmref{thm:main} that there exists a Lipschitz graph $\Gamma_{x,r}$ with $\lip(\Gamma_{x,r})\lesssim 1$ and
		\begin{equation*}
		\HH(E\cap B(x,r)\cap \Gamma_{x,r})\ge\HH(E_{x,r}\cap \Gamma_{x,r})\gtrsim \HH(E_{x,r})\gtrsim r.
		\end{equation*}
	\end{proof}

	\subsubsection*{Related work}
	\thmref{thm:main} implies the (planar versions of) results from three recent papers marking progress on the Favard length problem \cite{martikainen2018characterising, orponen2020plenty, dabrowski2022quantitative}. We discuss these results below.
	
	In \cite[p. 857]{david1993quantitative} David and Semmes asked a question similar to Question \ref{q:DS}. The only difference was that the ULFL condition was strengthened to the \emph{plenty of big projections} (PBP) condition: for every $x\in E$, $0<r<\diam(E)$ there exists $\theta_{x,r}\in [0,1]$ such that
	\begin{equation*}
	\HH(\pi_\theta(E\cap B(x,r)))\ge \delta r\quad \text{for all } |\theta-\theta_{x,r}|\le \delta.
	\end{equation*}
	
	This conjecture (and its higher dimensional version) was recently confirmed in a breakthrough article of Orponen \cite{orponen2020plenty}.	
	In the same paper \cite[Question 1]{orponen2020plenty} it was asked whether a single-scale version of PBP, i.e. 
	\begin{equation}\label{eq:Orpques}
	\HH(\pi_\theta(E\cap B(0,1)))\ge\delta\quad\text{for all }|\theta-\theta_0|\le \delta,
	\end{equation}
	already implies the existance of a big piece of a Lipschitz graph, with controlled constants. Since \eqref{eq:Orpques} implies $\Fav(E\cap B(0,1))\gtrsim \delta^2$, this follows readily from \thmref{thm:main}.
	
	 We remark that the proof of \thmref{thm:main} is very different from that in \cite{orponen2020plenty}, and we do not use Orponen's result. On the other hand, we crucially use the methods and results from two other papers where progress on the Favard length problem was made, namely \cite{martikainen2018characterising} and \cite{dabrowski2022quantitative}. 
	 
	 In \cite{martikainen2018characterising} it is shown that if $E$ is Ahlfors regular, and $F\subset E$ is large and satisfies
	 \begin{equation}\label{eq:MOasssump}
	 \int_I \|\pi_\theta\HH|_{F}\|_{L^2}^2\, d\theta\lesssim 1
	 \end{equation}
	 for a large interval $I\subset [0,1]$, then $F$ contains a big piece of a Lipschitz graph. Just like the PBP condition, assumption \eqref{eq:MOasssump} should be thought of as a strengthening of $\Fav(E)\gtrsim 1$: it follows from Cauchy-Schwarz inequality that
	 \begin{multline*}
	 1\lesssim \HH(F) = \frac{1}{\HH(I)}\int_I \pi_\theta \HH|_F(\pi_\theta(F))\, d\theta
	  \lesssim \int_I \|\pi_\theta\HH|_{F}\|_{L^2}\cdot\HH(\pi_\theta(F))^{1/2}\, d\theta\\
	  \le \left(\int_I \|\pi_\theta\HH|_{F}\|_{L^2}^2\, d\theta\right)^{1/2} \Fav(F)^{1/2}
	  \overset{\eqref{eq:MOasssump}}{\lesssim}\Fav(E)^{1/2}.
	 \end{multline*}
	 In the proof of \thmref{thm:main} we use the main technical proposition from \cite{martikainen2018characterising} as a black-box, see \thmref{thm:MOmainprop}.
	 
	 In \cite{dabrowski2022quantitative} it was shown that if $E$ is Ahlfors regular and
	 \begin{equation*}
	 \|\pi_\theta\HH|_E\|_{L^\infty}\lesssim 1\quad\text{for } \theta\in G,
	 \end{equation*}
	 where $G\subset [0,1]$ satisfies $\HH(G)\sim 1$, then $E$ contains a big piece of a Lipschitz graph. Note that this condition implies $\Fav(E)\gtrsim 1$. Compared to \eqref{eq:MOasssump}, the gain is that the set of good directions $G$ is only assumed to have large measure, but does not have to be an interval. This came at a price of strengthening the $L^2$ estimate on projections to an $L^\infty$ estimate. The proof of \thmref{thm:main} uses many of the ideas and tools from \cite{dabrowski2022quantitative}, and we discuss this in detail in Section \ref{sec:outline}.

	Regarding other work on the Favard length problem, in \cite{david1993quantitative} David and Semmes showed that a single big projection together with \emph{weak geometric lemma} gives a big piece of a Lipschitz graph. In \cite{chang2022structure} a version of \thmref{thm:main} was proven with assumption $\Fav(E)\gtrsim\HH(E)$ replaced by ``almost maximality of Favard length'', and with a correspondingly stronger conclusion.
	
	Regarding the classical Besicovitch projection theorem, its higher dimensional generalization is due to Federer \cite{federer1947varphi}, see also \cite{white1998newproof}. A version for metric spaces is due to Bate \cite{bate2020purely}. Different variants of the Besicovitch projection theorem have also been shown for \emph{generalized projections} \cite{hovila2012besicovitch,bond2011circular,cladek2020upper,bongers2021transversal,davey2022quantification}.
	\subsection{Vitushkin's conjecture} In 1967 Vitushkin conjectured \cite{vitushkin1967analytic} that a compact set $E\subset\R^2$ has positive \emph{analytic capacity} if and only if $\Fav(E)>0$. Instead of defining analytic capacity, which is denoted by $\gamma(E)$, let us only say that $\gamma(E)=0$ if and only if $E$ is \emph{removable for bounded analytic functions}: all bounded analytic function defined on $\mathbb{C}\setminus E$ are constant.
	
	Over the last 50 years plenty of work has been done towards the solution of Vitushkin's conjecture, and we only scratch the surface by saying that it was shown to be true for sets $E$ with $\HH(E)<\infty$ in \cite{calderon1977cauchy,david1998unrectifiable} and false for sets $E$ with infinite length \cite{mattila1986smooth,jones1988positive}. See \cite{tolsa2014analytic,verdera2021birth} for the full historical account.
	
	There is one outstanding open question remaining from the Vitushkin's conjecture.
	\begin{question}\label{q:qualVit}
		Suppose that $E\subset\R^2$ is compact. Is it true that
		\begin{equation}\label{eq:Favgamma}
		\Fav(E)>0\quad\Rightarrow\quad\gamma(E)>0?
		\end{equation}
	\end{question}
	One can also ask whether a stronger, quantitative result is true.
	\begin{question}\label{q:quantVit}
		Suppose that $E\subset\R^2$ is compact and $\Fav(E)\ge \kappa \diam(E)$ for some $\kappa>0$. Is it true that
		\begin{equation}\label{eq:quanVit}
		\gamma(E)\gtrsim_{\kappa}\diam(E)?
		\end{equation}
	\end{question}
	\begin{remark}
		In \eqref{eq:quanVit} the best dependence on $\kappa$ one can hope for is linear, which would be equivalent to an estimate $\gamma(E)\gtrsim\Fav(E)$. This is because $\gamma(E)\sim\HH(E)\sim\Fav(E)$ for subsets of lines (see \cite[Proposition 1.22]{tolsa2014analytic} or \cite{pommerenke1960uber}) and circles \cite{murai1990formula}.
	\end{remark}
	The answer to Question \ref{q:qualVit} is known to be positive for sets with $\sigma$-finite length by \cite{calderon1977cauchy,tolsa2003painleve}, but is open otherwise. Until now Question \ref{q:quantVit} was open even for Ahlfors regular sets. Thanks to \thmref{thm:main} we make progress on Question \ref{q:quantVit}.
	
	Firstly, we answer it in the positive for general compact sets satisfying ULFL. We do not assume $\HH(E)<\infty$. Before it was not even known whether such sets have positive analytic capacity.
	\begin{cor}\label{cor:Vit1}
		Suppose that a compact set $E\subset\R^2$ satisfies the ULFL condition, i.e. there exists $\kappa>0$ such that $\Fav(E\cap B(x,r))\ge \kappa r$ for all $x\in E$ and $0<r<\diam(E)$. 	
		Then,
		\begin{equation*}
		\gamma(E)\gtrsim_\kappa \diam(E).
		\end{equation*}
	\end{cor}
\begin{proof}
	This follows immediately from Corollary \ref{cor:ULFL} together with the proof of \cite[Theorem 1.3]{dabrowski2022analytic}, see \cite[Remark 1.9]{dabrowski2022analytic}. In \cite[Theorem 1.3]{dabrowski2022analytic} Villa and the author showed a version of Corollary \ref{cor:Vit1} with ULFL replaced by the stronger PBP condition. In the proof we used Orponen's result \cite{orponen2020plenty} as a black-box, and replacing all references to \cite{orponen2020plenty} by Corollary \ref{cor:ULFL} gives Corollary \ref{cor:Vit1}.
\end{proof}	
	
	Secondly, we get a positive answer to Question \ref{q:quantVit} for Ahlfors regular sets.
	\begin{cor}\label{cor:Vit2}
		Suppose that a compact set $E\subset\R^2$ is Ahlfors regular with constant $A$ and satisfies $\Fav(E)\ge\kappa\diam(E)$. Then,
		\begin{equation}\label{eq:corVit}
		\gamma(E)\gtrsim_{\kappa,A}\diam(E).
		\end{equation}
	\end{cor}
	\begin{proof}
		By \thmref{thm:main} there exists a Lipschitz graph $\Gamma$ with $\lip(\Gamma)\lesssim_{A,\kappa} 1$ and such that $\HH(E\cap \Gamma)\gtrsim_{A,\kappa} \HH(E)$. Without loss of generality, we may assume that $\diam(\Gamma)\lesssim\diam(E)$ simply by truncating the graph $\Gamma$.
		
		By Verdera's quantitative solution to Denjoy's conjecture \cite[Theorem 4.31]{tolsa2014analytic}
		\begin{equation*}
		\gamma(E)\gtrsim \frac{\HH_\infty(E\cap \Gamma)^{3/2}}{\HH(\Gamma)^{1/2}}\gtrsim_{\kappa,A} \frac{\HH_\infty(E\cap \Gamma)^{3/2}}{\diam(\Gamma)^{1/2}}\gtrsim \frac{\HH_\infty(E\cap \Gamma)^{3/2}}{\diam(E)^{1/2}}.
		\end{equation*}
		Finally, since $\Gamma$ is a bilispchitz image of $\R$ with bilispchitz constant $\lesssim_{\kappa,A} 1$, and the Hausdorff content $\HH_\infty$ restricted to $\R$ is equal to $\HH|_{\R}$, we get that
		\begin{equation*}
		\HH_\infty(E\cap \Gamma)\gtrsim_{A,\kappa}\HH(E\cap \Gamma)\gtrsim_{A,\kappa}\HH(E)\sim_A \diam(E).
		\end{equation*}
		Together with the previous estimates this gives $\gamma(E)\gtrsim_{\kappa,A}\diam(E)$.
	\end{proof}

	\subsubsection*{Related work} Question \ref{q:quantVit} was answered in the positive for sets satisfying PBP in \cite[Theorem 1.3]{dabrowski2022analytic}. In \cite{chang2017analytic} Chang and Tolsa showed a lower bound on analytic capacity of the form
	\begin{equation}\label{eq:CT2}
	\gamma(E)\ge c_I \frac{\mu(E)^2}{\int_{I}\|\pi_\theta\mu\|_{L^2}^2\, d\theta},
	\end{equation}
	whenever $\mu$ is a measure on $\mathbb{C}$ with $0<\int_{I}\|\pi_\theta\mu\|_{L^2}^2\, d\theta<\infty$ and $I$ is an interval. Since $\Fav(E)$ satisfies a very similar lower bound (see estimate (1.5) in \cite{chang2017analytic}) this suggests that the optimal dependence on $\kappa$ in \eqref{eq:quanVit} might be linear.
%	\begin{remark}
%		One could use \eqref{eq:CT2} to give a different proof of Corollary \ref{cor:Vit2}. In the proof of \thmref{thm:main} we obtain an upper bound for $\int_{J_0}\|\pi_\theta\HH|_{F}\|_{L^2}^2\, d\theta$, where $F\subset E$ is large, see \eqref{eq:CT3}. Namely, one could apply \eqref{eq:CT2} right after \eqref{eq:CT3} to get a lower bound for $\gamma(E)$. 
%	\end{remark}
	
	\subsection{Peres-Solomyak question}
	The following question was posed in \cite[Question 7.5]{peres2002howlikely}. Given a compact set $E\subset\R^2$ and $0<\delta<\diam(E)$ consider
	\begin{equation}\label{eq:ell}
	\ell(E,\delta)=\sup_{\Gamma}\,\HH_\infty(\Gamma(\delta)\cap E),
	\end{equation}
	where the supremum is taken over all rectifiable curves $\Gamma$ of length $\diam(E)$\footnote{In \cite{peres2002howlikely} in the definition of $\ell(E,\delta)$ it is assumed that $\HH(\Gamma)=1$. We changed their definition slightly to get linear scaling, so that $\ell(CE,C\delta)=C\ell(E,\delta)$.}, and $\Gamma(\delta)$ denotes the $\delta$-neigborhood of $\Gamma$. 
	
	Observe that if $E$ is purely unrectifiable (i.e. $\HH(E\cap\Gamma)=0$ for all rectifiable curves $\Gamma$) then $\ell(E,\delta)\to 0$ as $\delta\to 0$. At the same time, by the Besicovitch projection theorem we have $\Fav(E(\delta))\to 0$ for such sets. This inspired the following question.
	\begin{question}[Peres-Solomyak]\label{q:PS}
		Is there a quantitative estimate of $\Fav(E(\delta))$ in terms of $\ell(E,\delta)$?
	\end{question}
	See also Problem 4 in \cite{mattila2004hausdorff} which asked for estimates of $\Fav(E(\delta))$ for purely unrectifiable sets, and which can be seen as a variant of this question.
	
	For general Borel sets with $\HH(E)\sim\diam(E)\sim 1$, say, the answer to Question \ref{q:PS} is negative, as shown by the ``grid example'', see \cite[Proposition 6.1]{chang2022structure}. On the other hand, in the class of Ahlfors regular sets the answer to Question \ref{q:PS} turns out to be positive.
	
%	We give a positive answer to Question \ref{q:PS} for the class of Ahlfors regular sets.
	\begin{cor}\label{cor:PS}
		Suppose that $E\subset \R^2$ is a compact Ahlfors regular set. Then, for any $0<\ve<1$
		\begin{equation}\label{eq:ggrg}
		\Fav(E(\delta))\le \frac{C\diam(E)}{\left(\log\log\log(C_\ve\diam(E) \ell(E,18\delta)^{-1})\right)^{1/(3+\ve)}},
		\end{equation}
		with $C, C_\ve$ depending on the Ahlfors regularity constant of $E$, and $C_\ve$ depending also on $\ve$.
	\end{cor}
	\begin{proof}
		Without loss of generality assume that $\delta\in 2^{-\mathbb{N}}$, and set 
		\begin{equation*}
			\kappa\coloneqq \frac{\Fav(E(\delta))}{\diam(E)}.
		\end{equation*}
		We will find a lower bound on $\ell(E,9\delta)$ in terms of $\kappa$ (the change from $18\delta$ to $9\delta$ comes from the assumption $\delta\in 2^{-\mathbb{N}}$).
		
		Let $\mathfrak{D}(E,\delta)$ denote the dyadic cubes of sidelength $\delta$ intersecting $E(\delta)$. Set
		\begin{equation*}
		F\coloneqq \bigcup_{Q\in\mathfrak{D}(E,\delta)}\partial Q.
		\end{equation*}
		It is easy to see that $\Fav(F)\ge\Fav(E(\delta))=\kappa\diam(E)\sim\kappa\HH(F_\delta)$, see \eqref{eq:Favskeleton}. 
%		Note that $F_\delta$ can be written as a disjoint union $F_\delta=F_1\cup F_2$, where each $F_i$ is a union of segments parallel to the $x_i$-axis. Since $\Fav(F_1)+\Fav(F_2)\ge \Fav(F_\delta)\gtrsim \kappa\HH(F_\delta)$, at least one $F_i$ satisfies $\Fav(F_i)\gtrsim \kappa\HH(F_i)$. We set $F\coloneqq F_i$.		
		It is also not difficult to show that $F$ inherits Ahlfors regularity from the set $E$, see \lemref{lem:ARskeleton}. Thus, \thmref{thm:main} together with the explicit bounds \eqref{eq:mainthmLib} and \eqref{eq:mainthmSize} gives a Lipschitz graph $\Gamma$ with $\lip(\Gamma)\lesssim_\ve\exp(\exp(C\kappa^{-3-\ve}))$ and such that
		\begin{equation*}
		\HH(F\cap\Gamma)\gtrsim_\ve \exp(-\exp(\exp(C\kappa^{-3-\ve})))\HH(F).
		\end{equation*}
		Since $\HH(\Gamma)\lesssim 1+\lip(\Gamma)\sim\lip(\Gamma)$, there exists a curve $\gamma\subset\Gamma$ of length $1$ such that
		\begin{multline*}
		\HH(F\cap\gamma)\gtrsim \frac{\HH(F\cap\Gamma)}{\lip(\Gamma)}
		\gtrsim_\ve \exp\big(-\exp(\exp(C\kappa^{-3-\ve}))-\exp(C\kappa^{-3-\ve}))\big)\HH(F)\\
		\ge \exp(-\exp(\exp(C'\kappa^{-3-\ve})))\HH(F).
		\end{multline*}
		
		Remark that $\Gamma$ is a bilispchitz image of $\R$ with bilispchitz constant $\lesssim \lip(\Gamma)$. Since the Hausdorff content $\HH_\infty$ restricted to $\R$ is equal to $\HH|_{\R}$, we get that
		\begin{equation}\label{eq:maines}
		\HH_\infty(F\cap\gamma)\gtrsim \frac{\HH(F\cap\gamma)}{\lip(\Gamma)}\gtrsim_\ve \exp(-\exp(\exp(C''\kappa^{-3-\ve})))\HH(F).
		\end{equation}
		It remains to show that
		\begin{equation}\label{eq:ggg}
		\ell(E,9\delta)\ge\HH_\infty(\gamma(9\delta)\cap E)\gtrsim \HH_\infty(F\cap\gamma).
		\end{equation}
		In Appendix \ref{app:LCR} we show that if $E$ is Ahlfors regular and $\gamma$ is arbitrary, then $\HH_\infty(\gamma(9\delta)\cap E)\gtrsim \HH_\infty(E(3\delta)\cap \gamma)$. Since $F\subset E(3\delta)$, \eqref{eq:ggg} follows. Together with \eqref{eq:maines} and recalling that $\HH(F)\sim\diam(E)$, we obtain the desired estimate \eqref{eq:ggrg}.
	\end{proof}
	\subsubsection*{Related work}
	Suppose that $K\subset [0,1]^2$ is the classical 4-corners Cantor set, so that $K=C_{1/2}\times C_{1/2}$ where $C_{1/2}=\{\sum_{i=1}^\infty a_i 4^{-i} : a_i\in \{0,3\}\}$. It is well-known (and it can be shown e.g. using the Analyst's Travelling Salesman Theorem \cite{jones1991traveling}) that
	\begin{equation*}
	\ell(K,4^{-n})\sim \frac{1}{n}.
	\end{equation*}
	Thus, Corollary \ref{cor:PS} implies for $n$ large enough
	\begin{equation}\label{eq:Favdecay}
	\Fav(K(4^{-n}))\lesssim \frac{1}{\left(\log\log\log n\right)^{1/(3+\ve)}}.
	\end{equation}
	The rate of decay of $\Fav(E(\delta))$, with $E$ either the 4-corners Cantor set or more general self-similar and random sets, has been an area of intense study over the last two decades, see	\cite{mattila1990orthogonal,peres2002howlikely,tao2009quantitative,laba2010favard,bateman2010estimate,bond2010buffon,nazarov2011power,bond2014buffon,laba2014recent,wilson2017sets,bongers2019geometric,laba2022vanishing,Vardakis2024buffon}. The problem of obtaining such bounds is called by many authors the Favard length problem, or the Buffon's needle problem.
	
	With the exception of \cite{peres2002howlikely} and \cite{tao2009quantitative}, the upper bounds on $\Fav(E(\delta))$ obtained in aforementioned articles are vastly better than \eqref{eq:Favdecay}. For example, the best known estimate for the 4-corners Cantor set is
	\begin{equation}\label{eq:NPV}
	\Fav(K(4^{-n}))\lesssim \frac{1}{n^c}
	\end{equation}
	for some $0<c<1/6$, see \cite{nazarov2011power}. It has been conjectured that \eqref{eq:NPV} should hold for every $0<c<1$ (with implicit constant depending on $c$). The estimate fails for $c=1$ by \cite{bateman2010estimate}.
	
	Thus, Corollary \ref{cor:PS} is of little use when restricting attention to sets with plenty of extra structure offered by self-similarity or randomness. The novelty of our estimate is that it holds without any additional assumptions beyond Ahlfors regularity. 
	
	It remains an interesting open question to find out optimal dependence on $\ell(E,\delta)$ in such estimates (we imagine \eqref{eq:ggrg} is far from optimal). Other than improving \eqref{eq:ggrg}, one can also seek lower bounds for $\Fav(E(\delta))$. By \cite{bateman2010estimate} for the 4-corners Cantor set we have
	\begin{equation}\label{eq:BVlow}
	\Fav(K(\delta))\gtrsim \ell(K,\delta) \log(\ell(K,\delta)^{-1}).
	\end{equation}
	\begin{problem}
		Find examples of Ahlfors regular sets $E$ improving the lower bound \eqref{eq:BVlow}.
	\end{problem}	
	In \cite{wilson2017sets} purely unrectifiable sets with arbitrarily slow decay of $\Fav(E(\delta))$ have been constructed. It is not clear to us whether these examples imply a better lower bound of the type \eqref{eq:BVlow} than the 4-corners Cantor set.
	
	In \cite{tao2009quantitative} an estimate on $\Fav(E(\delta))$ has been shown that in some regards is even more general than Corollary \ref{cor:PS}. Instead of Ahlfors regularity the set is only assumed to be roughly Ahlfors regular along a fixed sequence of scales. However, to apply Tao's result one also needs to estimate \emph{rectifiability constants} of $E$ along that scale sequence. It is not clear to us whether it is possible to obtain an estimate on $\Fav(E(\delta))$ in terms of $\ell(E,\delta)$ using Tao's result. In the case of the 4-corners Cantor set, Tao's result implies an estimate of $\Fav(K(4^{-n}))$ in terms of the inverse tower function $\log_*(n)$, which is substantially weaker than \eqref{eq:Favdecay}.	
	
	\subsection*{Acknowledgements} I am very grateful to Alan Chang, Tuomas Orponen, Xavier Tolsa, and Michele Villa for many inspiring discussions about the Favard length problem over the years. Special thanks are due to Xavier for introducing me to the problem, and to Tuomas for encouraging me to keep working on it.

	\section{Outline of the proof}\label{sec:outline}
	\subsection{Conical energies} Before outlining the proof of \thmref{thm:main}, we describe our key tool: the \emph{conical energies}. They appeared implicitly for the first time in \cite{martikainen2018characterising}, were explicitly introduced in \cite{chang2017analytic}, and further explored in \cite{dabrowski2020cones,dabrowski2020two}. Similar quantities have also been used in \cite{badger2020radon} to characterize measures carried by Lipschitz graphs.
	
	The conical energy of measure $\mu$ at point $x$ with respect to a set of angles $I\subset [0,1]$ is the quantity
	\begin{equation*}
	\E_\mu(x,I)\coloneqq \int_0^\infty\frac{\mu(X(x,I,r))}{r}\, \frac{dr}{r},
	\end{equation*}
	where $X(x,I,r)$ is the cone centered at $x$ associated to angles from $I$ and truncated at scale $r$. We do not assume \textit{a priori} that $I$ is an interval. There are three main reasons why conical energies are an excellent tool for finding big pieces of Lipschitz graphs.
	
	Firstly, if the set of angles $I$ is an interval, and $\mu$ is Ahlfors regular with $E=\supp\mu$, then one can use $\E_\mu(x,I)$ to estimate the number of ``bad scales'': the scales $j\in\Z$ such that the double truncated cone $X(x,0.9I, 2^{-j-1},2^{-j})$ intersects $E$:
		\begin{multline}\label{eq:badscalesbyenergy}
		\#\Bad(x,0.9I) = \#\{j\in\Z : X(x,0.9I, 2^{-j-1},2^{-j})\cap E\neq\varnothing \}\\
		\lesssim \HH(I)^{-1}\E_\mu(x,I).
		\end{multline}
		See \lemref{lem:scales} for a more precise statement. This was observed in \cite[Lemma 1.9]{martikainen2018characterising}, and used extensively in \cite{chang2017analytic, dabrowski2020cones,dabrowski2022quantitative}. 
		
		Estimating $\#\Bad(x,I)$ is relevant to us for the following reason: suppose that $F\subset E$ is such that $\#\Bad(x,I)=0$ for all $x\in F$. Then, $F\cap X(x,I)=\varnothing$ for all $x\in F$, and consequently $F\subset\Gamma$ for some Lipschitz graph $\Gamma$ with $\lip(\Gamma)\sim \HH(I)^{-1}$. Moreover, in \cite{martikainen2018characterising} it was shown that as soon as $\#\Bad(x,I)\le M_0$ for all $x\in F\subset E$, then $F$ contains a big piece of a Lipschitz graph. See \thmref{thm:MOmainprop} for a precise statement. Thus, in order to find a big piece of a Lipschitz graph in $E$, it suffices to find a large interval $I\subset [0,1]$ and a big piece $F\subset E$ with uniformly bounded $\#\Bad(x,I)$ (or uniformly bounded $\E_\mu(x,I)$, by \eqref{eq:badscalesbyenergy}).
		
		The second reason conical energies are very useful is that they are ``additive in angles'': observe that if $I_1,\dots I_n\subset [0,1]$ are disjoint and $J\coloneqq\bigcup_j I_j$ then
		\begin{equation}\label{eq:energy-addit}
		\E_\mu(x,J) = \sum_{j=1}^n \E_\mu(x,I_j).
		\end{equation}
		Thus, if we are able to show $\E_\mu(x,I_j)\lesssim \HH(I_j)$, and $J = \bigcup_j I_j$ happens to be an interval, then by \eqref{eq:badscalesbyenergy}
		\begin{multline*}
		\#\Bad(x,0.9J)\lesssim\HH(J)^{-1}\E_\mu(x,J) = \HH(J)^{-1} \sum_{j=1}^n \E_\mu(x,I_j)\\
		\lesssim \HH(J)^{-1}\sum_{j=1}^n\HH(I_j) = 1.
		\end{multline*}
		This idea was used in \cite{dabrowski2022quantitative}. It is not clear to us whether one can directly estimate $\#\Bad(x,0.9J)\lesssim 1$ knowing that $\#\Bad(x,I_j)\lesssim 1$ for all $j$, and so bounding conical energies presents an advantage over estimating $\#\Bad(x,I_j)$ directly.
		
		Finally, conical energies are very well suited for stopping-time arguments and multi-scale decompositions in the spirit of David and Semmes, which was exploited in \cite{chang2017analytic, dabrowski2020cones,dabrowski2022quantitative}.
		
		\subsection{Sketch of the proof of \thmref{thm:main}}\label{subsec:outline}
		For the sake of simplicity, we make some claims below which are not literally true (e.g. we omit error terms, assumptions, we claim estimates that in reality look different). We only concentrate on communicating the key ideas and difficulties.
		
		Suppose that $E\subset\R^2$ is Ahlfors regular and satisfies $\Fav(E)\gtrsim \HH(E)$. Let $\mu=\HH|_E$, and suppose that $\diam(E)=1$. We want to find a big piece of a Lipschitz graph inside $E$. 
		
		As mentioned in the previous subsection, by \thmref{thm:MOmainprop} and \lemref{lem:scales} it suffices to find $F\subset E$ with $\mu(F)\gtrsim \mu(E)$ and a large interval $J_0\subset [0,1]$ such that $\E_\mu(x,J_0)\lesssim 1$ for all $x\in F$. By Chebyshev's inequality, this is equivalent to finding $E'\subset E$ with $\mu(E')\gtrsim \mu(E)$ and such that
		\begin{equation}\label{eq:sketch-goal}
		\int_{E'}\int_0^1 \frac{\mu(X(x,J_0,r))}{r}\, \frac{dr}{r}d\mu(x)\lesssim \mu(E).
		\end{equation}
		
		Using the assumption $\Fav(E)\gtrsim\HH(E)$ it is not too difficult to prove the following:  there exist an interval $J_0\subset [0,1]$ and a set $E'\subset E$ with $\mu(E')\gtrsim \mu(E)$ such that for every $x\in E'$ there exists a set of good directions $G(x)\subset J_0$ satisfying $\HH(G(x))\gtrsim 1$ and 
		\begin{equation}\label{eq:Gxest}
		\E_\mu(x,G(x)) = \int_0^1 \frac{\mu(X(x,G(x),r))}{r}\, \frac{dr}{r}\lesssim 1.
		\end{equation}
		This is done in \lemref{lem:base-iteration}. There are two big problems preventing us from concluding \eqref{eq:sketch-goal}:
		\begin{enumerate}
			\item[(a)] $G(x)$ need not be an interval. The best we can get is that it is a finite union of intervals, but with no bound on the number of intervals nor their size.
			\item[(b)] $G(x)$ may vary wildly between different points.
		\end{enumerate}
		
		Let us ignore the second issue for a moment, and assume that for some fixed set $G_0\subset J_0$ we have $G(x)= G_0$ for all $x\in E'$. Assume also that $E'=E$. This is essentially the setup we were able to deal with in \cite{dabrowski2022quantitative}. In \cite[Corollary 4.3]{dabrowski2022quantitative} we showed that in this situation \emph{the good directions propagate}: we can find a new good set $G_1$ such that $G_0\subset G_1\subset J_0$, $\HH(G_1)\ge (1+\ve)\HH(G_0)$ with $\ve\sim 1$, and
		\begin{equation}\label{eq:propen}
		\int_{E}\int_0^1 \frac{\mu(X(x,G_1,r))}{r}\, \frac{dr}{r}d\mu(x)\lesssim \int_{E}\int_0^1 \frac{\mu(X(x,G_0,r))}{r}\, \frac{dr}{r}d\mu(x)\overset{\eqref{eq:Gxest}}{\lesssim} \mu(E).
		\end{equation}
		Since the size of the new good set has increased by $\sim\ve$, after at most $k\sim\ve^{-1}$ iterative applications of \cite[Corollary 4.3]{dabrowski2022quantitative} we get that $G_{k}=J_0$ and so
		\begin{equation*}
		\int_{E}\int_0^1 \frac{\mu(X(x,J_0,r))}{r}\, \frac{dr}{r}d\mu(x)\lesssim \ve^{-1}\mu(E),
		\end{equation*} 
		which was our goal \eqref{eq:sketch-goal}.
		
		Since it is relevant to the proof of \thmref{thm:main}, we describe briefly the proof of \cite[Corollary 4.3]{dabrowski2022quantitative}. Firstly, the new good set $G_1$ is defined as follows. Start by covering $G_0$ with a family of maximal dyadic intervals $\cG_1$ such that for $I\in \cG_1$ we have $\HH(I\cap G_0)\ge (1-\ve)\HH(I)$. Note that such intervals may be arbitrarily short. Then, denoting by $I^1$ the dyadic parent of $I$, we have $\HH(I^1\cap G_0)\le (1-\ve)\HH(I^1)$, and it is easy to see that $G_1\coloneqq\bigcup_{I\in\cG_1}I^1$ satisfies $G_0\subset G_1$ and $\HH(G_1)\ge (1+\ve)\HH(G_0)$. 
		
		What remains to show is the energy estimate \eqref{eq:propen}. Using the additivity of conical energies \eqref{eq:energy-addit} and the fact that for every $I\in\cG_1$ we have $I^1\subset 3I$, to prove \eqref{eq:propen} it suffices to show for every $I\in\cG_1$
		\begin{equation}\label{eq:propDab}
		\int_{E}\int_0^1 \frac{\mu(X(x,3I,r))}{r}\, \frac{dr}{r}d\mu(x)\lesssim \int_{E}\int_0^1 \frac{\mu(X(x,I\cap G_0,r))}{r}\, \frac{dr}{r}d\mu(x).
		\end{equation}
		This is achieved in \cite[Proposition 4.1]{dabrowski2022quantitative}, and proving it occupies most of that paper. In the proof of \eqref{eq:propDab} it is crucial to work with a generalized dyadic lattice $\DD(E,I)$ consisting of (generalized) rectangles adapted to the interval $I$. As usual, $\DD(E,I)=\bigcup_k\DD_k(E,I)$ where each $\DD_k(E,I)$ is a partition of $E$. Each rectangle $Q\in \DD_k(E,I)$ has dimensions $\HH(I)2^{-k}\times 2^{-k}$, with the longer edge pointing in the directions corresponding to $I$. Since it may happen that $\HH(I)\ll 1$, one should think of these rectangles as very narrow tubes. For $Q\in\DD_k(E,I)$ we set $\ell(Q)=2^{-k}$. 
		
		The reason we work with lattices $\DD(E,I)$ is that we want to create a multiscale decomposition associated with \eqref{eq:propDab}, of the form
		\begin{equation*}
		\int_{E}\int_0^1 \frac{\mu(X(x,3I,r))}{r}\, \frac{dr}{r}d\mu(x) \sim \sum_{Q\in\DD(E,I)}\int_Q\frac{\mu(X(x,3I,\ell(Q),\ell(Q^1)))}{\ell(Q)}\, d\mu(x).
		\end{equation*}
		In order to upper bound the right hand side above, it is crucial that for any choice of $x_Q\in Q$ the family of double-truncated cones $\{X(x_Q,3I,\ell(Q),\ell(Q^1))\}_{Q\in\DD(E,I)}$ has bounded intersections. This can only be achieved if we work with the rectangles described before, and not with the usual dyadic squares. In other words, the geometry of the problem forces us to use these rectangles. We discuss this further in Subsection \ref{subsec:metrics}.
		
		Going back to \eqref{eq:Gxest}, the discussion above indicates how to deal with problem (a) assuming all $G(x)$ coincide. The main achievement of this article is that we managed to overcome problem (b).
		
		The key idea is the following. As before, our goal is to prove the propagation of good directions. For every $x\in E'$ we construct a new good set $G_*(x)$ in a similar way as $G_1$ before, but using $G(x)$ instead of $G_0$. Thus, the new good set may vary from point to point. We do it in such a way that $G(x)\subset G_*(x)\subset J_0$ for some fixed interval $J_0$, and for many $x\in E'$ we have $\HH(G_*(x))\ge (1+\ve)\HH(G(x))$. Finally, we have a counterpart of \eqref{eq:propen}:
		\begin{equation}\label{eq:sketch-goal2}
		\int_{E'}\int_0^1 \frac{\mu(X(x,G_*(x),r))}{r}\, \frac{dr}{r}d\mu(x)\lesssim \int_{E'}\int_0^1 \frac{\mu(X(x,G(x),r))}{r}\, \frac{dr}{r}d\mu(x).
		\end{equation}
		See \propref{prop:main} for the precise statement. With our main proposition at hand, an iteration procedure quickly yields \thmref{thm:main}, and we describe this in Subsection \ref{subsec:proof-main}. The proof of \propref{prop:main} occupies Sections \ref{sec:pruning}--\ref{sec:KGL}. The definition of new good sets $G_*(x)$ takes place in Section \ref{sec:pruning}, and the remainder of the paper is devoted to the proof of \eqref{eq:sketch-goal2}.
		
		Since $G_*(x)$ may differ from point to point, we are not able to reduce \eqref{eq:sketch-goal2} to estimates over single intervals as in \eqref{eq:propDab}, and we are not able to fix a single lattice $\DD(E,I)$ to work with. Instead, we derive an estimate of the form
		\begin{equation}\label{eq:sketch-tree}
		\int_{E'}\int_0^1 \frac{\mu(X(x,G_*(x),r))}{r}\, \frac{dr}{r}d\mu(x)\lesssim \sum_{Q\in\T} \int_Q\frac{\mu(X(x,J_Q,\ell(Q),\ell(Q^1)))}{\ell(Q)}\, d\mu(x).
		\end{equation}
		Here $\T$ is a tree-like structure consisting of cubes from many different lattices $\DD(E,I)$. For every $Q\in \T$ we denote by $J_Q$ the interval of directions corresponding to $Q$, so that $Q\in \DD(E,J_Q)$. The tree structure is understood in the following way: if $Q,P\in \T$ and $\ell(Q)\le \ell(P)$, then either $Q\subset P$ and $J_Q\subset J_P$, or $Q\cap P=\varnothing$, or $J_Q\cap J_P=\varnothing$. Equivalently, either $Q\times J_Q\subset P\times J_P$, or $Q\times J_Q\cap P\times J_P=\varnothing$. Thus, $\T$ can be seen as a multiscale decomposition of $E\times [0,1]$, so that there is an additional layer of complexity: we simultaneously work with scales, locations, and directions. 
		
		We describe the lattices we work with in Section \ref{sec:lattices}, and we construct the tree $\T$ in Section \ref{sec:tree-construction}. The rough idea is that if $x\in E'$ and $J$ is one of the ``good intervals for $x$ at scale $k$'', then we should add the cube $Q\in\DD_k(E,J)$ containing $x$ to $\T$. Of course, \eqref{eq:sketch-tree} is only useful if we can estimate the right hand side by the conical energies related to $G(x)$, and this complicates the actual construction. We prove various properties of $\T$ in Section \ref{sec:prop-tree}, and we get the estimate \eqref{eq:sketch-tree} (formulated in the language of ``bad cubes'') in Section \ref{sec:con-to-bad}.
		
		One key aspect of $\T$ is that we have a decomposition
		\begin{equation*}
		\T=\bigcup_{R\in\Roots}\T(R)
		\end{equation*}
		such that within each $\T(R)$ all the cubes belong to a common lattice associated with $J_R$, i.e. $\T(R)\subset \DD(E,J_R)$. Moreover, we are able to prove a packing condition on the family $\Roots$:
		\begin{equation}\label{eq:sketch-sh}
		\sum_{R\in\Roots}\HH(J_R)\mu(R)\lesssim \mu(E).
		\end{equation}
		In our language, the cubes inside $\T$ \emph{do not shatter too often}: a cube $Q$ shatters if its $\T$-children correspond to strict sub-intervals of $J_Q$. The family $\Roots$ consists precisely of cubes that resulted from shattering. Constructing $\T$ in a way that prevents shattering from happening too often was another key difficulty that made the construction rather delicate.
		
		Since each sub-tree $\T(R)$ consists of cubes from a single lattice $\DD(E,J_R)$, we are in a situation resembling \cite[Proposition 4.1]{dabrowski2022quantitative}. In Sections \ref{sec:inter}--\ref{sec:KGL} we adapt the proof from \cite{dabrowski2022quantitative} to our setup. Eventually, we show that for every $R\in\Roots$
		\begin{equation*}
		\sum_{Q\in\T(R)} \int_Q\frac{\mu(X(x,J_Q,\ell(Q),\ell(Q^1)))}{\ell(Q)}\, d\mu(x)
		\lesssim \HH(J_R)\mu(R)\cdot \E,
		\end{equation*}
		where $\E\coloneqq \mu(E')^{-1}\int_{E'}\int_0^1 {\mu(X(x,G(x),r))}/r\, \frac{dr}{r}d\mu(x)$. Together with the packing condition \eqref{eq:sketch-sh} and the estimate \eqref{eq:sketch-tree}, this gives the desired bound \eqref{eq:sketch-goal2}. This concludes the proof of \propref{prop:main}, and of \thmref{thm:main}.
		
		One important aspect of our proof which we ignored in the discussion above is the role played by some maximal function estimates on projected measures. In \cite[Proposition 4.1]{dabrowski2022quantitative} it was crucial to have estimates on $\|\pi_\theta\mu\|_{L^\infty}$ for certain directions $\theta$. Here, this is replaced by pointwise estimates on the Hardy-Littlewood maximal function of the projected measure
		\begin{equation*}
		\mu_\theta(x)\coloneqq\cM(\pi_\theta\mu)(\pi_\theta(x)).
		\end{equation*}
		The importance of such estimates becomes clear in Lemmas \ref{lem:bound-rectang}, \ref{lem:cone-trash}, and especially Lemma \ref{lem:ener-interior}.
	\section{Preliminaries}
	\subsection{Notation}
	Throughout the article notation $f\lesssim g$ stands for $f\le Cg$ for some constant $C>0$. If the constant $C$ depends on a parameter $h$, we write $f\lesssim_h g$. The estimate $f\lesssim g\lesssim f$ is denoted by $f\sim g$, and $f\sim_h g$ stands for $f\lesssim_h g\lesssim_h f$. We often use letter $C$ to denote absolute constants, and the precise value of $C$ may change from line to line. We use $C$ for large constants, and $c$ for small constants. The absolute constants whose value we wish to preserve for future reference are usually assigned a subscript. 
	
	Let $\TT=\R/\Z\simeq[0,1)$.	Given $x\in\R^2$ and $\theta\in\TT$ we set
	\begin{align*}
	e_\theta &\coloneqq (\cos(2\pi\theta),\sin(2\pi\theta))\in\mathbb{S}^1,\\
	\ell_{x,\theta}&\coloneqq x+\spn(e_\theta),\\
	\ell_\theta&\coloneqq \ell_{0,\theta}.
	\end{align*}
	We denote by $\pi_\theta:\R^2\to\R$ the orthogonal projection onto $\ell_\theta$, so that $\pi_\theta(x)=x\cdot e_\theta$.	
	
	For $x\in\R^2$ and a measurable set $I\subset \TT$ we define the cone centered at $x$ with directions in $I$ as
	\begin{equation*}
	X(x,I) = \bigcup_{\theta\in I}\ell_{x,\theta}.
	\end{equation*}
	Note that we do not require $I$ to be an interval. We also set $\theta^\perp = \theta+1/4$, $I^\perp = I+1/4$, $\pi_\theta^\perp=\pi_{\theta^\perp}$, etc.
	
	For $0<r<R$ we define truncated cones as
	\begin{gather*}
	X(x,I,r) = X(x,I)\cap B(x,r),\\
	X(x,I,r,R) = X(x,I,R)\setminus B(x,r).
	\end{gather*}

	In case $I = [\theta - a, \theta + a]$, we have an algebraic characterization of $X(x,I)$: $y\in X(x,I)$ if and only if
	\begin{equation}\label{eq:cone algebraic}
	|\pi^\perp_{\theta}(y)-\pi^\perp_{\theta}(x)|\le \sin(2\pi a)|x-y|.
	\end{equation}
	
	We will denote by $\Delta$ the usual family of half-open \emph{triadic} intervals on $[0,1)\simeq \TT$, so that
	\begin{equation*}
	\Delta=\bigcup_{j\in 0}^\infty \Delta_j=\bigcup_{j\in 0}^\infty\{[k3^{-j}, (k+1)3^{-j}) : k\in \{0,\dots,3^j-1\} \}
	\end{equation*}
	Using triadic intervals instead of dyadic intervals bears a minor technical convenience: every interval $I\in\Delta$ has a triadic child with the same center as $I$. If $J\in \Delta$, then $\Delta(J)$ denotes the collection of triadic intervals contained in $J$. For $I\in\Delta\setminus\{[0,1)\}$, the notation $I^1$ will be used for the triadic parent of $I$. We also write $\Ch_\Delta(I)$ to denote the triadic children of $I$.
	
	Given an interval $I\subset\TT$ and $C>0$, we will write $CI$ to denote the interval with the same midpoint as $I$ and length $C\HH(I)$.
	
	The closure of a set $A$ will be denoted by $\overline{A}$, and its interior by $\mathrm{int}(A)$.

	Suppose that $\mu$ is a Radon measure on $\R^2$. We denote by $\pi_\theta\mu$ the measure on $\R$ obtained as a push-forward of $\mu$ under $\pi_\theta$. 
	
	Given a Radon measure $\nu$ on $\R$ we denote by $\cM\nu$ its Hardy-Littlewood maximal function:
	\begin{equation*}
	\cM\nu(t) = \sup_{r>0}\frac{\nu((t-r,t+r))}{2r}.
	\end{equation*}
	
	Let $x\in\R^2$. In the proof of \thmref{thm:main} quantities of the form $\cM(\pi_\theta\mu)(\pi_\theta(x))$ will play a crucial role. To simplify the notation we set
	\begin{align*}
		\mu_\theta(x) &\coloneqq \cM(\pi_\theta\mu)(\pi_\theta(x)),\\
		\mu_\theta^\perp(x)& \coloneqq \cM(\pi_\theta^\perp\mu)(\pi_\theta^\perp(x)).
	\end{align*}

	\subsection{From big projections to bounded projections}\label{subsec:bigprojtobddproj}
	Recall that $\mu_\theta(x) = \cM(\pi_\theta\mu)(\pi_\theta(x))$. In this subsection we show how to use  lower bounds on $\Fav(E)$ to find a large subset $E'\subset E$ such that for $x\in E'$ we have $\mu_\theta(x)\lesssim 1$ for many $\theta\in \TT$.
	\begin{lemma}\label{lem:big-bdd-subs}
		There exists an absolute constant $C\ge 1$ such that the following holds. Let $\theta\in\TT$, and suppose that $E\subset\R^2$ is $\HH$-measurable with $\HH(E)<\infty$. Set $\mu=\HH|_E$. If
		\begin{equation*}
		\HH(\pi_\theta (E))>0
		\end{equation*}
		and
		\begin{equation}\label{eq:Massump}
			M\ge C\frac{\HH(E)}{\HH(\pi_\theta(E))},
		\end{equation}
		then for $E_M\coloneqq \{x\in E :  \mu_\theta(x)\le M\}$ we have
		\begin{equation*}
		\HH(E_M)\ge \frac{\HH(\pi_\theta(E))}{2}.
		\end{equation*}
	\end{lemma}
	\begin{proof}
		By the weak-$(1,1)$ estimate for the Hardy-Littlewood maximal function (see \cite[Theorem 2.5]{tolsa2014analytic}) we have
		\begin{equation*}
		\HH(\{t\in\R : \cM(\pi_\theta\mu)(t)> M \})\le C'\frac{\|\mu\|}{M} = C'\frac{\HH(E)}{M}.
		\end{equation*}
		By \eqref{eq:Massump} we get
		\begin{equation*}
		\HH(\{t\in\pi_\theta(E): \cM(\pi_\theta\mu)(t)> M \})\le \frac{\HH(\pi_\theta (E))}{2}
		\end{equation*}
		as soon as $C\ge 2C'$. Thus,
		\begin{multline*}
		\HH(\pi_\theta(E_M))=\HH(\{t\in\pi_\theta(E) : \cM(\pi_\theta\mu)(t)\le M \})\\
		\ge \HH(\pi_\theta (E)) - \frac{\HH(\pi_\theta (E))}{2}= \frac{\HH(\pi_\theta (E))}{2}.
		\end{multline*}
		Since $\pi_\theta:\R^2\to\R$ is $1$-Lipschitz, we have
		\begin{equation*}
		\HH(E_M)\ge \HH(\pi_\theta(E_M)) \ge \frac{\HH(\pi_\theta (E))}{2}.
		\end{equation*}
	\end{proof}

	\begin{lemma}\label{lem:big-bdd-subs2}
		Suppose that $\kappa>0$, $M= C\kappa^{-1},$ $G\subset \TT$ is measurable with $\HH(G)>0$, $E\subset\R^2$ is $\HH$-measurable with $\HH(E)<\infty$, and for all $\theta\in G$
		\begin{equation*}
			\HH(\pi_\theta (E))>\kappa\HH(E).
		\end{equation*}
		Let $\mu=\HH|_E$, $E_{M,\theta}\coloneqq \{x\in E : \mu_\theta(x)\le M\}$, and for $x\in E$ let $G(x)\coloneqq\{\theta\in G: x\in E_{M,\theta}\}$. Then,
		\begin{equation*}
			E'\coloneqq \left\{x\in E\ :\ \HH(G(x))\ge \frac{\kappa}{4}\HH(G)\right\}
		\end{equation*}
		satisfies
		\begin{equation*}
			\HH(E')\ge \frac{\kappa}{4}\HH(E).
		\end{equation*}
	\end{lemma}
	\begin{proof}
		By \lemref{lem:big-bdd-subs}, for every $\theta\in G$ we have $\HH(E_{M,\theta})\ge \frac{\kappa}{2}\HH(E).$ By Fubini's theorem
		\begin{equation*}
			\int_E \HH(G(x))\, d\HH(x) = \int_G \HH(E_{M,\theta})\, d\theta \ge \frac{\kappa}{2}\HH(G)\HH(E).
		\end{equation*}
		At the same time,
		\begin{multline*}
			\int_E \HH(G(x))\, d\HH(x) = \int_{E'} \HH(G(x))\, d\HH(x)+\int_{E\setminus E'} \HH(G(x))\, d\HH(x)\\
			\le \HH(G)\HH(E')+\frac{\kappa}{4}\HH(G)\HH(E\setminus E')
			\le \HH(G)\HH(E')+\frac{\kappa}{4}\HH(G)\HH(E).
		\end{multline*}
		Comparing the two inequalities we get
		\begin{equation*}
			\HH(G)\HH(E')\ge \frac{\kappa}{4}\HH(G)\HH(E).
		\end{equation*}
	\end{proof}
	\subsection{From bounded projections to conical energies}	

	In \cite[Corollary 3.3]{chang2017analytic} Chang and Tolsa related conical energies\footnote{The conical energies in \cite{chang2017analytic} are defined using a different formula, but it is equivalent to our definition, see \cite[Remark 1.22]{dabrowski2020cones}.} and $L^2$-norms of projections by proving
	\begin{equation}\label{eq:CT}
		\iint_0^{\infty} \frac{\mu(X(x,I^\perp,r))}{r}\, \frac{dr}{r}d\mu(x) \lesssim \int_I \|\pi_\theta\mu\|_{L^2}^2\, d\theta.
	\end{equation}
	Their proof uses Fourier analysis and it can be adapted to give the following more general bound.
	\begin{lemma}\label{lem:Fourier-calc}
		Suppose that $\mu$ and $\nu$ are compactly supported Borel measure on $\R^2$. For any open $I\subset\TT$ we have
		\begin{equation*}
			\iint_0^{\infty} \frac{\mu(X(x,I^\perp,r))}{r}\, \frac{dr}{r}d\nu(x) \lesssim \int_I \int_\R\pi_\theta\mu(t)\cdot\pi_\theta\nu(t)\, dt\, d\theta.
		\end{equation*}
	\end{lemma}
	The proof is virtually identical to that of \cite[Corollary 3.3]{chang2017analytic}, one just needs to replace one ``copy'' of $\mu$ by $\nu$ in all inequalities.
	\begin{cor}\label{cor:CT}
		Suppose that $\mu$ is a compactly supported Borel measure on $\R^2$. Let $I\subset\TT$ be open, and suppose that
		\begin{equation}\label{eq:L2assump}
			\int_{I} \|\pi_\theta\mu\|_{L^2}^2\, d\theta <\infty.
		\end{equation}
		Then, for $\mu$-a.e. $x\in\supp\mu$
		\begin{equation}\label{eq:ChangTolsabd}
			\int_0^{\infty} \frac{\mu(X(x,I^\perp,r))}{r}\, \frac{dr}{r} \lesssim \int_I \pi_\theta\mu(\pi_\theta(x))\, d\theta.
		\end{equation}
	\end{cor}
	\begin{proof}
		Denote the quantity on the left hand side of \eqref{eq:ChangTolsabd} by $f(x)$, and the quantity from the right hand side by $g(x)$. Note that both $f$ and $g$ are $\mu$-integrable by \eqref{eq:L2assump} and \eqref{eq:CT}.
		
		Let $x\in\supp\mu$ and $\ve>0$. Applying \lemref{lem:Fourier-calc} to $\mu$ and $\nu_\ve \coloneqq \mu|_{B(x,\ve)}/\mu(B(x,\ve))$ we get
		\begin{multline*}
			\frac{1}{\mu(B(x,\ve))}\int_{B(x,\ve)}f(y)\,d\mu(y) \lesssim \int_I \int_\R\pi_\theta\mu\cdot\pi_\theta\nu_\ve\, dt\, d\theta\\
			= \int_I \int_{\R^2}\pi_\theta\mu(\pi_\theta(y))\, d\nu_\ve(y)\, d\theta = \int_{\R^2} \int_I\pi_\theta\mu(\pi_\theta(y))\, d\theta\, d\nu_\ve(y)\\
			= \frac{1}{\mu(B(x,\ve))}\int_{B(x,\ve)}g(y)\,d\mu(y).
		\end{multline*}
		By the Lebesgue differentiation theorem applied to $f$ and $g$, letting $\ve\to 0$ implies \eqref{eq:ChangTolsabd} for $\mu$-a.e. $x\in\supp\mu$.
	\end{proof}
	
	In Appendix \ref{app:paral} we show that in the proof of Theorem \ref{thm:main} we may restrict attention to sets $E$ which are finite unions of parallel segments. The following lemma will be the starting point for an iteration procedure described in Section \ref{sec:main proposition}.
	
	Recall that $\mu_\theta(x)$ stands for $\cM(\pi_\theta\mu)(\pi_\theta(x))$.
	\begin{lemma}\label{lem:base-iteration}
			Suppose that $\kappa>0$, $M= C\kappa^{-1},$ $J_0\in\Delta(\TT)$, $G\subset J_0$ is measurable with $\HH(G)>0$, $E\subset\R^2$ is a finite union of parallel segments, $\mu=\HH|_E$, and for all $\theta\in G$
		\begin{equation}\label{eq:big-proj-asump}
			\HH(\pi_\theta (E))>\kappa\mu(E).
		\end{equation}
		Then, there exists a measurable set $E'\subset E$ such that 
		\begin{enumerate}
			\item $\mu(E')\ge \tfrac{\kappa}{4}\mu(E)$ and 
			\item for every $x\in E'$ there exists a finite family $\cG(x)$ of disjoint triadic sub-intervals of $J_0$ with $G(x)=\bigcup_{I\in\cG(x)} I$ satisfying $\HH(G(x))\ge \tfrac{\kappa}{5}\HH(G)$,
			\item for every $x\in E'$ and $I\in\cG(x)$ there exists $\theta_I\in I$ such that $\mu_{\theta_I}(x) \le M$
			\item for every $x\in E'$
			\begin{equation*}
				\int_0^{\infty} \frac{\mu(X(x,G^\perp(x),r))}{r}\, \frac{dr}{r}\lesssim M \HH(G).
			\end{equation*}
		\end{enumerate} 
	\end{lemma}
	\begin{proof}
%		For brevity set $s=\HH(G)>0$. 
		We apply \lemref{lem:big-bdd-subs2} to obtain $E'\subset E$ with $\HH(E')\ge \tfrac{\kappa}{4}\HH(E)$ and such that for every $x\in E'$ we have a set $G'(x)\subset G$ satisfying $\HH(G'(x))\ge \tfrac{\kappa}{4}\HH(G)$ and
		\begin{equation}\label{eq:Linfbd}
			\mu_\theta(x)\le M\quad\text{for all $\theta\in G'(x)$}.
		\end{equation}
				
		Suppose that $E$ is a union of $N$ parallel segments $J_1,\dots, J_N$. Without loss of generality we may assume these segments are pairwise disjoint. Let $\theta_0\in\TT$ be such that all the segments are perpendicular to $\ell_{\theta_0}$. Clearly, $\pi_{\theta_0}(E)=0$, and if $0<\ve< 1$ and $|\theta-\theta_0|\le \ve$ then
		\begin{equation*}
			\HH(\pi_\theta(E))\le \sum_{i=1}^N\HH(\pi_\theta(J_i))\lesssim \ve \sum_{i=1}^N\HH(J_i) = \ve\HH(E).
		\end{equation*}
		In particular, our big projections assumption \eqref{eq:big-proj-asump} implies that 
		\begin{equation}\label{eq:Gdisjoint}
			G\cap (\theta_0-c\kappa, \theta_0+c\kappa)=\varnothing
		\end{equation}
		for some small absolute $c>0$.

		Fix $x\in E'$. By the outer regularity of Lebesgue measure, for every $\ve>0$ there exists an open set $G'_\ve(x)\supset G'(x)$ with 
		\begin{equation}\label{eq:Gvesmallmeas}
			\HH(G'_{\ve}(x)\setminus G'(x))\le \ve
		\end{equation}
		We will choose some small $\ve$ below. Recall that any open subset of $\R$ can be written as a union of a countable family of disjoint triadic intervals, and let $\cG_\ve'(x)$ be such family for $G'_\ve(x)$. Without loss of generality we assume that every $I\in\cG'_\ve(x)$ intersects $G'(x)$. By \eqref{eq:Gdisjoint} we can also assume that
		\begin{equation}\label{eq:Gedisjo}
			G'_\ve(x)\cap (\theta_0-c\kappa, \theta_0+c\kappa)=\varnothing.
		\end{equation}
		Moreover, since $G'(x)\subset G\subset J_0\in\Delta(\TT)$, we can assume that all $I\in\cG_\ve'(x)$ are contained in $J_0$.
		
		Since
		\begin{equation*}
			\HH(G'_\ve(x))\ge\HH(G'(x))\ge\tfrac{\kappa}{4}\HH(G),
		\end{equation*}
		we may choose a finite subset $\cG_\ve(x)\subset\cG_\ve'(x)$ such that $\HH(G_\ve(x))\ge\tfrac{\kappa}{5}\HH(G),$ where $G_\ve(x)=\bigcup_{I\in\cG_\ve(x)}I$.
		
		Observe that for any $x\in E$ and $\theta\in J_0\setminus (\theta_0-c\kappa, \theta_0+c\kappa)$ we have
		\begin{equation}\label{eq:proj-meas-bdd}
			\pi_\theta\mu(\pi_\theta(x))\lesssim N\kappa^{-1}.
		\end{equation}
		This follows from the fact that $E$ is a union of $N$ parallel segments, and the angle they make with the line $\pi_\theta^{-1}(\pi_\theta(x))$ is $\gtrsim\kappa$.
		
		Let $I\in\Delta(J_0)$ be such that $I\in \cG_\ve(x)$ for some $x\in E'$. By \eqref{eq:Gedisjo} we have $I\cap (\theta_0-c\kappa, \theta_0+c\kappa)=\varnothing$. Thus, by \eqref{eq:proj-meas-bdd}
		\begin{equation*}
			\int_{I} \|\pi_\theta\mu\|_{L^2}\, d\theta = \int_I \int_E \pi_\theta\mu(\pi_\theta(x))\, d\mu(x) d\theta\lesssim N\kappa^{-1}\HH(I)\mu(E)<\infty.
		\end{equation*}
		Consequently, we may apply Corollary \ref{cor:CT} to conclude that for $\mu$-a.e. $x\in E$
		\begin{equation}\label{eq:hlsghrg}
			\int_0^{\infty} \frac{\mu(X(x,I^\perp,r))}{r}\, \frac{dr}{r} \lesssim \int_{I} \pi_\theta\mu(\pi_\theta(x))\, d\theta.
		\end{equation}
		Denote by $F_I\subset E$ the set of points where the inequality above fails, so that $\mu(F_I)=0$. Let
		\begin{equation*}
			F=\bigcup\{F_I:I\in\Delta(J_0) \text{ such that } I\in \cG_\ve(x) \text{ for some } x\in E'\}.
		\end{equation*}
		Since the union above is at most countable, we get that $\mu(F)=0$. At the same time, for $x\in E'\setminus F$ the estimate \eqref{eq:hlsghrg} holds for all $I\in \cG_\ve(x)$. Thus,
		\begin{align*}
			\int_0^{\infty} &\frac{\mu(X(x,G_\ve(x)^\perp,r))}{r}\, \frac{dr}{r} = \sum_{I\in\cG_\ve(x)}\int_0^{\infty} \frac{\mu(X(x,I^\perp,r))}{r}\, \frac{dr}{r}\\
			&\lesssim \sum_{I\in\cG_\ve(x)}\int_{I} \pi_\theta\mu(\pi_\theta(x))\, d\theta = \int_{G_\ve(x)} \pi_\theta\mu(\pi_\theta(x))\, d\theta\\
			&=\int_{G_\ve(x)\cap G'(x)} \pi_\theta\mu(\pi_\theta(x))\, d\theta + \int_{G_\ve(x)\setminus G'(x)} \pi_\theta\mu(\pi_\theta(x))\, d\theta\\
			&\overset{\eqref{eq:Linfbd},\eqref{eq:proj-meas-bdd}}{\lesssim}\int_{G_\ve(x)\cap G'(x)} M\, d\theta + \int_{G_\ve(x)\setminus G'(x)} N\kappa^{-1}\, d\theta\overset{\eqref{eq:Gvesmallmeas}}{\le} M\HH(G) + \ve N\kappa^{-1}.
		\end{align*}
		We choose $\ve=\ve(N,\kappa,\HH(G))$ so small that the last term above is dominated by $M\HH(G)$. For $x\in E'\setminus F$ we set $\cG(x)\coloneqq \cG_\ve(x)$. Since $\mu(F)=0$, we may redefine $E'$ by removing $F$, and we don't lose the measure bound $\HH(E')\ge \tfrac{\kappa}{4}\HH(E)$.
		
		The proof is almost finished. It is clear that $E'$ and $\cG(x)$ have the desired properties (1), (2), and (4). The remaining property (3) is easy to establish: for every $I\in\cG(x)$ we want to find $\theta_I\in I$ such that 
		\begin{equation}\label{eq:ddwfe}
			\mu_{\theta_I}(x)\le M
		\end{equation}
		Recall that $\cG(x)\subset\cG_\ve'(x)$, and we assumed that every $I\in\cG_\ve'(x)$ has non-empty intersection with $G'(x)$. We pick any $\theta_I\in I\cap G'(x)$ and use \eqref{eq:Linfbd} to get \eqref{eq:ddwfe}.
	\end{proof}

\subsection{A family of metrics}\label{subsec:metrics}
Given an interval $I=(\theta-a, \theta+a)\subset\TT$ set 
\begin{equation*}
\pi_I\coloneqq \pi_\theta.
\end{equation*}
We define the metric
\begin{align*}
d_{I}(x,y) &\coloneqq \left( \HH(I)^{-2}|\pi_I^\perp(x)-\pi_I^\perp(y)|^2+|\pi_I(x)-\pi_I(y)|^2 \right)^{1/2}\\
 &\sim\max\left(\HH(I)^{-1}|\pi_I^\perp(x)-\pi_I^\perp(y)|, |\pi_I(x)-\pi_I(y)|\right).
\end{align*}

%The reason we define $d_I$ in this odd way is to ensure that if $I\subset J$, then $d_I(x,y)\ge d_J(x,y)$.
%\begin{lemma}
%	For all $\theta,\eta\in I$ and $x,y\in\R^2$ we have
%	\begin{equation*}
%	d_{\theta,I}(x,y)\sim d_{\eta,I}(x,y),
%	\end{equation*}
%	with the implicit constant independent of $\HH(I)$. In particular, $d_{\theta,I}(x,y)\sim d_{I}(x,y)$.
%\end{lemma}
\begin{remark}\label{rem:dI-isom}
	Observe that $(\R^2,d_I)$ is isometric to $(\R^2,d_{euc})$, where $d_{euc}$ denotes the usual Euclidean distance. Indeed, let $R_I:\R^2\to\R^2$ be the rotation such that	$\pi_I(R_I(x_1,x_2))=x_2$ and $\pi_I^\perp(R_I(x_1,x_2))=x_1$, and let $S_I:\R^2\to\R^2$ be the linear map defined by
	\begin{equation*}
	S_I(x_1,x_2) = (\HH(I)x_1,x_2).
	\end{equation*}
	Then, $d_I(S_I\circ R_I(x), S_I\circ R_I(y))= |x-y|$, so that $S_I\circ R_I:(\R^2, d_{euc})\to (\R^2,d_I)$ is an isometry.
\end{remark}

\begin{remark}
	It is easy to check that if $I, J\subset\TT$ are two intervals satisfying $I\subset CJ$ and $J\subset CI$, then the metrics $d_I$ and $d_J$ are comparable, in the sense that for any $x,y\in\R^2$ we have
	\begin{equation*}
	C^{-1}d_J(x,y)\lesssim d_I(x,y)\lesssim Cd_J(x,y).
	\end{equation*}
	In other words, nearby intervals of comparable length give rise to comparable metrics.
\end{remark}
% and
%\begin{equation*}
%B_I(S_I\circ R_I(x),r)=S_I\circ R_I\big(B(x,r)\big).
%\end{equation*}

Throughout the article we will use the notation
\begin{equation*}
B_I(x,r) \coloneqq B_{d_I}(x,r)=\{y\in\R^2 : d_I(x,y)<r\}.
\end{equation*}
Observe that $B_I(x,r)$ is essentially a tube of dimensions $\HH(I)r\times r$, pointing in the direction of $I$. We will also write $\diam_{I}(A)$ to denote diameter of $A\subset\R^2$ with respect to $d_I$, and $\dist_{I}(x,A)$ to denote the distance from $x$ to $A$ with respect to $d_I$.

	The reason we define $d_I$ as above is the relationship between $d_I$-balls and the cones associated to $I$, which is explored in the following three lemmas. We will use them repeatedly throughout the article, often without mentioning it explicitly.
	\begin{lemma}\label{lem:coneinball}
		If $x\in\R^2$, $I\subset\TT$ is an interval, and $\alpha\ge 1$, then
		\begin{equation*}
		X(x,\alpha I,r)\subset B_I(x,C\alpha r).
		\end{equation*}
	\end{lemma}
	\begin{proof}
		Let $y\in X(x,\alpha I,r)$. Since $|x-y|<r$, to show $y\in B_I(x,C\alpha r)$ we only need to prove $\HH(I)^{-1}|\pi_I^\perp(x)-\pi_I^\perp(y)|\lesssim \alpha r$. But this follows from	 \eqref{eq:cone algebraic}:
		\begin{equation*}
		|\pi_I^\perp (y) - \pi_I^\perp (x)|\le\sin(\pi\alpha\HH(I))|x-y|\le \pi\alpha\HH(I)r.
		\end{equation*}
	\end{proof}
\begin{lemma}
	For every $\alpha>1$ there exists $c=c(\alpha)\in (0,1)$ such that if $x\in\R^2$, $I\subset\TT$ is an interval, and $y\in X(x,I,r,2r)$, then
	\begin{equation*}
		B_I(y,cr)\subset X(x,\alpha I,r/2,4r).
	\end{equation*}
\end{lemma}
\begin{remark}
	In this paper we will only use the following corollary: for $x$ and $y$ as above we have
	\begin{equation}\label{eq:ball-in-cone}
		B(y,c\HH(I)r)\subset X(x,\alpha I,r/2,4r).
	\end{equation}
	This follows from $B(y,c\HH(I)r)\subset B_I(y,cr)$.
\end{remark}
\begin{proof}
	Assume that $y\in X(x,I,r,2r)$ and $z\in B_I(y,cr)$ for some small $0<c<1/2$ to be chosen below. Note that
	\begin{equation*}
	\frac{r}{2}\le|x-y|-|y-z|\le |x-z|\le |x-y|+|y-z|\le 4r.
	\end{equation*}
	It remains to prove that $z\in X(x,\alpha I)$. Using the triangle inequality and \eqref{eq:cone algebraic} we get
	\begin{multline*}
	|\pi_I^\perp (z-x)|\le |\pi_I^\perp (z-y)|+|\pi_I^\perp (y-x)|\le \HH(I) d_I(z,y) + \sin(\pi\HH(I))|y-x|\\
	\le c\HH(I)r + \sin(\pi\HH(I))(|x-z|+cr)\\
	\le 2c\HH(I)|x-z|+(1+2c)\sin(\pi\HH(I))|x-z|\\
	\le\sin(\pi\alpha\HH(I))|x-z|
	\end{multline*}
	if we choose $c=c(\alpha)$ small enough. By \eqref{eq:cone algebraic}, this implies $z\in X(x,\alpha I)$ and finishes the proof.
\end{proof}
%	: for some absolute constant $C\ge 1$ and for every $y\in X(x,I, r, 2r)$ we have
%	\begin{equation*}
%		B_I(x,C^{-1}r)\subset X(y,2I, r/2, 4r)\subset B_I(x,Cr).
%	\end{equation*}
%	We will use the following lemma time and time again.
	\begin{lemma}\label{lem:cone-in-cone}
		For every $\alpha>1$ there exists $C=C(\alpha)>4$ such that if $x, y\in \R^2$, $I\subset\TT$ is an interval, and $R>r>Cd_I(x,y)$, then
		\begin{equation*}
		X(x,I,r,R)\subset X(y,\alpha I,r/2,2R).
		\end{equation*}
	\end{lemma}
	\begin{proof}
		Assume that $z\in X(x,I, r, R)$. Let's show first that $\tfrac{1}{2}r\le |z-y|\le 2R$.		
		Since $R>r>Cd_I(x,y)\ge C|x-y|$,
		\begin{equation*}
		|z-y|\le |z-x|+|y-x|\le R+C^{-1}R
		\end{equation*}
		and 
		\begin{equation*}
		|z-y|\ge |z-x|-|y-x|\ge r-C^{-1}r,
		\end{equation*}
		and so the desired inequalities hold as soon as $C\ge 2$.
		
		Now we show that $z\in X(y,\alpha I)$. Observe that
		\begin{multline*}
		|\pi_I^\perp(z)-\pi^\perp_I(y)|\le |\pi_I^\perp(z)-\pi^\perp_I(x)|+|\pi_I^\perp(x)-\pi^\perp_I(y)|\\
		\le \sin(\pi\HH(I))|z-x|+C^{-1} \HH(I)r,
		\end{multline*}
		where in the last estimate we used \eqref{eq:cone algebraic} and the fact that $|\pi_I^\perp(x)-\pi^\perp_I(y)|\le \HH(I) d_I(x, y)\le C^{-1} \HH(I)r$.
		
		Since $|z-x|\le |z-y|+|y-x|\le |z-y|+C^{-1}r$, we get
		\begin{equation*}
		\sin(\pi\HH(I))|z-x|+C^{-1} \HH(I)r \le \sin(\pi\HH(I))|z-y|+5C^{-1} \HH(I)r.
		\end{equation*}
		Recalling that $|z-y|\ge r/2$, we arrive at
		\begin{multline*}
		\sin(\pi\HH(I))|z-y|+5C^{-1} \HH(I)r\le\\
		\sin(\pi\HH(I))|z-y| + 10C^{-1}\HH(I)|z-y|
		\le \sin(\pi\HH(\alpha I))|z-y|
		\end{multline*}
		assuming $C\gg (\alpha-1)^{-1}$ is large enough. Together with our previous estimates this gives
		\begin{equation*}
		|\pi_I^\perp(z)-\pi^\perp_I(y)|\le\sin(\pi\HH(\alpha I))|z-y|,
		\end{equation*}
		which is equivalent to $z\in X(y,\alpha I)$ by \eqref{eq:cone algebraic}.
	\end{proof}

In the following lemma we show how the maximal function estimates obtained in Subsection \ref{subsec:bigprojtobddproj} can be used to estimate the measure of $d_I$-balls. Recall that $\mu_{\theta}^\perp(x) = \cM(\pi_{\theta}^\perp\mu)(\pi_\theta^\perp(x))$.
\begin{lemma}\label{lem:bound-rectang}
	Suppose that $\alpha \ge1$, $x\in\R^2$, $I\subset\TT$ is an interval, and $\mu_{\theta}^\perp(x) \le M$ for some $\theta\in \alpha I$. Then, for any $r>0$
	\begin{equation*}
	\mu(B_I(x,r))\lesssim \alpha M \HH(I)r.
	\end{equation*}
\end{lemma}
\begin{proof}
	Let $\theta_I$ be the midpoint of $I$. Observe that
	\begin{multline*}
	|(\pi_\theta^\perp(x)-\pi_\theta^\perp(y)) - (\pi_I^\perp(x)-\pi_I^\perp(y))| = |(e_\theta^\perp-e_{\theta_I}^\perp)\cdot(x-y) |\\
	\lesssim |\theta-\theta_I||x-y|\lesssim \alpha \HH(I)|x-y|.
	\end{multline*}
	It follows that for $y\in B_I(x,r)$
	\begin{equation*}
	|\pi_\theta^\perp(x)-\pi_\theta^\perp(y)|\le |\pi_I^\perp(x)-\pi_I^\perp(y)|+C\alpha\HH(I)|x-y|\le \HH(I)r + C\alpha \HH(I)r,
	\end{equation*}
	and so
	\begin{equation*}
	B_I(x,r)\subset T\coloneqq \{z\in\R^2 :|\pi_\theta^\perp(x)-\pi_\theta^\perp(z)|\le C'\alpha \HH(I)r \}.
	\end{equation*}
	Setting $t=\pi_{\theta}^\perp(x)$ and using the maximal function bound we get
	\begin{multline*}
	\mu(T)\le (\pi_\theta^\perp\mu)( [t-C'\alpha \HH(I)r, t+C'\alpha \HH(I)r])\\
	\le \cM(\pi_{\theta}^\perp\mu)(t)\cdot 2C'\alpha \HH(I)r = \mu_\theta^\perp(x)\cdot 2C'\alpha \HH(I)r\lesssim M\alpha \HH(I)r.
	\end{multline*}
\end{proof}
\subsection{Conical energy estimates}
In this subsection we prove several estimates involving cones and conical energies.
\begin{lemma}\label{lem:standard-comp}
	Suppose that $\mu$ is a Radon measure on $\R^2$ and $G\subset \TT$ is measurable. Then, for any $x\in\R^2$ and $0<\rho\le 1/2$ we have
	\begin{equation}\label{eq:standard2}
	\int_{0}^{1}\frac{\mu(X(x, G, \rho r, r))}{r}\frac{dr}{r}\sim \int_{0}^1 \frac{\mu(X(x, G, r))}{r}\frac{dr}{r}.
	\end{equation}
	Moreover, for any integers $j\ge l\ge 0$
	\begin{multline}\label{eq:standard}
	\sum_{k= l+1}^{j-1} \frac{\mu(X(x, G, \rho^{k+1}, \rho^k))}{\rho^k}\lesssim_\rho \int_{\rho^j}^{\rho^l} \frac{\mu(X(x, G, \rho r, r))}{r}\frac{dr}{r}\\
	 \lesssim_\rho \sum_{k= l}^j \frac{\mu(X(x, G, \rho^{k+1}, \rho^k))}{\rho^k}.
	\end{multline}
\end{lemma}
\begin{proof}
	%		Let's show first the ``$\lesssim$'' inequality. Assume that $\sum_{k\ge 0} \frac{\mu(X(x, G, 2^{-k}, 2^{-k+1}))}{2^{-k}}<\infty$, otherwise there is nothing to prove. 
	In \eqref{eq:standard2} the ``$\le$'' bound is immediate because $X(x, G, \rho r, r)\subset X(x, G, r)$. To see the converse, we estimate
	\begin{multline*}
	\int_{0}^1 \frac{\mu(X(x, G, r))}{r}\frac{dr}{r} = \sum_{l\ge 0}\int_{\rho^{l+1}}^{\rho^l}\sum_{k\ge 0}\frac{\mu(X(x, G, \rho^{k+1}r, \rho^{k}r))}{r}\frac{dr}{r}\\
	= \sum_{l\ge 0}\sum_{k\ge 0}\rho^{k}\int_{\rho^{l+1}}^{\rho^l}\frac{\mu(X(x, G, \rho^{k+1}r, \rho^{k}r))}{\rho^{k}r}\frac{dr}{r}\\
	=\sum_{l\ge 0}\sum_{k\ge 0}\rho^{k}\int_{\rho^{k+l+1}}^{\rho^{k+l}}\frac{\mu(X(x, G, \rho s, s))}{s}\frac{ds}{s}\\
	= \sum_{k\ge 0}\rho^{k}\int_{0}^{\rho^{k}}\frac{\mu(X(x, G, \rho s, s))}{s}\frac{ds}{s}\le \sum_{k\ge 0}\rho^k \int_{0}^{1}\frac{\mu(X(x, G, \rho s, s))}{s}\frac{ds}{s}\\
	\le  2\int_{0}^{1}\frac{\mu(X(x, G, \rho s, s))}{s}\frac{ds}{s},
	%	\sum_{l\ge 0}\int_{\rho^{l+1}}^{\rho^l}\sum_{k\ge 0}\frac{\mu(X(x, G, 2^{-k-l-2}, \rho^{k+l}))}{\rho^{l+1}}\frac{dr}{r}\\
	%	\sim \sum_{l\ge 0}\sum_{k\ge 0}\frac{\mu(X(x, G, 2^{-k-l-2}, \rho^{k+l}))}{\rho^{l+1}}
	\end{multline*}
	where in the last line we used the fact that $0<\rho\le 1/2$.
	
	Now we move on to \eqref{eq:standard}. Regarding the second inequality,
	\begin{multline*}
	\int_{\rho^j}^{\rho^l}\frac{\mu(X(x, G, \rho r, r))}{r}\frac{dr}{r}
	 = \sum_{k=l}^{j-1} \int_{\rho^{k+1}}^{\rho^k}\frac{\mu(X(x, G, \rho r, r))}{r}\frac{dr}{r}\\
	\le \sum_{k=l}^{j-1} \int_{\rho^{k+1}}^{\rho^k}\frac{\mu(X(x, G, \rho^{k+2}, \rho^k))}{\rho^{k+1}}\frac{dr}{r}
	\lesssim_\rho \sum_{k=l}^{j-1} \frac{\mu(X(x, G, \rho^{k+2}, \rho^k))}{\rho^{k+1}}\\
	 = \sum_{k=l}^{j-1} \frac{\mu(X(x, G, \rho^{k+1}, \rho^k))}{\rho^{k+1}}+\sum_{k=l}^{j-1} \frac{\mu(X(x, G, \rho^{k+2}, \rho^{k+1}))}{\rho^{k+1}}\\
	 \lesssim_\rho \sum_{k=l}^{j} \frac{\mu(X(x, G, \rho^{k+1}, \rho^k))}{\rho^{k}}.
	\end{multline*}
	Conversely, setting $\delta=\rho^{1/2}$
	\begin{multline*}
	\int_{\rho^j}^{\rho^l}\frac{\mu(X(x, G, \rho r, r))}{r}\frac{dr}{r} = \sum_{k=2l}^{2j-1} \int_{\delta^{k+1}}^{\delta^k}\frac{\mu(X(x, G, \delta^2 r, r))}{r}\frac{dr}{r}\\
	\ge \sum_{k=2l}^{2j-1} \int_{\delta^{k+1}}^{\delta^k}\frac{\mu(X(x, G, \delta^{k+2}, \delta^{k+1}))}{\delta^{k}}\frac{dr}{r}\sim_\rho \sum_{k=2l}^{2j-1} \frac{\mu(X(x, G, \delta^{k+2}, \delta^{k+1}))}{\delta^k}\\
	\gtrsim_\rho \sum_{k=l+1}^{j-1} \left(\frac{\mu(X(x, G, \delta^{2k+2}, \delta^{2k+1}))}{\delta^{2k}}+\frac{\mu(X(x, G, \delta^{2k+1}, \delta^{2k}))}{\delta^{2k-1}}\right)\\
	\sim_\rho \sum_{k=l+1}^{j-1}\frac{\mu(X(x, G, \rho^{k+1},\rho^k))}{\rho^{k}}.
	\end{multline*}	
\end{proof}

In the following two lemmas we show how the maximal function estimates for projections can be used to estimate measure of cones. Recall that $\mu_{\theta}^\perp(x) = \cM(\pi_{\theta}^\perp\mu)(\pi_\theta^\perp(x))$.

\begin{lemma}\label{lem:cone-trash}
	Suppose that $\alpha\ge 1$, $\mu$ is a Radon measure on $\R^2$, and $\cI$ is a family of disjoint intervals in $\TT$. Let $x\in\R^2$, $r>0$, and assume that for every $I\in\cI$ there exists $\theta_I\in \alpha I$ and $y_I\in B_I(x,\alpha r)$ with $\mu_{\theta_I}^\perp(y_I) \le M$. Then, setting $H=\bigcup_{I\in\cI}I$ we have
	\begin{equation*}
	\mu(X(x, H, r))\lesssim \alpha^2 M\HH(H)r.
	\end{equation*}
\end{lemma}
\begin{proof}
	For every $I\in\cI$ we have
	\begin{equation*}
	X(x,I,r)\subset B_I(x,Cr)\subset B_I(y_I, (C+\alpha)r),
	\end{equation*}
	and so by \lemref{lem:bound-rectang}
	\begin{equation*}
	\mu(X(x,I,r))\lesssim \alpha^2 M\HH(I)r.
	\end{equation*}
	Summing over $I\in\cI$ finishes the proof.
\end{proof}
\begin{lemma}\label{lem:ener-interior}
	There exists a small absolute constant $c_0>0$ such that the following holds. Let $\alpha\ge 1$, $x\in\R^2$, $\delta\in [0,1/2)$, and suppose that $J\subset\TT$ is an interval. Assume that 
	\begin{enumerate}[label=({\alph*})]
		\item $E$ is Ahlfors regular with constant $A$, and $\mu=\HH|_E$,
		\item $\cI$ is a family of disjoint intervals contained in $J$ such that for every $I\in\cI$ there exists $\theta_I\in \alpha I$ and $y_I\in B_I(x,\alpha\delta)$\footnote{If $\delta=0$ this means $y_I=x$.} with $\mu_{\theta_I}^\perp(y_I) \le M$,
		\item $H\coloneqq \bigcup_{I\in\cI}I$ satisfies
		\begin{equation}\label{eq:smallH}
		\HH(H)\le \ve\HH(J)
		\end{equation}
		with $\ve\le c_0 \alpha^{-2}M^{-1}A^{-1}$.
	\end{enumerate}
	Then, for any $r\ge \delta$
	\begin{equation}\label{eq:desss}
	\mu(X(x,0.9J,r,2r))\lesssim \mu(X(x,J\setminus H,r/2,4r)).
	\end{equation}
%	\begin{equation*}
%	\int_{2\delta}^{1/2} \frac{\mu(X(x,0.9J,r,2r))}{r}\frac{dr}{r}\lesssim \int_\delta^1 \frac{\mu(X(x,J\setminus H,r,2r))}{r}\frac{dr}{r}.
%	\end{equation*}
\end{lemma}
\begin{proof}
Without loss of generality assume that $\delta>0$ (the case $\delta=0$ follows easily).

%Let $j\ge 0$ be such that $\delta \le 2^{-j}\le 2\delta$. By \lemref{lem:standard-comp}
%\begin{equation}\label{eq:good-sum2}
%\int_{2\delta}^{1/2} \frac{\mu(X(x, 0.9J, r,2r))}{r}\frac{dr}{r} \lesssim \sum_{k= 0}^j \frac{\mu(X(x, 0.9J, 2^{-k-1}, 2^{-k}))}{2^{-k}}
%\end{equation}
%and also
%\begin{equation}\label{eq:good-sum}
%\sum_{k=1}^{j-1} \frac{\mu(X(x, J\setminus H, 2^{-k-1}, 2^{-k}))}{2^{-k}}\lesssim \int_{\delta}^1 \frac{\mu(X(x, J\setminus H, r,2r))}{r}\frac{dr}{r}. 
%\end{equation}
%We define 
%\begin{equation*}
%K = \{0\le k\le j: X(x, 0.9J, 2^{-k}, 2^{-k+1})\cap E\neq\varnothing \}.
%\end{equation*}
%Note that for $k\in\mathbb{N}\setminus K$ the summands in \eqref{eq:good-sum2} are zero. At the same time, since $X(x, 0.9J, 2^{-k}, 2^{-k+1}))\subset B_J(x,C2^{-k})$ for some absolute $C\ge 1$, we may use \eqref{eq:good-sum2} and \eqref{eq:ball-est} to estimate
%\begin{equation}\label{eq:des-est2}
%\int_{2\delta}^{1/2} \frac{\mu(X(x, 0.9J, r,2r))}{r}\frac{dr}{r}\lesssim \sum_{k\in K}\frac{\mu(B_{J}(x,C2^{-k}))}{2^{-k}}\lesssim M\HH(I)\#K.
%\end{equation}
%We are going to use \eqref{eq:good-sum} to get an upper bound on $\#K$. 
Observe that
\begin{multline*}
\mu(X(x,0.9J,r,2r)) = \mu(X(x,0.9J\cap H,r,2r)) + \mu(X(x,0.9J\setminus H,r,2r))\\
\le \mu(X(x,H,2r)) + \mu(X(x,J\setminus H,r/2,4r)),
\end{multline*}
and so to get \eqref{eq:desss} it suffices to show $\mu(X(x,H,2r)) \lesssim \mu(X(x,J\setminus H,r/2,4r))$.

If $X(x, 0.9J, r, 2r)\cap E=\varnothing$ the estimate \eqref{eq:desss} is trivial, so suppose there exists $z\in  X(x, 0.9J, r, 2r)\cap E$. Let $B_z\coloneqq B(z,c\HH(J)r)$ for some small absolute constant $0<c<1$.	
%		By \eqref{eq:ener-est} we have
%		\begin{multline}\label{eq:bound-energy}
%		\int_{0}^1 \frac{\mu(X(x, 3I\cap G(x), r))}{r}\frac{dr}{r}\lesssim\E_1\HH(I)\\
%		\le \frac{\HH(I)}{\tau\HH(J_0)}\E_0 = c\tau^{-1}AM\E_0 \HH(I).
%		\end{multline}
If $c$ is small enough, then
\begin{equation*}
B_z\subset X(x,J,r/2, 4r).
\end{equation*}
%Consider $\cB_0=\{B_z\}_{z\in X(x, 0.9J, r, 2r)\cap E}$. We use the $5r$-covering lemma to choose a family of disjoint balls $\cB_1\subset\cB_0$ such that $\{5B\}_{B\in\cB_1}$ covers $X(x, 0.9J, r, 2r)\cap E$.
By Ahlfors regularity of $E$ we get
\begin{equation}\label{eq:big-cone1}
\mu(X(x,J,r/2, 4r))\ge \mu(B_z)\ge cA^{-1}\HH(J)r.
\end{equation}
At the same time, since $r\ge\delta$ we may use \lemref{lem:cone-trash} to get
\begin{equation}\label{eq:small-comp3}
\mu(X(x,H,4r))\lesssim \alpha^2 M\HH(H)r\overset{\eqref{eq:smallH}}{\le} c_0A^{-1}\HH(J)r.
\end{equation}
%		Indeed, recall that since $I\in\cG_2(x)$, we have $3I\in\cG_1(x)\subset\cG_0(x)$. Let $\cB$ be the collection of maximal triadic intervals contained in $3I\setminus G(x)$. For any $J\in\cB$, its parent $J^1$ contains some $J'\in\cG(x)$, and by the assumption (e) of \propref{prop:main} there exists $\theta_J\in J$ such that $\cM(\pi_{\theta_J}^\perp\mu)(\pi^\perp_{\theta_J}(x))\le M$. By \lemref{lem:cone-trash}	we get \eqref{eq:small-comp}.	
Assuming $c_0$ is small enough, we get from \eqref{eq:big-cone1} and \eqref{eq:small-comp3} that
\begin{equation*}
\mu(X(x,H,4r)) \le \tfrac{1}{2}\mu(X(x,J,r/2, 4r))\le \tfrac{1}{2}\mu(X(x,H,4r)) + \tfrac{1}{2}\mu(X(x,J\setminus H,r/2, 4r)),
\end{equation*}
which implies the desired estimate 
\begin{equation*}
\mu(X(x,H,4r))\le \mu(X(x,J\setminus H,r/2,4r)).
\end{equation*}
%Thus, for any $k\in K$ we have
%\begin{equation*}
%\frac{\mu(X(x, 3I\cap G(x), 2^{-k-1}, 2^{-k+2}))}{2^{-k}}\gtrsim A^{-1}\,\HH(I).
%\end{equation*}
%Combining this with \eqref{eq:good-sum} yields
%\begin{equation*}
%\E_1 \HH(I)\gtrsim \sum_{k\in K} \frac{\mu(X(x, 3I\cap G(x), 2^{-k}, 2^{-k+1}))}{2^{-k}}\gtrsim \#K A^{-1}\,\HH(I).
%\end{equation*}
%Rearranging, we get $\#K\lesssim A\E_1$. Plugging this in \eqref{eq:des-est2} we get the desired estimate \eqref{eq:des-est}.		
\end{proof}

In the lemma below we show that estimates for conical energies associated to $I$ propagate to points that are close with respect to $d_I$, up to an appropriate truncation.
\begin{lemma}\label{lem:con-ener-stable}
	Assume that $\rho\in (0,1)$, $x, y\in \R^2$, $\mu$ is a Radon measure on $\R^2$, and $I\subset\TT$ is an interval. Then
	\begin{equation}\label{eq:con-ener-stable}
		\int_{Cd_I(x,y)}^{1/2} \frac{\mu(X(y, 1.5I, \rho r, r))}{r}\frac{dr}{r}\lesssim \int_{Cd_I(x,y)/2}^1 \frac{\mu(X(x, 2I, \rho r, r))}{r}\frac{dr}{r},
	\end{equation}
	where $C=C(\rho)>4$.
\end{lemma}
\begin{proof}
	By \lemref{lem:cone-in-cone} we have
	\begin{equation*}
		X(y, 1.5I, \rho r, r) \subset X(x, 2I, \rho r/2, 2r)
	\end{equation*}
	as soon as $\rho r\ge C'd_I(x,y)$, where $C'\ge 4$ is absolute. Thus, taking $C=C'/\rho$
	\begin{multline*}
		\int_{Cd_I(x,y)}^{1/2} \frac{\mu(X(y, 1.5I, \rho r, r))}{r}\frac{dr}{r}\lesssim \int_{Cd_I(x,y)}^{1/2} \frac{\mu(X(x, 2I, \rho r/2, 2r))}{r}\frac{dr}{r}\\
		\lesssim \int_{Cd_I(x,y)/2}^1 \frac{\mu(X(x, 2I, \rho r,r))}{r}\frac{dr}{r}.
	\end{multline*}	
\end{proof}
 \subsection{Bad scales}

Suppose that $\rho\in (0,1)$ and $E\subset \R^2$ are fixed. Given an interval $J\subset\TT$ and $x\in E$ we define the set of \emph{bad scales} as
\begin{equation*}
	\Bad(x,J)=\{k\ge 0 : X(x,J,\rho^{k+1},\rho^k)\cap E\neq\varnothing\}.
\end{equation*}
Given a subset $F\subset E$ we define the restricted set of bad scales as
\begin{equation*}
	\Bad_{F}(x,J)=\{k\ge 0 : X(x,J,\rho^{k+1},\rho^k)\cap F\neq\varnothing\}.
\end{equation*}
Additionally, if $0\le l\le j$ are integers we set
\begin{equation*}
	\Bad(x,J,l,j)=\{l\le k\le j : k\in \Bad(x,J)\}.
\end{equation*}
Observe that a set $E\subset\R^2$ with $\diam(E)\le 1$ is a Lipschitz graph with respect to $J$ if and only if
\begin{equation*}
	\Bad(x,J)=\varnothing\quad\text{for all $x\in E$,}
\end{equation*}
see Lemma 15.13 in \cite{mattila1999geometry}.

The following theorem is a variant of the main proposition from \cite{martikainen2018characterising}. We state only the planar case.
\begin{theorem}[{\cite[Proposition 1.12]{martikainen2018characterising}}]\label{thm:MOmainprop}
	Assume that $\rho=1/2$, $E\subset B(0,1)$ is an Ahlfors regular set with constant $A$, $\diam(E)=1$, and $F\subset E$ is $\HH$-measurable with $\HH(F)\ge\tau$. Suppose that $J\subset\TT$ is an interval with $\HH(J)\le c_J$ for some absolute $0<c_J<1$, and that for every $x\in F$ we have
	\begin{equation*}
		\#\Bad_{F}(x,J)\le M_0.
	\end{equation*}
	Then, there exists a Lipschitz graph $\Gamma$ with $\lip(\Gamma)\sim 2^{M_0}\HH(J)^{-1}$ and
	\begin{equation}\label{eq:MOquant}
		\HH(F\cap \Gamma)\ge (c\HH(J) A^{-2}\tau)^{2^{M_0}},
	\end{equation}
	where $0<c<1$ is absolute.
\end{theorem}
\begin{remark}
	The statement above is slightly different than \cite[Proposition 1.12]{martikainen2018characterising}. Firstly, in \cite{martikainen2018characterising} the authors do not explicitly keep track of parameter dependence, and so instead of \eqref{eq:MOquant} they only state $\HH(F\cap \Gamma)\gtrsim_{\tau,M_0,A,\HH(J)} 1$. We describe in Appendix \ref{app:const} how to get the explicit estimate we claim above.
	
	We also added the assumption $\diam(E)=1$, which after a tiny tweak to the proof in \cite{martikainen2018characterising} yields a much better estimate in \eqref{eq:MOquant} than without it. Finally, in \cite{martikainen2018characterising} one obtains a Lipschitz graph $\Gamma$ with $\lip(\Gamma)\sim \HH(J)^{-1}$, which is better than what we claim above. We could have done the same, but this would come at a price of an extra exponential loss in \eqref{eq:MOquant}, and for our purposes the statement above is better. See also Remark \ref{rem:expMO}.
\end{remark}

In the following lemma we relate the conical energies with the number of bad scales. This plays a similar role as Lemma 1.9 in \cite{martikainen2018characterising}.

Recall that $\mu_{\theta}^\perp(x) = \cM(\pi_{\theta}^\perp\mu)(\pi_\theta^\perp(x))$.
\begin{lemma}\label{lem:scales}
	Fix $\alpha>1$ and $\rho\in (0,1)$. Assume that $E$ is Ahlfors regular with constant $A$, and $\mu=\HH|_E$. Suppose $x\in\R^2$, $J\subset\TT$ is an interval, and $0\le l\le j$ are integers.
	Then,
	\begin{equation}\label{eq:scales1}
		\#\Bad(x,J,l,j)\lesssim_{\alpha,\rho} A\HH(J)^{-1}\int_{\rho^j}^{\rho^l}\frac{\mu(X(x, \alpha J, \rho r, r))}{r}\frac{dr}{r} + 1.
	\end{equation}
	Assume additionally that there exist $y\in \R^2$ with $d_J(x,y)\le \alpha \rho^j$ and $\theta\in \alpha J$ such that $\mu_{\theta}^\perp(y) \le M$. Then,
	\begin{equation}\label{eq:scales2}
		\int_{\rho^j}^{\rho^l} \frac{\mu(X(x, J, \rho r, r))}{r}\frac{dr}{r}\lesssim_{\alpha,\rho} M\HH(J)\cdot \#\Bad(x,J,l,j).
	\end{equation}
\end{lemma}
\begin{proof}
	We start by proving \eqref{eq:scales1}.
%	and
%	\begin{equation}\label{eq:blah}
%		\int_{\rho^k}^1 \frac{\mu(X(x_0, J_Q, \rho r, r))}{r}\frac{dr}{r}\lesssim \sum_{j=0}^{k} \frac{\mu(X(x_0, J_Q, \rho^{j+1}, \rho^j))}{\rho^j}=\sum_{j\in S} \frac{\mu(X(x_0, J_Q, \rho^{j+1}, \rho^j))}{\rho^j}.
%	\end{equation}
	Fix $k\in \Bad(x,J,l,j)$, and let $x_k\in X(x, J, \rho^{k+1}, \rho^k)\cap E$. By \eqref{eq:ball-in-cone} if $0<c=c(\alpha,\rho)<1$ is small enough, then the ball $B_k\coloneqq B(x_k, c\HH(J)\rho^k)$ satisfies
	\begin{equation*}
		B_k\subset X(x, \alpha J, \rho^{k+2}, \rho^{k-1}).
	\end{equation*}
	By Ahlfors regularity of $\mu$ this gives
	\begin{equation*}
		\mu(X(x, \alpha J, \rho^{k+2}, \rho^{k-1}))\ge \mu(B_k)\gtrsim_{\alpha,\rho} A^{-1}\HH(J)\rho^k.
	\end{equation*}
	Let $S\coloneqq \Bad(x,J,l,j)\setminus \{l, l+1,j-1,j\}$. The inequality above gives
	\begin{multline*}
		\#S\cdot A^{-1}\HH(J)\lesssim_{\alpha,\rho} \sum_{k\in S}\frac{\mu(X(x, \alpha J, \rho^{k+2}, \rho^{k-1}))}{\rho^k}\\
		= \sum_{k\in S}\frac{\mu(X(x, \alpha J, \rho^{k+2}, \rho^{k+1}))}{\rho^k}+\frac{\mu(X(x, \alpha J, \rho^{k+1}, \rho^{k}))}{\rho^k}+\frac{\mu(X(x, \alpha J, \rho^{k}, \rho^{k-1}))}{\rho^k}\\
		\lesssim_{\rho} \sum_{k=l+1}^{j-1} \frac{\mu(X(x, \alpha J, \rho^{k+1}, \rho^k))}{\rho^k}.
	\end{multline*}
	Rearranging yields
	\begin{equation*}
		\# S \lesssim_{\alpha,\rho} A\HH(J)^{-1} \sum_{k=l+1}^{j-1} \frac{\mu(X(x, \alpha J, \rho^{k+1}, \rho^k))}{\rho^k}.
	\end{equation*}
	By \eqref{eq:standard} we have
	\begin{equation}
		\sum_{k=l+1}^{j-1} \frac{\mu(X(x, \alpha J, \rho^{k+1}, \rho^k))}{\rho^k}\lesssim_{\rho} \int_{\rho^j}^{\rho^l}\frac{\mu(X(x, \alpha J, \rho r, r))}{r}\frac{dr}{r}
	\end{equation}
	Together with the previous estimate and the fact that $\#\Bad(x,J,l,j)\le \# S + 4$, this gives \eqref{eq:scales1}.
	
	We move on to \eqref{eq:scales2}. By \eqref{eq:standard} we have
	\begin{equation*}
		\int_{\rho^j}^{\rho^l} \frac{\mu(X(x, J, \rho r, r))}{r}\frac{dr}{r}\lesssim \sum_{k=l}^{j} \frac{\mu(X(x, J, \rho^{k+1}, \rho^k))}{\rho^k}= \sum_{k\in\Bad(x,J,l,j)} \frac{\mu(X(x, J, \rho^{k+1}, \rho^k))}{\rho^k}.
	\end{equation*}
	Observe that for any $l\le k\le j$ we have
	\begin{equation*}
		\mu(X(x,J,\rho^k))\le \mu(B_{J}(x,C\rho^k))\le \mu(B_J(y,C(\alpha)\rho^k)\lesssim_\alpha M\HH(J)\rho^k,
	\end{equation*}
	where in the second inequality we used the fact that $d_J(x,y)\le \alpha\rho^j\lesssim_\alpha \rho^k$, and in the last inequality we used \lemref{lem:bound-rectang}. Together with the previous estimate this implies
	\begin{equation*}
		\int_{\rho^j}^{\rho^l} \frac{\mu(X(x, J, \rho r, r))}{r}\frac{dr}{r}\lesssim_\alpha  M\HH(J)\cdot \#\Bad(x,J,l,j).
	\end{equation*}
	This finishes the proof of \eqref{eq:scales2}.
\end{proof}

	\section{Main proposition}\label{sec:main proposition}
	In this section we state our main proposition, which asserts that ``good directions propagate''. Then we use it to prove \thmref{thm:main}.
	
	Recall that $\mu_{\theta}^\perp(x) = \cM(\pi_{\theta}^\perp\mu)(\pi_\theta^\perp(x))$.
	\begin{prop}[Main Proposition]\label{prop:main}
		Let $0<\delta, \tau<1$ and $1\le A,M<\infty$. There exist absolute constants $0<c_\ve<1$ and $C>0$ such that if $\ve=c_{\ve}A^{-1}M^{-1}$ and $C_{\mathsf{Prop}}=C\tau^{-1}\delta^{-1}(AM)^C$, then the following holds. 		
		Assume that:
		\begin{enumerate}[label=({\alph*})]
			\item $E\subset [0,1]^2$ is an Ahlfors regular set with constant $A$, and $\mu = \HH|_{E}$,
			\item $J_0\in\Delta(\TT)$ is a triadic interval with $\HH(J_0)\le c_JA^{-1} M^{-1}$, where $c_J>0$ is a small absolute constant,
			\item $E'\subset E$ satisfies $\mu(E')\ge \delta \mu(E)$,
			\item for every $x\in E'$ there exists a finite family $\cG(x)\subset\Delta(J_0)$ of disjoint triadic sub-intervals of $J_0$, with $G(x)=\bigcup_{I\in\cG(x)} I$ satisfying $\HH(G(x))\ge \tau \HH(J_0)$,
			\item for every $I\in\cG(x)$ there exists $\theta_I\in I$ such that $\mu_{\theta_I}^\perp(x) \le M$.
		\end{enumerate}
		Then, for every $x\in E'$ we may find a finite family $\cG_*(x)\subset\Delta(J_0)$ of disjoint triadic sub-intervals of $J_0$ with the following properties:
		\begin{enumerate}
			\item every $I\in\cG(x)$ is contained in some $I_*\in \cG_*(x)$, and every $I_*\in\cG_*(x)$ contains some $I\in\cG(x)$,
			\item denoting $G_*(x)=\bigcup_{I\in\cG_*(x)} I$, either the set $E_\Fin\coloneqq \{x\in E' : G_*(x)=J_0\}$ is large:
			\begin{equation*}
			\mu(E_\Fin)\ge \frac{\mu(E')}{4} ,
			\end{equation*}or the average length of $G_*(x)$ is significantly larger than the length of $G(x)$:
			\begin{equation}\label{eq:measu}
				\int_{E'}\HH(G_*(x))\, d\mu(x)\ge \int_{E'}\HH(G(x))\, d\mu(x)+c\tau \ve\mu(E')\HH(J_0),
			\end{equation}
			where $0<c<1$ is absolute,
		\item we control the conical energies related to $G_*(x)$:
		\begin{multline}\label{eq:ener}
			\int_{E'}\int_0^{\diam(E)} \frac{\mu(X(x,G_*(x),r))}{r}\, \frac{dr}{r}d\mu(x)\\
			\le C_{\mathsf{Prop}}\bigg(\int_{E'}\int_0^{\diam(E)} \frac{\mu(X(x,G(x),r))}{r}\, \frac{dr}{r}d\mu(x) + \HH(J_0)\mu(E)\bigg).
		\end{multline}
	\end{enumerate}
	\end{prop}
%	\begin{remark}
%		To our surprise, in the proof of \propref{prop:main} upper Ahlfors regularity of $E$ is not used at any point. In other words, the result is also valid if instead of (a) we assume that $\mu$ is a lower Ahlfors regular probability measure on $[0,1]^2$, with $E=\supp\mu$ and $\diam(E)\sim 1$. The reason behind this is that all the upper estimates on $\mu$ we need are provided by the maximal function estimates from assumption (e).
%	\end{remark}
	We prove Proposition \ref{prop:main} in Sections \ref{sec:pruning}--\ref{sec:KGL}. Now we show how to use it to obtain \thmref{thm:main}.
	
	\subsection{Proof of \thmref{thm:main}}\label{subsec:proof-main}
	First, observe that \propref{prop:main} is made for iteration: by assertion (1) the families $\cG_*(x)$ we obtain satisfy assumptions (d) and (e), just like the original families $\cG(x)$. We get the following corollary.
	\begin{cor}\label{cor:main2}
		Let $0<\delta, \tau<1$ and $1\le A,M<\infty$. Assume that $E'\subset E\subset [0,1]^2$, $J_0\in\Delta(\TT)$ and families $\cG(x)\subset\Delta(J_0)$ are as in \propref{prop:main}, so that they satisfy assumptions (a)--(e).
		
		Then, there exists a set $F\subset E'$ with
		\begin{equation}\label{eq:Fbig}
		\mu(F)\gtrsim\mu(E')\ge \delta\mu(E)
		\end{equation}
		and such that
		\begin{multline}\label{eq:enercor}
		\int_{F}\int_0^{\diam(E)} \frac{\mu(X(x,J_0,r))}{r}\, \frac{dr}{r}d\mu(x)\\
		\le (2C_{\mathsf{Prop}})^{CAM\tau^{-1}}\bigg(\int_{E'}\int_0^{\diam(E)} \frac{\mu(X(x,G(x),r))}{r}\, \frac{dr}{r}d\mu(x) + \HH(J_0)\mu(E)\bigg).
		\end{multline}
		Here $C_{\mathsf{Prop}}=C\tau^{-1}\delta^{-1}(AM)^C$ is as in \propref{prop:main}.
	\end{cor}
	\begin{proof}
		We prove \eqref{eq:enercor} by iterated application of \propref{prop:main}.
		
		For every $x\in E'$ set $\cG_0(x)\coloneqq \cG(x)$. We apply \propref{prop:main} with families $\cG_0(x)$, and as a result for every $x\in E'$ we obtain new families of intervals, which we denote by $\cG_1(x)$. By assertion (1) of the proposition, each $I_*\in\cG_1(x)$ contains an interval $I\in\cG_0(x)$.
		
		By the assertion (3) of \propref{prop:main}, we have the following estimates for the conical energies associated to the sets $G_1(x)\coloneqq \bigcup_{I\in\cG_1(x)}I$:
		\begin{multline}\label{eq:new-ener}
		\int_{E'}\int_0^{1} \frac{\mu(X(x,G_1(x),r))}{r}\, \frac{dr}{r}d\mu(x)\\
		\le C_{\mathsf{Prop}}\bigg(\int_{E'}\int_0^{1} \frac{\mu(X(x,G_0(x),r))}{r}\, \frac{dr}{r}d\mu(x) + \HH(J_0)\mu(E)\bigg).
		\end{multline}	
		Assertion (2) of \propref{prop:main} gives rise to two distinct cases. 
		
		If the set $E_{\Fin,1}\coloneqq\{x\in E' : G_1(x)=J_0\}$ satisfies $\mu(E_{\Fin, 1})\ge \mu(E')/4$, then we finish the iteration.
		
		Otherwise, by \eqref{eq:measu} we have
		\begin{align*}
		\int_{E'}\HH(G_1(x))\, d\mu(x)&\ge \int_{E'}\HH(G_0(x))\, d\mu(x)+c\tau \ve\mu(E')\HH(J_0)\\
		&\ge c\ve\tau\HH(J_0)\mu(E')
		\end{align*}
		for some small absolute constant $c>0$. Recall that $\ve=c_{\ve}A^{-1}M^{-1}$ with $c_\ve$ absolute.
		
		If the iteration has not finished, then we apply again \propref{prop:main}, with the same parameters $M,\delta$ and $\tau$, the same set $E'$, and with the family of intervals $\cG_1(x)$. The assumptions (a) -- (d) are clearly satisfied. To see that (e) also holds, recall that each of the intervals $I_*\in\cG_1(x)$ contains some $I\in \cG_0(x)$. At the same time, each $I\in\cG_0(x)$ contained some $\theta_I\in I$ with the desired property. Taking $\theta_{I_*}\coloneqq \theta_I$ establishes (e).
		
		Denote the resulting new families of intervals by $\cG_2(x)$. As before, we get
		\begin{align*}
		\int_{E'}\int_0^{1} &\frac{\mu(X(x,G_2(x),r))}{r}\, \frac{dr}{r}d\mu(x)\\
		&\le C_{\mathsf{Prop}}\bigg(\int_{E'}\int_0^{1} \frac{\mu(X(x,G_1(x),r))}{r}\, \frac{dr}{r}d\mu(x) + \HH(J_0)\mu(E)\bigg)\\
		&\overset{\eqref{eq:new-ener}}{\le} 2C_{\mathsf{Prop}}^2\bigg(\int_{E'}\int_0^{1} \frac{\mu(X(x,G_0(x),r))}{r}\, \frac{dr}{r}d\mu(x) + \HH(J_0)\mu(E)\bigg).
		\end{align*}
		Furthermore, if the iteration has not finished (i.e. $\mu(E_{\Fin,2})<\mu(E')/4$, where $E_{\Fin,2}=\{x\in E' : G_2(x)=J_0\}$), then
		\begin{align*}
		\int_{E'}\HH(G_2(x))\, d\mu(x)&\ge \int_{E'}\HH(G_1(x))\, d\mu(x)+c\tau\ve\mu(E')\HH(J_0)\\
		&\ge 2c\ve\tau\HH(J_0)\mu(E').
		\end{align*}
		
		In general, after applying \propref{prop:main} $k$-many times, we have
		\begin{multline}\label{eq:ener-est-fin}
		\int_{E'}\int_0^{1} \frac{\mu(X(x,G_{k}(x),r))}{r}\, \frac{dr}{r}d\mu(x)\\
		\le  (2C_{\mathsf{Prop}})^k\bigg(\int_{E'}\int_0^{1} \frac{\mu(X(x,G_0(x),r))}{r}\, \frac{dr}{r}d\mu(x) + \HH(J_0)\mu(E)\bigg),
		\end{multline}
		and if the iteration has not finished (i.e. $\mu(E_{\Fin,k})<\mu(E')/4$), then
		\begin{equation}\label{eq:iter-unfin}
		\int_{E'}\HH(G_{k}(x))\, d\mu(x)
		\ge kc\ve\tau\HH(J_0)\mu(E').
		\end{equation}
		At the same time, recall that for every $x\in E'$ and any $k\ge 0$ we have $G_k(x)\subset J_0$. Thus,
		\begin{equation*}
		\int_{E'}\HH(G_{k}(x))\, d\mu(x)\le \HH(J_0)\mu(E').
		\end{equation*}
		Comparing this to \eqref{eq:iter-unfin}, we get that if the iteration has not finished after $k$ steps, then
		\begin{equation*}
		kc\ve\tau\le 1.
		\end{equation*}
		It follows that the iteration finishes after at most $(c\ve\tau)^{-1}$-many steps. Recalling that $\ve\sim A^{-1}M^{-1}$, we get that for some integer $1\le k_0\le CAM\tau^{-1}$ the iteration has stopped after $k_0$ steps.
		
		In particular, the set $F\coloneqq E_{\Fin,k_0}=\{x\in E' : G_{k_0}(x)=J_0\}$ satisfies
		\begin{equation*}
		\mu(E_\Fin)\ge \frac{\mu(E')}{4}\gtrsim\delta \mu(E),
		\end{equation*}
		and the estimate \eqref{eq:enercor} follows from
		\begin{align*}
		\int_{F}\int_0^{\diam(E)} &\frac{\mu(X(x,J_0,r))}{r}\, \frac{dr}{r}d\mu(x)\\
		&\le \int_{E'}\int_0^{\diam(E)} \frac{\mu(X(x,G_{k_0}(x),r))}{r}\, \frac{dr}{r}d\mu(x)\\
		&\overset{\eqref{eq:ener-est-fin}}{\le}  (2C_{\mathsf{Prop}})^{k_0}\bigg(\int_{E'}\int_0^{1} \frac{\mu(X(x,G_0(x),r))}{r}\, \frac{dr}{r}d\mu(x) + \HH(J_0)\mu(E)\bigg)\\
		\end{align*}
		and the fact that $k_0\le CAM\tau^{-1}$.
		
%		\begin{lemma}\label{lem:number of iterations}
%			The iteration finishes after at most $CA\kappa^{-3}$ applications of \propref{prop:main}.
%		\end{lemma}
%		\begin{proof}
%			Recall that for every $x\in E'$ and any $k\ge 1$ we have $G_k(x)\subset J_0$. Thus,
%			\begin{equation*}
%			\int_{E'}\HH(G_{k}(x))\, d\mu(x)\le \HH(J_0)\mu(E').
%			\end{equation*}
%			Comparing this to \eqref{eq:iter-unfin}, we get that if the iteration has not finished after $k$ steps, then
%			\begin{equation*}
%			(k-1)c'A^{-1}\kappa^3\le 1.
%			\end{equation*}
%			In particular, the iteration finishes after at most $CA\kappa^{-3}$ steps.
%		\end{proof}
	\end{proof}
	We continue with the proof of \thmref{thm:main}. In Appendix \ref{app:paral} we reduce matters to the case of finite unions of parallel segments. Hence, to prove \thmref{thm:main} it is enough to establish the following.
	\begin{lemma}\label{lem:mainforseg}
		Suppose that $E\subset [0,1]^2$ is a finite union of parallel segments, that it is Ahlfors regular with constant $A$, and that $\Fav(E)\ge\kappa\HH(E)$.
		
		Then, there exists a Lipschitz graph $\Gamma$ with $\lip(\Gamma)\lesssim_{A,\kappa} 1$ such that
		\begin{equation*}
		\HH(E\cap \Gamma)\gtrsim_{\kappa,A} \HH(E).
		\end{equation*}
	\end{lemma}
	\begin{proof}
	Without loss of generality assume that $\diam(E)=1$. Our goal is to apply Corollary \ref{cor:main2}. We need to find a subset $E'\subset E$, an interval $J_0\in\Delta(\TT)$, and families $\cG(x)$ that satisfy the assumptions of that corollary. For that purpose we will use \lemref{lem:base-iteration}.
	
	Let
	\begin{equation*}
	G\coloneqq \{\theta\in\TT : \HH(\pi_\theta(E))\ge \tfrac{\kappa}{2}\HH(E)\}.
	\end{equation*}
	Since
	\begin{equation*}
	\int_{\TT\setminus G}\HH(\pi_\theta(E))\, d\theta\le \frac{\kappa}{2}\HH(E),
	\end{equation*}
	the assumption $\Fav(E)\ge \kappa\HH(E)$ gives
	\begin{equation*}
	\int_{G}\HH(\pi_\theta(E))\, d\theta\ge \frac{\kappa}{2}\HH(E).
	\end{equation*}
	Taking into account the inequality $\HH(\pi_\theta(E))\le\HH(E)$, we arrive at
	\begin{equation*}
	\HH(G)\ge\frac{\kappa}{2}.
	\end{equation*}
	Consider all triadic intervals $J\in\Delta(\TT)$ of some fixed length $ \HH(J)\sim c_J A^{-1}\kappa$, where $c_J\in (0,1)$ is the constant from \propref{prop:main} (b). Since such intervals partition $\TT$, there exists at least one such triadic interval $J_0\in\Delta(\TT)$ satisfying 
	\begin{equation*}
	\HH(J_0\cap G)\ge \frac{\kappa}{2}\HH(J_0).
	\end{equation*}
	
%	Set $J_0\coloneqq J^\perp$ and $G_0\coloneqq (J_0\cap G)^\perp = J\cap G^\perp$. 
	We apply \lemref{lem:base-iteration} to $E$ and $G_0\coloneqq J_0\cap G$. It follows that there exists a set $E'\subset E$ such that
	\begin{itemize}
		\item $\HH(E')\gtrsim \kappa\HH(E)$,
		\item for every $x\in E'$ we have a finite family of disjoint triadic intervals $\cG_1(x)\subset\Delta(J_0)$ with $\HH(G_1(x))\gtrsim \kappa \HH(G_0)\gtrsim \kappa^2\HH(J_0)$,
		\item for every $x\in E'$ and $I\in\cG_1(x)$ there exists $\theta_I\in I$ such that $\mu_{\theta_I}(x) \le M\coloneqq C\kappa^{-1}$
		\item for every $x\in E'$
		\begin{equation}\label{eq:bdd ener}
		\int_0^{\infty} \frac{\mu(X(x,G_1^\perp(x),r))}{r}\, \frac{dr}{r}\lesssim M \HH(G_0).
		\end{equation}
	\end{itemize}
	Hence, we may use Corollary \ref{cor:main2} with parameters $\delta\sim \kappa$, $\tau\sim\kappa^2$, $M\sim \kappa^{-1}$, and data $E'\subset E$, $J_0$, $\cG(x)=\{I^\perp : I\in\cG_1(x)\}.$
	
	We obtain a subset $F\subset E'$ with $\mu(F)\gtrsim \kappa\mu(E)$ and such that \eqref{eq:enercor} holds. Since $C_{\mathsf{Prop}}\sim\tau^{-1}\delta^{-1}(AM)^C\sim (\kappa^{-1}A)^C$ and $AM\tau^{-1}\sim A\kappa^{-3}$, this gives
	\begin{multline}\label{eq:CT3}
		\int_{F}\int_0^{1} \frac{\mu(X(x,J_0,r))}{r}\, \frac{dr}{r}d\mu(x)\\
		\le (CA\kappa^{-1})^{CA\kappa^{-3}} \bigg(\int_{E'}\int_0^{\diam(E)} \frac{\mu(X(x,G_1^\perp(x),r))}{r}\, \frac{dr}{r}d\mu(x) + \HH(J_0)\mu(E)\bigg)\\
		 \overset{\eqref{eq:bdd ener}}{\lesssim}(CA\kappa^{-1})^{CA\kappa^{-3}}\HH(J_0)\mu(E).
	\end{multline}
	Let $F'\subset F$ consist of points $x\in F$ such that
	\begin{equation*}
		\int_0^{1} \frac{\mu(X(x,J_0,r))}{r}\, \frac{dr}{r} \le 2\frac{1}{\mu(F)}\int_{F}\int_0^{1} \frac{\mu(X(x,J_0,r))}{r}\, \frac{dr}{r}d\mu(x),
	\end{equation*}
	so that by Chebyshev's inequality $\mu(F')\ge \mu(F)/2\gtrsim \kappa\mu(E)$. 

	For every $x\in F'$ we apply the estimate \eqref{eq:scales1} from \lemref{lem:scales} to bound
	\begin{multline*}
		\#\Bad(x,0.5J_0) \lesssim A\HH(J_0)^{-1}\int_0^1\frac{\mu(X(x,J_0,r))}{r}\, \frac{dr}{r} + 1\\
		 \le 2A\HH(J_0)^{-1}\frac{1}{\mu(F)}\int_{F}\int_0^{1} \frac{\mu(X(x,J_0,r))}{r}\, \frac{dr}{r}d\mu(x)+1\\
		 \lesssim A\HH(J_0)^{-1}\frac{1}{\mu(F)}(CA\kappa^{-1})^{CA\kappa^{-3}} \HH(J_0)\mu(E)\lesssim (CA\kappa^{-1})^{CA\kappa^{-3}}.
	\end{multline*}
	Having a uniform bound for $\Bad(x,0.5J_0)$ for all $x\in F'$, we may use \thmref{thm:MOmainprop} with $\tau=\HH(F') \sim \kappa\HH(E)\gtrsim \kappa A^{-1}$ and $M_0\sim (CA\kappa^{-1})^{CA\kappa^{-3}}$. 
	
	As a result, we obtain a Lipschitz graph $\Gamma$ with 
	\begin{equation*}\label{eq:mainthmLib2}
	\lip(\Gamma)\lesssim 2^{M_0}\HH(J_0)^{-1}\lesssim \exp((CA\kappa^{-1})^{CA\kappa^{-3}})\lesssim \exp(\exp(CA\kappa^{-3}\log(A\kappa^{-1})))
	\end{equation*}
	and such that
	\begin{equation*}\label{eq:mainthmSize2}
		\HH(E\cap\Gamma)\ge \HH(F'\cap \Gamma)\gtrsim (c\kappa^2 A^{-4})^{2^{M_0}}\gtrsim \exp(-\exp(\exp(CA\kappa^{-3}\log(A\kappa^{-1})))).
	\end{equation*}
	This finishes the proof of \lemref{lem:mainforseg}.
\end{proof}
Lemma \ref{lem:mainforseg} together with \lemref{lem:parallel2} give \thmref{thm:main}, with estimates
\begin{equation}\label{eq:mainthmLib}
\lip(\Gamma)\lesssim \exp(\exp(CA^5\kappa^{-3}\log(A\kappa^{-1})))
\end{equation}
and
\begin{equation}\label{eq:mainthmSize}
\HH(E\cap\Gamma)\gtrsim \exp(-\exp(\exp(CA^5\kappa^{-3}\log(A\kappa^{-1}))))\cdot \HH(E).
\end{equation}

The rest of the article is dedicated to the proof of \propref{prop:main}.

	\section{Choosing new good directions $\cG_*(x)$}\label{sec:pruning}
	We begin the proof of \propref{prop:main}. In this section for every $x\in E'$ we choose the family of intervals $\cG_*(x)$, and we verify assertions (1) and (2) of \propref{prop:main}.
	
	Without loss of generality, assume $\diam(E)=1$. 
	
	\subsection{Choosing very good directions $\cG_2(x)$}\label{subsec:verygooddirs}
	Before we define $\cG_*(x)$ for all $x\in E'$, we choose a particularly nice set $E_0\subset E'$ and for every $x\in E_0$ a set of very good directions $\cG_2(x)$. Roughly speaking, they are points and directions where we can bound corresponding conical energies pointwise. Throughout most of the article we will only work with $E_0$ and $\cG_2(x)$.	
	
	\vspace{0.5em}
	
	Set
	\begin{equation*}
		\E_0 \coloneqq \frac{1}{\mu(E')}\int_{E'}\int_0^{1} \frac{\mu(X(x,G(x),r))}{r}\, \frac{dr}{r}d\mu(x) + 1.
	\end{equation*}
	The case $\E_0=\infty$ is not interesting because that makes the energy estimate \eqref{eq:ener} trivial, and then we may simply choose $\cG_*(x)=\{J_0\}$ for all $x\in E'$. From now on we assume $\E_0<\infty$.
	
	Consider 
	\begin{equation*}
		E_0\coloneqq \bigg\{x\in E'\ :\ \int_0^{1} \frac{\mu(X(x,G(x),r))}{r}\, \frac{dr}{r}\le 2 \E_0\bigg\},
	\end{equation*}
	so that by Chebyshev's inequality
	\begin{equation}\label{eq:E0large}
		\mu(E_0)\ge \frac{1}{2}\mu(E').
	\end{equation}	
	Recall that
	\begin{equation*}
		\ve = c_\ve (AM)^{-1}
	\end{equation*}
	for some small absolute constant $c_\ve>0$.
	
	For every $x\in E_0$ let $\cG_0(x)$ be the family of maximal triadic intervals $I\in\Delta(J_0)$ such that
	\begin{equation}\label{eq:def-G0}
		\HH(I\cap G(x))\ge (1-\ve)\HH(I).
	\end{equation}
	By maximality, the intervals in $\cG_0(x)$ are disjoint. Since $G(x)$ is a finite union of triadic intervals, every $I\in\cG_0(x)$ contains some $I'\in\cG(x)$, and every $I'\in\cG(x)$ is contained in a unique $I\in\cG_0(x)$. In particular, the family $\cG_0(x)$ is a covering of $G(x)$. 
	
	Set $G_0(x)=\bigcup_{I\in\cG_0(x)} I$, and define
	\begin{align}
		\cG_1(x) &\coloneqq \bigg\{I\in\cG_0(x)\ :\ \int_0^{1} \frac{\mu(X(x,I\cap G(x),r))}{r}\, \frac{dr}{r}\le 4 \frac{\HH(I)}{\HH(G_0(x))}\E_0\bigg\}, \notag\\
		\cG_2(x) &\coloneqq \{J\in\Delta(J_0) : 3J=I\ \text{for some }I\in\cG_1(x)\}.\label{eq:G2def}
	\end{align}
	We also set $G_1(x)=\bigcup_{I\in\cG_1(x)} I$ and $G_2(x)=\bigcup_{I\in\cG_2(x)} I$. In the lemma below we prove that $G_1(x)$ covers a substantial portion of $G(x)$.
	\begin{lemma}
		For every $x\in E_0$ we have
		\begin{equation}\label{eq:G1large}
		\HH(G_1(x)\cap G(x))\gtrsim \HH(G(x)).
		\end{equation}
	\end{lemma}
	\begin{proof}
		Observe that for $x\in E_0$
		\begin{multline*}
		\sum_{I\in\cG_0(x)\setminus\cG_1(x)}\HH(I)\le \frac{\HH(G_0(x))}{4\E_0} \sum_{I\in\cG_0(x)\setminus\cG_1(x)} \int_0^{1} \frac{\mu(X(x,I\cap G(x),r))}{r}\, \frac{dr}{r}\\
		\le \frac{\HH(G_0(x))}{4\E_0} \int_0^{1} \frac{\mu(X(x,G(x),r))}{r}\, \frac{dr}{r}\le \frac{\HH(G_0(x))}{2},
		\end{multline*}
		where in the last inequality we used the definition of $E_0$.
		Hence,
		\begin{equation*}
		\sum_{I\in\cG_1(x)}\HH(I)\ge \frac{\HH(G_0(x))}{2},
		\end{equation*}
		and in particular
		\begin{multline*}
		\HH(G_1(x)\cap G(x))=\sum_{I\in\cG_1(x)}\HH(I\cap G(x))\\
		= \sum_{I\in\cG_1(x)}\HH(I) - \sum_{I\in\cG_1(x)}\HH(I\setminus G(x))\ge \frac{\HH(G_0(x))}{2} - \sum_{J\in\cG_0(x)}\HH(J\setminus G(x))\\
		\overset{\eqref{eq:def-G0}}{\ge} \frac{\HH(G_0(x))}{2} - \ve\sum_{I\in\cG_0(x)}\HH(J) \ge \left(\frac{1}{2}-\ve\right)\HH(G_0(x)) \ge \frac{\HH(G(x))}{3},
		\end{multline*}
		where in the last line we used that $G(x)\subset G_0(x)$. 
	\end{proof}
	\begin{remark}\label{rem:lowerbdintervals}
		Without loss of generality, we may assume that for every $x\in E_0$ and $I\in\cG_1(x)$ we have $\HH(I)\ge 3^{-N}$ for some huge $N>0$ (over which we have no quantitative control). Indeed, let
		\begin{equation*}
		\cG_{1,N}(x) = \{I\in\cG_1(x): \HH(I)\ge 3^{-N}\},
		\end{equation*}
		$G_{1,N}(x) = \bigcup_{I\in\cG_{1,N}(x)}I$, and
		\begin{equation*}
		E_{0,N}=\{x\in E_0 : \HH(G_{1,N}(x))\ge \HH(G_1(x))/2 \}.
		\end{equation*}
		By continuity of measures, if we take $N$ large enough then $\mu(E_{0,N})\ge \mu(E_0)/2$. Thus, \eqref{eq:E0large} and \eqref{eq:G1large} still hold with $E_{0,N}$ and $G_{1,N}(x)$ replacing $E_0$ and $\cG_1(x)$, up to changing the absolute constants slightly. 
	\end{remark}

	\subsection{Properties of $\cG_2(x)$}
	In the following lemmas we prove some properties of $\cG_2(x)$ that make these directions very nice, and which will be crucial later on.
	
	Recall that $\mu_{\theta}^\perp(x) = \cM(\pi_{\theta}^\perp\mu)(\pi_\theta^\perp(x))$.
	\begin{lemma}\label{lem:find-bdd-dir}
		For every $x\in E_0$ and $J\in\cG_2(x)$ there exists $\theta_J\in 3J$ such that $\mu_{\theta_J}^\perp(x)\le M$. Consequently, for all $r>0$
		\begin{equation}\label{eq:ball-est}
		\mu(B_J(x,r))\lesssim M\HH(J)r.
		\end{equation}
	\end{lemma}
	\begin{proof}
		Let $I\in\cG_1(x)$ be such that $I=3J$, and recall that $\cG_1(x)\subset\cG_0(x)$, so that $I$ is a maximal triadic interval satisfying \eqref{eq:def-G0}. Since $G(x)$ is a finite union of triadic intervals from $\cG(x)$, it follows that $I$ contains some interval $I'\in\cG(x)$ (it may happen that $I'=I$). By assumption (e) of \propref{prop:main}, there exists $\theta_{I'}\in I'$ such that $\mu_{\theta_{I'}}^\perp(x)\le M$. Since $\theta_{I'}\in I'\subset I=3J$, we may take $\theta_J=\theta_{I'}$.
		
		The estimate \eqref{eq:ball-est} now follows immediately from \lemref{lem:bound-rectang}.
	\end{proof}

%	\begin{lemma}\label{lem:small-comp}
%		For every $x\in E_0$, $J\in\cG_2(x)$ and $r>0$ we have
%		\begin{equation}\label{eq:small-comp}
%		\mu(X(x,3J\setminus G(x),r))\lesssim M\ve\HH(J) r.
%		\end{equation}
%	\end{lemma}
%	\begin{proof}
%				Let $\cI$ be the collection of maximal triadic intervals contained in $3J\setminus G(x)$. We are going to apply \lemref{lem:cone-trash}, but first we need to check that $\cI$ satisfies the assumptions of that lemma. 
%				
%				For any $I\in\cI$ its parent $I^1$ satisfies 
%				\begin{equation*}
%				I^1\cap G(x)\neq\varnothing.
%				\end{equation*}
%				Since $G(x)$ is a union of triadic intervals from $\cG(x)$, we have that either $I^1\subset G(x)$ or $I^1$ contains some interval from $\cG(x)$. The first alternative is impossible because $I\cap G(x)=\varnothing$. Thus, there exists $J'\in\cG(x)$ such that $J'\subset I^1$, and by the assumption (e) of \propref{prop:main} there exists $\theta_{J'}\in J'$ such that $\cM(\pi_{\theta_{J'}}^\perp\mu)(\pi^\perp_{\theta_{J'}}(x))\le M$. 
%				Setting $\theta_I\coloneqq \theta_{J'}$, we see that
%				\begin{equation*}
%				\theta_I\in I^1\subset 5I.
%				\end{equation*}
%				Setting also $y_I=x$ we get that $\cI$ satisfies the assumptions of 	\lemref{lem:cone-trash}. Thus,
%				\begin{equation*}
%				\mu(X(x,3J\setminus G(x),r))\lesssim M\HH(3J\setminus G(x)) r.
%				\end{equation*}
%				Finally, since $J\in\cG_2(x)$ we have $3J\in\cG_1(x)\subset\cG_0(x)$, and so $\HH(3J\setminus G(x))\lesssim \ve\HH(J)$ by the definition of $\cG_0(x)$ \eqref{eq:def-G0}.
%	\end{proof}

		Recall that by the definition of $\cG_2(x)$ \eqref{eq:G2def} for every $x\in E_0$ and $I\in\cG_2(x)$ we have
	\begin{equation}\label{eq:ener-est}
	\int_0^{1} \frac{\mu(X(x,3I\cap G(x),r))}{r}\, \frac{dr}{r}\lesssim \frac{\E_0}{\HH(G_0(x))}\HH(I).
	\end{equation}
	We set
	\begin{equation*}
	\E_1\coloneqq \max\left(\frac{\E_0}{\HH(G_0(x))},M\right).
	\end{equation*}
	In the lemma below we show that the estimate \eqref{eq:ener-est} implies a similar estimate with $3I\cap G(x)$ replaced by $2I$.
	
	Recall that $\ve=c_\ve M^{-1}A^{-1}$.
	\begin{lemma}\label{lem:init-inter2}
		If $c_\ve>0$ is chosen small enough, then for all $x\in E_0$ and $I\in \cG_2(x)$ we have
		\begin{equation}\label{eq:des-est}
		\int_{0}^1 \frac{\mu(X(x, 2I, r))}{r}\frac{dr}{r}\lesssim \E_1\HH(I).
		\end{equation}
%		where 
%		\begin{equation*}
%		\E_2\coloneqq MA\E_1 = MA\frac{\E_0}{\HH(G_0(x))}.
%		\end{equation*}
	\end{lemma}
	\begin{proof}
		We are going to use \lemref{lem:ener-interior} with $J=3I$, $\delta=0, \alpha=5$. We define $\cI$ as the the collection of maximal triadic intervals contained in $3I\setminus G(x)$. Let us check that the assumptions of \lemref{lem:ener-interior} are met.
		
		First, we show that for every $J\in \cI$ there exists $\theta_{J}\in 5J$ such that $\mu_{\theta_{J}}^\perp(x)\le M$. 
		By maximality, for any $J\in\cI$ its parent $J^1$ satisfies 
		\begin{equation*}
		J^1\cap G(x)\neq\varnothing.
		\end{equation*}
		Since $G(x)$ is a union of triadic intervals from $\cG(x)$, we have that either $J^1\subset G(x)$ or $J^1$ contains some interval from $\cG(x)$. The first alternative is impossible because $J\cap G(x)=\varnothing$. Thus, there exists $J'\in\cG(x)$ such that $J'\subset J^1$, and by the assumption (e) of \propref{prop:main} there exists $\theta_{J'}\in J'$ such that $\mu_{\theta_{J'}}^\perp(x)\le M$. 
		Setting $\theta_J\coloneqq \theta_{J'}$, we see that
		\begin{equation*}
		\theta_J\in J^1\subset 5J.
		\end{equation*}
		
		The second assumption of \lemref{lem:ener-interior} we should check is the measure bound \eqref{eq:smallH}. But this follows immediately from the definition of $\cG_2(x)$: since $I\in\cG_2(x)$ we have $3I\in\cG_1(x)\subset\cG_0(x)$, and so $\HH(3I\setminus G(x))\lesssim \ve\HH(J)$ by the definition of $\cG_0(x)$ \eqref{eq:def-G0}. This gives \eqref{eq:smallH} assuming $c_\ve\le c_0/25$.
		
		Thus, we may apply \lemref{lem:ener-interior} to conclude that for any $r>0$
		\begin{equation*}
		\mu(X(x,2I,r, 2r))\lesssim \mu(X(x,3I\cap G(x),r/2, 4r)).
		\end{equation*}
		Integrating yields 
		\begin{multline*}
				\int_{0}^1 \frac{\mu(X(x, 2I, r))}{r}\frac{dr}{r}\overset{\eqref{eq:standard2}}{\sim} \int_{0}^{1/2} \frac{\mu(X(x, 2I, r,2r))}{r}\frac{dr}{r}\\
				\lesssim \int_{0}^1 \frac{\mu(X(x,3I\cap G(x),r/2, 4r))}{r}\frac{dr}{r}
			\lesssim  \int_{0}^4 \frac{\mu(X(x,3I\cap G(x),r))}{r}\frac{dr}{r}\\
			\overset{\eqref{eq:ener-est}}{\lesssim}\E_1\HH(I) + \int_{1}^4 \frac{\mu(X(x,3I\cap G(x),r))}{r}\frac{dr}{r}.
		\end{multline*}
		Since
		\begin{equation*}
		\int_{1}^4 \frac{\mu(X(x,3I\cap G(x),r))}{r}\frac{dr}{r}\le \int_{1}^4 \frac{\mu(B_I(x,Cr))}{r}\frac{dr}{r}\lesssim \mu(B_I(x,4C))\overset{\eqref{eq:ball-est}}{\lesssim}M\HH(I),
		\end{equation*}
		and $\E_1\ge M$, we finally get \eqref{eq:des-est}.

	\end{proof}
	
	In the lemma below we use \lemref{lem:con-ener-stable} to transfer the conical energy estimate \eqref{eq:des-est} to nearby points.
	\begin{lemma}\label{lem:init-inter}
		Let $\rho\in (0,1)$. If $x\in E_0$, $I\in \cG_2(x)$, and $y\in B_{I}(x, 10\rho^k)$, then
		\begin{equation}\label{eq:init-inter}
		\int_{\rho^k}^1 \frac{\mu(X(y, 1.5I, \rho r, r))}{r}\frac{dr}{r}\lesssim_\rho \E_1\HH(I).
		\end{equation}
	\end{lemma}
	\begin{proof}
		By \lemref{lem:con-ener-stable} we have
		\begin{equation*}
		\int_{C\rho^k}^{1/2} \frac{\mu(X(y, 1.5I, \rho r, r))}{r}\frac{dr}{r}\lesssim \int_{0}^1 \frac{\mu(X(x, 2I, r))}{r}\frac{dr}{r}\overset{\eqref{eq:des-est}}{\lesssim}\E_1\HH(I).
		\end{equation*}
		for some $C=C(\rho)> 1$. At the same time,
		\begin{multline*}
		\int_{\rho^k}^{C\rho^k} \frac{\mu(X(y, 1.5I, \rho r, r))}{r}\frac{dr}{r}\lesssim \int_{\rho^k}^{C\rho^k} \frac{\mu(B_I(y,C'\rho^k))}{\rho^k}\frac{dr}{r}\lesssim_\rho \frac{\mu(B_I(y,C'\rho^k))}{\rho^k}\\
		\lesssim \frac{\mu(B_I(x,C''\rho^k))}{\rho^k}\overset{\eqref{eq:ball-est}}{\lesssim}M\HH(I)
		\end{multline*}
		and similarly
		\begin{equation*}
		\int_{1/2}^{1} \frac{\mu(X(y, 1.5I, \rho r, r))}{r}\frac{dr}{r}\lesssim \mu(B_I(x,2))\overset{\eqref{eq:ball-est}}{\lesssim}M\HH(I).
		\end{equation*}
		Since $\E_1\ge M$, the three estimates above give \eqref{eq:init-inter}.
	\end{proof}	
	
	\subsection{Definition of $\cG_*(x)$}\label{subsec:Gstar}
	In this subsection we finally define the families $\cG_*(x)$ and establish assertions (1) and (2) of \propref{prop:main}.
	
	First, for every $x\in E_0$ we define $\cG_1^1(x)$ to be the maximal intervals from $\{I^1\ :\ I\in\cG_1(x)\}$. By maximality, the intervals in $\cG_1^1(x)$ are pairwise disjoint. We set $G_1^1(x)= \cG_1^1(x)$.
	
	For the sake of future reference we make the following observation: for every $I\in\cG_1(x)$ we have $I^1\subset  5I$, and since there exists $J\in\cG_2(x)$ with $I=3J$, it follows that $I^1\subset 15J$. Hence, 
	\begin{equation}\label{eq:G11cover}
		G_1^1(x)\subset \wt{G_2}(x)\coloneqq \bigcup_{J\in\cG_2(x)}15J.
	\end{equation}
%	\vspace{0.5em}
	
	There are three cases to consider when defining $\cG_*(x)$. We partition $E_0$ into 
	\begin{align*}
		E_{00}&\coloneqq\{x\in E_0\, :\, \cG_0(x)=\{J_0\}\},\\
		E_{01}&\coloneqq\{x\in E_0\, :\, \cG_0(x)\neq\{J_0\}\}.
	\end{align*}

	\emph{Points $x\in E'\setminus E_0$.} For such $x$ we set $\cG_*(x)\coloneqq\cG(x)$. 
		\vspace{0.3em}
		
	\emph{Points $x\in E_{00}$.} For such $x$ we set $\cG_*(x)=\{J_0\}$. Recalling that $E_\Fin = \{x\in E':G_*(x)=J_0\}$, we get that $E_{00}\subset E_{\Fin}$.	
	Remark that $x\in E_{00}$ if and only if $\cG_1(x) = \{J_0\}$, simply because $\cG_1(x)\subset \cG_0(x)$ and $\cG_1(x)$ is non-empty.
			\vspace{0.3em}
	
	\emph{Points $x\in E_{01}$.}
	Observe that if $x\in E_0$ satisfies $\cG_0(x)\neq\{J_0\}$, then every $I\in\cG_1(x)\subset \cG_0(x)$ is strictly contained in $J_0$, and in particular $\cG_1^1(x)\subset\Delta(J_0)$.  For all such $x$ we define $\cG_*(x)$ to be the family of maximal intervals from
	\begin{equation*}
		\cG(x)\cup \cG_1^1(x).
	\end{equation*}
	Taking into account that every $I\in\cG_1(x)$ contains some $I'\in\cG(x)$, it is easy to see that $\cG_*(x)$ consists of all the intervals from $\cG_1^1(x)$ and those intervals from $\cG(x)$ which are not contained in any interval from $\cG_1^1(x)$. 
			\vspace{0.3em}
	
	We set $G_*(x)=\bigcup_{I\in\cG_*(x)}I$.

	\subsubsection*{Assertion (1)}
	We check now that assertion (1) of \propref{prop:main} holds for the choice of $\cG_*(x)$ we made above.
	\begin{lemma}
		Every $I\in\cG(x)$ is contained in some $I_*\in\cG_*(x)$, and every $I_*\in\cG_*(x)$ contains some $I\in\cG(x)$.
	\end{lemma}
	\begin{proof}	
		If $x\in E'\setminus E_0$ we have $\cG_*(x)=\cG(x)$, so this is trivial. Similarly, for $x\in E_{00}$ we have $\cG_*(x)=\{J_0\}$ so there is nothing to check. Suppose that $x\in E_{01}$.
		
		Since $\cG_*(x)$ is the family of maximal intervals from $\cG(x)\cup \cG_1^1(x)$, it is clear that every $I\in\cG(x)$ is contained in some $I_*\in\cG_*(x)$. Now suppose that $I_*\in\cG_*(x)$, and we wish to find $I\in\cG(x)$ contained in $I_*$. If $I_*\in\cG(x)$ then there is nothing to do, so assume $I_*\in\cG_1^1(x)$. Let $J\subset I_*$ be an interval from $\cG_1(x)$ such that $J^1=I_*$. Since $\cG_1(x)\subset \cG_0(x)$ we get that $J$ is a maximal triadic subinterval of $J_0$ satisfying $\HH(J\cap G(x))\ge (1-\ve)\HH(J)$. But this implies that there is some $I\in\cG(x)$ with $I\subset J\subset J^1=I_*$. This concludes the proof of assertion (1).
	\end{proof}

	\subsubsection*{Assertion (2)}
	We will derive assertion (2) of \propref{prop:main} from the following lemma.
	\begin{lemma}\label{lem:propG*}
		There exists an absolute constant $0<c<1$ such that if $x\in E_0\setminus E_\Fin$, then $\HH(G_*(x))\ge (1+c\ve)\HH(G(x))$.
	\end{lemma}
	\begin{proof}
		By the definition of $\cG_0(x)$ \eqref{eq:def-G0}, for every $I\in\cG_0(x)$ strictly contained in $J_0$ we have
		\begin{equation*}
			\HH(I^1\cap G(x))\le (1-\ve)\HH(I^1).
		\end{equation*}
		In particular, since $x\notin E_\Fin\supset E_{00}$ we get that $x\in E_{01}\setminus E_{\Fin}$. Consequently, all the $I\in\cG_*(x)$ are strictly contained in $J_0$, and for every $I\in\cG_1^1(x)$ we have $\HH(I\cap G(x))\le (1-\ve)\HH(I)$. Moreover, taking into account that $G_1(x)\subset G_1^1(x)$, \eqref{eq:G1large} implies
		$\HH(G(x)\cap G^1_1(x))\gtrsim \HH(G(x))$.
		It follows that
		\begin{align*}
			\HH(G_*(x))&=\sum_{I\in \cG_*(x)\cap\cG(x)}\HH(I) + \sum_{I\in \cG_*(x)\setminus\cG(x)}\HH(I)=\sum_{I\in \cG_*(x)\cap\cG(x)}\HH(I) + \sum_{I\in \cG_1^1(x)}\HH(I)\\
			&\ge \sum_{I\in \cG_*(x)\cap\cG(x)}\HH(I) + \frac{1}{1-\ve}\sum_{I\in \cG_1^1(x)}\HH(I\cap G(x))\\
			&= \HH(G(x)\setminus G^1_1(x)) + \frac{1}{1-\ve}\HH(G(x)\cap G^1_1(x))\\
			&= \HH(G(x))- \HH(G(x)\cap G^1_1(x)) + \frac{1}{1-\ve}\HH(G(x)\cap G^1_1(x))\\
			&\ge \HH(G(x)) + \ve \HH(G(x)\cap G^1_1(x))\overset{\eqref{eq:G1large}}{\ge} (1+c\ve)\HH(G(x)).
		\end{align*}
	\end{proof}

	Now we can easily establish assertion (2) of \propref{prop:main}.	
	If 
	\begin{equation*}
		\mu(E_\Fin)\ge\frac{\mu(E_0)}{2}\overset{\eqref{eq:E0large}}{\ge} \frac{\mu(E')}{4},
	\end{equation*}
	then we get the first alternative of assertion (2).
	
	 Suppose the opposite inequality holds: $\mu(E_\Fin)\le\mu(E_0)/2$. Then,
	\begin{equation}\label{eq:dcefe}
		\mu(E_0\setminus E_\Fin)\ge\frac{\mu(E_0)}{2}\overset{\eqref{eq:E0large}}{\ge} \frac{\mu(E')}{4}.
	\end{equation}
	
%	Recall that for $x\in E'\setminus E_0$ we have $G_*(x)=G(x)$, and for $x\in E_\Fin$ we have $G_*(x)=J_0\supset G(x)$. In particular, for $x\in E'\setminus (E_0\setminus E_\Fin) = (E'\setminus E_0)\cup E_\Fin$
	By the already established assertion (1) we have $\HH(G_*(x))\ge\HH(G(x))$ for every $x\in E'$. Using also \lemref{lem:propG*} and that $\HH(G(x))\ge \tau\HH(J_0)$ by the assumption (d) of \propref{prop:main}, we get
	\begin{multline*}
	\int_{E'}\HH(G_*(x))\, d\mu(x) = \int_{E'\setminus (E_0\setminus E_\Fin)}\HH(G_*(x))\, d\mu(x)+\int_{E_0\setminus E_\Fin}\HH(G_*(x))\, d\mu(x)\\
	\ge \int_{E'\setminus (E_0\setminus E_\Fin)}\HH(G(x))\, d\mu(x)+\left(1+c\ve\right)\int_{E_0\setminus E_\Fin}\HH(G(x))\, d\mu(x)\\
	 = \int_{E'}\HH(G(x))\, d\mu(x) + c\ve\int_{E_0\setminus E_\Fin}\HH(G(x))\, d\mu(x)\\
	 \ge \int_{E'}\HH(G(x))\, d\mu(x) + c\ve\tau\mu(E_0\setminus E_\Fin)\HH(J_0)\\
	 \overset{\eqref{eq:dcefe}}{\ge}\int_{E'}\HH(G(x))\, d\mu(x) + \tfrac{c}{4}\ve\tau\mu(E')\HH(J_0).
	\end{multline*}
	This establishes the second alternative of assertion (2), the estimate \eqref{eq:measu}.
	
	\subsubsection*{Towards assertion (3)}
	To complete the proof of \propref{prop:main} it remains to prove assertion (3), which states that
	\begin{multline}\label{eq:ener2}
		\int_{E'}\int_0^{1} \frac{\mu(X(x,G_*(x),r))}{r}\, \frac{dr}{r}d\mu(x)\\
		\le C_{\mathsf{Prop}}\bigg(\int_{E'}\int_0^{1} \frac{\mu(X(x,G(x),r))}{r}\, \frac{dr}{r}d\mu(x) + \HH(J_0)\mu(E)\bigg),
	\end{multline}
	where $C_{\mathsf{Prop}}=C\tau^{-1}\delta^{-1}(AM)^C$.
	Recalling that for $x\in E'\setminus E_0$ we have $G_*(x)=G(x)$, the estimate above follows from
	\begin{multline}\label{eq:ener2}
		\int_{E_0}\int_0^{1} \frac{\mu(X(x,G_*(x)\setminus G(x),r))}{r}\, \frac{dr}{r}d\mu(x)\\
		\le (C_{\mathsf{Prop}}-1)\bigg(\int_{E'}\int_0^{1} \frac{\mu(X(x,G(x),r))}{r}\, \frac{dr}{r}d\mu(x) + \HH(J_0)\mu(E)\bigg)\\
		= (C_{\mathsf{Prop}}-1)\bigg(\E_0\mu(E') + \HH(J_0)\mu(E)\bigg).
	\end{multline}

	\begin{lemma}
		For all $x\in E_0$ we have 
		\begin{equation*}
			G_*(x)\setminus G(x)\subset \wt{G_2}(x),
		\end{equation*}
		where $\wt{G_2}(x)=\bigcup_{J\in\cG_2(x)}15J$.
	\end{lemma}
	\begin{proof}
		By \eqref{eq:G11cover}, it suffices to show that $G_*(x)\setminus G(x)\subset G_1^1(x)$.
		
		If $x\in E_{00}$, then we already remarked that $J_0\in \cG_1(x)$. In particular, $G_*(x)=J_0\subset J_0^{1}=G_1^1(x)$.
		
		Now suppose that $x\in E_{01}$. Then $G_*(x) = G(x)\cup G_1^1(x)$ and so $G_*(x)\setminus G(x)\subset G^1_1(x)$.
	\end{proof}

	In the light of lemma above and \eqref{eq:ener2}, and using the fact that
	\begin{multline*}
		\E_1\HH(J_0)\mu(E)= \max\left(\E_0\frac{\HH(J_0)}{\HH(G_0(x))}\frac{\mu(E)}{\mu(E')}\mu(E'),M\HH(J_0)\mu(E)\right)\\
		\lesssim \max\left(\tau^{-1}\delta^{-1} \E_0\mu(E'),M\HH(J_0)\mu(E)\right),
	\end{multline*}
	to finish the proof of \propref{prop:main}, it suffices to prove the following proposition.
	\begin{prop}\label{prop:E0energy-est}
		We have
		\begin{equation}\label{eq:E0energy-est}
			\int_{E_0}\int_0^1 \frac{\mu(X(x,\wt{G_2}(x),r))}{r}\, \frac{dr}{r}d\mu(x)
			\lesssim (AM)^{11}\E_1\HH(J_0)\mu(E).
		\end{equation}
	\end{prop}
	The remainder of the article is dedicated to the proof of this estimate.

%	By \eqref{eq:G11cover}, this follows from 
%	\begin{equation}\label{eq:E0energy-est}
%	\int_{E_0}\int_0^1 \frac{\mu(X(x,\bigcup_{I\in \cG_2(x)}15I,r))}{r}\, \frac{dr}{r}d\mu(x)
%	\le C({M,A})\E_0\mu(E).
%	\end{equation}

	\section{Dyadic lattices}\label{sec:lattices}
	\subsection{David-Mattila dyadic lattice}
	In the proof of \propref{prop:E0energy-est} we will extensively use generalized dyadic lattices constructed by David and Mattila in \cite[Theorem 3.2]{david2000removable}. Although \cite{david2000removable} deals with the usual Euclidean space, their result easily implies a similar statement valid for the metric spaces $(\R^2,d_J)$, where $d_J$ are the metrics defined in Subsection \ref{subsec:metrics}. We explain this further in Remark \ref{rem:DMlatt}.
	\begin{lemma}\label{lem:dyadic cubes}
		There exist parameters $C_0>1$ and $0<\rho=\rho(C_0)< (5000C_0)^{-1}$ such that the following holds. Let $\mu$ be a Radon measure on $\R^2$, $E=\supp\mu$. Let $J\subset \TT$ be an interval. There exists a sequence of partitions of $E$ into Borel subsets $\DD(E,J)=\{\DD_k(E,J):k\ge 0\}$ with the following properties:
		\begin{enumerate}[label={\alph*})]
			\item For each integer $k\ge 0$
			\begin{equation*}
			E=\bigcup_{Q\in \mathcal{D}_{k}(E,J)} Q
			\end{equation*}
			and the sum above is disjoint.
			\item If $k<l$, $Q\in\mathcal{D}_{l}(E,J)$, and $R\in\mathcal{D}_k(E,J)$, then either $Q\cap R = \varnothing$ or else $Q\subset R$.
			\item For each $k\ge0$ and each cube $Q\in\mathcal{D}_{k}(E,J)$ there is a ball $B(Q)=B_J(x_Q, r(Q))$ such that 
			\begin{gather}\label{eq:cube-size1}
			x_Q\in Q,\quad \rho^k\le r(Q)\le C_0 \rho^k,\\
			E\cap B(Q)\subset Q\subset E\cap 28 B(Q) = E\cap B_J(x_Q,28 r(Q)),\label{eq:cube-size3}
			\end{gather}
			and the balls $5B(Q), Q\in \mathcal{D}_{k}(E,J),$ are disjoint.
			\item The cubes $Q\in\DD_k(E,J)$ have small boundaries in the following sense. For each $Q\in\DD_k(E,J)$ and each integer $l\ge0$ set
			\begin{gather*}
			N_l^{ext}(Q) = \{x\in E\setminus Q : \dist_J(x,Q)<\rho^{k+l} \},\\
			N_l^{int}(Q) = \{x\in Q : \dist_J(x,E\setminus Q)<\rho^{k+l} \},\\
			N_l(Q) = N_l^{ext}(Q)\cup N_l^{int}(Q).
			\end{gather*}
			Then
			\begin{equation}\label{eq:doubl-DM}
			\mu(N_l(Q))\le \rho^{l/2}\mu(90 B(Q)).
			\end{equation}
		\end{enumerate}
	\end{lemma}
	
%	\begin{lemma}\label{lem:dyadic cubes}
%		Let $0<\rho\le 1/4$. Suppose that $(X,d,\mu)$ is a geometrically doubling metric measure space. Then, for every $k\in\Z$ there exists a collection $\DD_k$ of generalized cubes on $X$ such that the following hold:
%		\begin{enumerate}
%			\item For each $k\in\Z$, $X=\bigcup_{Q\in\DD_k}Q$, and the union is disjoint.
%			\item If $Q_1,Q_2\in\bigcup_k\DD_k$ satisfy $Q_1\cap Q_2\neq\varnothing$, then either $Q_1\subset Q_2$ or $Q_2\subset Q_1$.
%			\item For every $Q\in\DD_k$ there exists $x_Q\in Q$ such that
%			\begin{equation*}
%			B_X(x_Q, 0.1\rho^k)\subset Q\subset B_X(x_Q,2\rho^k).
%			\end{equation*}
%			%			\item A variant of thin boundaries: let $\DD^{db}_k$ consist of cubes $Q\in\DD_k$ such that
%			%			\begin{equation*}
%			%				\mu(B_X(x_Q, 2\rho^{k-l})\le M^l \mu(Q)\quad\text{for $l\ge 0$.}
%			%			\end{equation*}
%			%			Consider
%			%			\begin{align*}
%			%				\cN_l^{int}(Q) &= \{P\in\DD_{k+l}^{db}: P\subset Q, B(x_P, 4\rho^{k+l})\cap X\setminus Q\neq\varnothing\}\\			
%			%				\cN_l^{ext}(Q) &= \{P\in\DD_{k+l}^{db}: P\cap Q=\varnothing, B(x_P, 4\rho^{k+l})\cap Q\neq\varnothing\}\\					
%			%				N_l(Q) &= \bigcup_{P\in \cN_l^{int}(Q)\cup \cN_l^{ext}(Q)}P.
%			%			\end{align*}
%			%			 If $Q\in\DD^{db}_k$ then
%			%			\begin{equation}
%			%				\mu(N_l(Q))\le C2^{-l}\mu(Q),
%			%			\end{equation}		
%		\end{enumerate}
%	\end{lemma}
\begin{remark}\label{rem:DMlatt}
	As mentioned above, the construction from \cite[Theorem 3.2]{david2000removable} only deals with the standard Euclidean metric, which corresponds to the case of $J=\TT$ above. We explain briefly how this implies \lemref{lem:dyadic cubes} for general $J$.
	
	Recall that there exists a linear map $A_J:\R^2\to \R^2$ such that $A_J:(\R^2,d_{euc})\to (\R^2,d_{J})$ is an isometry, see \remref{rem:dI-isom}.
 Given a measure $\mu$ on $(\R^2,d_J)$ with $\supp\mu=E$ we define the push-forward measure $\mu_J = A_J^{-1}\mu$, we apply \cite[Theorem 3.2]{david2000removable} to $\mu_J$ to get a sequence of partitions $\DD_k(E_J)$, where $E_J=A_J^{-1}(E)$, and finally we obtain $\DD(E,J)$ as
 \begin{equation*}
 \DD_k(E,J)\coloneqq \{A_J(Q)\, :\, Q\in\DD_k(E_J) \}.
 \end{equation*} 
\end{remark}
\begin{remark}
	It may seem that using the construction from \cite{david2000removable} is an overkill given that the classical construction of Christ \cite{christ1990tb} is already valid in metric measure spaces endowed with a doubling measure, and it has the thin boundaries property. However, we cannot ensure that $(E, d_J,\mu)$ is doubling, and so we cannot apply the result from \cite{christ1990tb}. Nevertheless, most of the balls appearing in our proof indeed will be doubling, which will allow us to use \eqref{eq:doubl-DM}.
\end{remark}
	We introduce additional notation. For $Q\in\DD_k(E,J)$ we set $\ell(Q)=\rho^k$ and
	\begin{equation*}
	B_Q\coloneqq 50B(Q) = B_J(x_Q,50r(Q)),
	\end{equation*}
	so that $r(B_Q)\sim_{C_0} \ell(Q)$, $Q\subset E\cap B_Q$, and $B_Q\subset B_P$ whenever $Q\subset P$. From now on we fix $C_0$ and $\rho$ as above, and we consider them absolute constants.
	
	\subsection{Shattering cubes}\label{subsec:shattering}
	In Section \ref{sec:tree-construction} we are going to define a tree-like structure $\T$ consisting of subsets of $E$. Typically, similar tree constructions go as follows:
	\begin{enumerate}
		\item Fix a dyadic lattice $\DD$ on $E$ and choose a root cube $R\in\DD$. Add $R$ to $\T$.
		\item Define a set of conditions you want the cubes in $\T$ to satisfy.
		\item Inductively add cubes to $\T$: if $Q\in \T$, then for any dyadic child $Q'\subset Q$ test whether $Q'$ satisfies the desired conditions. If yes, add $Q'$ to $\T$, otherwise stop the construction on $Q'$.
	\end{enumerate}
	Our construction will differ from the above algorithm in the following way: instead of using a single dyadic lattice $\DD$, we need to use a countable family of lattices corresponding to different metrics $\{d_J\}_{J\in\Delta(\TT)}$. We will explain this in detail in Section \ref{sec:tree-construction}. One of the operations we need to perform is dividing a cube $Q\in\T$ with a corresponding metric $d_Q$ into subcubes $Q'\subset Q$ of the same generation, i.e. $\ell(Q')=\ell(Q)$, but corresponding to a different metric $d_{Q'}$. We call this operation ``shattering a cube'', and we describe it in this subsection.
	\vspace{0.5em}

	We begin with some notation. Set
	\begin{equation*}
	\bD\coloneqq \DD(E,\TT),
	\end{equation*}
	so that $\bD$ is the ``usual'' lattice of David-Mattila cubes, corresponding to the Euclidean distance. We also set $\bD_k\coloneqq \DD_k(E,\TT)$. 
	Note that by \eqref{eq:cube-size1} for every $Q\in\bD_k$
	\begin{equation}\label{eq:cube-size2}
		E\cap B(x_Q,\rho^k)\subset Q\subset B(x_Q,\rho^{k-1}).
	\end{equation}
	
	We are going to use $\bD$ as a building block for a family of partitions of $E$ that will be used in Section \ref{sec:tree-construction}.

	\begin{lemma}\label{lem:gen-gen}
		Assume that $J_P\subset\TT$ is an interval and $k\in\mathbb{N}$. Let $m(P)\in\mathbb{N}$ be the unique integer such that $\HH(J_P)\rho^{k+2}< 5\rho^{m(P)}\le \HH(J_P)\rho^{k+1}$. Suppose that $\cP\subset\bD_{m(P)}$ is a non empty family of cubes, and set $P=\bigcup_{Q\in\cP}Q$. 
		
		Let $J\subset J_P$ also be an interval, and let $l\ge 0$ be an integer. We can define a family of ``dyadic descendants'' of $P$ corresponding to metric $d_J$, denoted by $\bD_{k+l}(P,J)$, such that 
		\begin{equation}\label{eq:child-partition}
		P=\bigcup_{Q\in\bD_{k+l}(P,J)}Q,
		\end{equation}
		the sets in $\bD_{k+l}(P,J)$ are pairwise disjoint, and
		for every $Q\in\bD_{k+l}(P,J)$:
		\begin{enumerate}[label=({\roman*})]
			\item there exists $\cQ\subset\bD_m$ such that $Q=\bigcup_{S\in\cQ}S$, where $m\ge m(P)$ is the unique integer satisfying
			\begin{equation}\label{eq:side}
			\HH(J)\rho^{k+l+2}< 5\rho^m\le \HH(J)\rho^{k+l+1}.
			\end{equation}
			\item there exists $S\in\cQ$ such that for $x_Q\coloneqq x_S$ we have
			\begin{equation}\label{eq:structure-cubes}
			B_{J}(x_Q,0.5\rho^{k+l})\cap P\subset Q\subset B_{J}(x_Q,4\rho^{k+l})\cap P.
			\end{equation}
		\end{enumerate}	
	\end{lemma}
	\begin{remark}
		The lemma is formulated in this slightly awkward way to underline the fact that the resulting cubes $Q\in \bD_{k+l}(P,J)$ have similar structure to the cube $P$ we started with, and so the lemma may be applied to each $Q$ again (with $k+l$ replacing $k$ and $J$ replacing $J_P$).
	\end{remark}
	\begin{remark}\label{rem:chsh}
		In our application we will mostly use this construction for $P\in\bD_k(S,J_P)$ for some $S\subset E$, $J_P\in\Delta(J_0)$, and:
		\begin{enumerate}[label=({\roman*})]
			\item for $l=0$, $J\in\Ch_\Delta(J_P)$. To simplify notation, we set
			\begin{equation*}
			\Div(P,J)\coloneqq \bD_{k}(P,J).
			\end{equation*}
			$\Div(P,J)$ should be seen as a ``shattering'' of $P$ into cubes of the same generation as $P$ corresponding to directions from $J$.
%			\todocomm{figures?}
			\item for $l=1,$ $J=J_P$. We set
			\begin{equation*}
			\Ch(P,J_P)\coloneqq \bD_{k+1}(P,J_P).
			\end{equation*}
			If $Q\in \Ch(P,J_P)$, we will say that $P$ is the parent of $Q$, and we will write $Q^1\coloneqq P$.
		\end{enumerate}
	\end{remark}
	\begin{proof}[Proof of \lemref{lem:gen-gen}]
		Let $m\in\mathbb{N}$ be the unique integer satisfying \eqref{eq:side}.
		Since $\HH(J)\le \HH(J_P)$ and $l\ge 0$, we have $m\ge m(P)$. Define
		\begin{equation*}
		\cP' \coloneqq \{S\in\bD_m : S\subset P \},
		\end{equation*}
		so that $\cP'$ is a partition of $P$, and for every $S\in\cP'$ there is a unique $R\in\cP$ such that $S\subset R$. 	
		We are going to partition $\cP'$ into a finite number of subfamilies $\{\cQ_i\}_{i\in I}$ such that 
		\begin{equation*}
		\bD_{k+l}(P,J)\coloneqq \{Q_i\ :\ Q_i\coloneqq\bigcup_{S\in\cQ_i}S,\ i\in I\}
		\end{equation*}
		will have all the desired properties.
		
		First, let $\{x_i\}_{i\in I}$ be a maximal subset of $\{x_S : S\in\cP'\}$ such that for any $i\neq i'\in I$ we have $d_J(x_i, x_{i'})>3\rho^{k+l}$. In other words, it is a maximal $3\rho^{k+l}$-separated subset of $\{x_S : S\in\cP'\}$ with respect to metric $d_J$. 
		
		Now to each $S\in\cP'$ we will associate some $i\in I$, denoted by $i(S)$. We do this as follows: if there exists $i\in I$ such that $d_J(x_i,x_S)\le \rho^{k+l}$, then we say that \emph{$S$ is very close to $x_i$}, and we set $i(S)\coloneqq i$. Note that there is at most one such $i\in I$, because if we also had $d_J(x_{i'},x_S)\le \rho^{k+l}$ for another $i'\in I$, then by the triangle inequality $d_J(x_i,x_{i'})\le 2\rho^{k+l}$, which is a contradiction with the $3\rho^{k+l}$-separation of $\{x_i\}_{i\in I}$.
		
		If $S$ is not very close to any $x_i$, then we pick any $i\in I$ satisfying $d_J(x_i,x_S)\le 3\rho^{k+l}$ and set $i(S)=i$. Note that at least one such $i\in I$ exists by the maximality of $\{x_i\}_{i\in I}$. In this case we say that \emph{$S$ is pretty close to $x_i$}.
		
		Having chosen $i(S)\in I$ for every $S\in \cP'$, we are ready to define the families $\cQ_i$. For each $i\in I$ we set
		\begin{equation*}
		\cQ_i\coloneqq \{S\in\cP' : i(S)=i \}
		\end{equation*}
		and $Q_i = \bigcup_{S\in\cQ_i}S$. We also set $x_{Q_i}\coloneqq x_i$ and $\bD_{k+l}(P,J)\coloneqq \{Q_i : i\in I\}$. Note that
		\begin{equation*}
		P = \bigcup_{S\in\cP'}S = \bigcup_{i\in I}\bigcup_{S\in\cQ_i}S = \bigcup_{i\in I}Q_i,
		\end{equation*}
		so that \eqref{eq:child-partition} holds. Since each $S\in\cP'$ belongs to exactly one family $\cQ_i$, the sets $Q_i$ are pairwise disjoint.
		
		It remains to show \eqref{eq:structure-cubes}. Fix $Q_i$. First, observe that for all $S\in\cP'\subset\mathbb{D}_m$ we have
		\begin{equation*}
		\diam_{J}(S)\le \HH(J)^{-1}\diam(S)\overset{\eqref{eq:cube-size2}}{\le} 2\HH(J)^{-1}\rho^{m-1}\overset{\eqref{eq:side}}{\le} 0.4 \rho^{k+l}.
		\end{equation*}
		Let $z\in B_{J}(x_i,0.5\rho^{k+l})\cap P$, and let $S\in\cP'$ be the unique cube from $\cP'$ containing $z$. Then,
		\begin{equation*}
		d_J(x_S, x_i)\le d_J(x_S, z)+d_J(z, x_i)\le 0.4 \rho^{k+l} + 0.5\rho^{k+l} \le \rho^{k+l},
		\end{equation*}
		and so $S$ is very close to $x_i$. This means that $S\in \cQ_i$, which implies $z\in S\subset Q_i$. Hence,
		\begin{equation*}
		B_{J}(x_i,0.5\rho^{k+l})\cap P\subset Q_i.
		\end{equation*}
		
		Now suppose that $z\in Q_i$, and let $S\in\cQ_i$ be the unique cube containing $z$. Note that $S$ is either very close or pretty close to $x_i$, so that $d_J(x_S, x_i)\le 3\rho^{k+l}$. It follows that 
		\begin{equation*}
		d_J(z,x_i)\le d_J(z,x_S)+d_J(x_S,x_i)\le \diam_{J}(S) + d_J(x_S,x_i) \le 0.4 \rho^{k+l} + 3\rho^{k+l}.
		\end{equation*}
		This gives $Q_i\subset B_{J}(x_i,4\rho^{k+l})\cap P$.
	\end{proof}
	\begin{remark}
		Recall that the main reason behind the construction above is to be able to shatter a cube, i.e. divide it into a family $\Div(P,J)= \bD_{k}(P,J)$. It may seem more natural to simply apply the David-Mattila lattice (\lemref{lem:dyadic cubes}) to $\mu|_P$ and $J$ to get the desired partition $\wt{\Sh}(P,J)\coloneqq\DD_k(P,J)$. However, it is not clear whether the cubes $Q\in\DD_k(P,J)$ satisfy
		\begin{equation*}
			\mu(Q)\gtrsim \HH(J)\rho^k
		\end{equation*}
		and we desperately need such lower bounds in subsequent arguments.
		
		On the other hand, for the partitions we constructed above such lower bounds are easy to obtain: any $Q\in\Sh(P,J)=\bD_k(P,J)$ contains some $Q'\in \bD_m$ with $\rho^m\sim \HH(J)\rho^{k}$, and then by lower Ahlfors regularity of $\mu$ we get
		\begin{equation*}
			\mu(Q)\ge\mu(Q')\ge \mu(B(x_Q,\rho^m))\gtrsim A^{-1}\rho^m\gtrsim A^{-1}\HH(J)\rho^k.
		\end{equation*}
	\end{remark}
	
	\section{Construction of a tree}\label{sec:tree-construction}
	In this section we construct a tree-like structure $\T$ consisting of subsets of $E$, which is adapted to our set of ``good directions'' 
	\begin{equation*}
	\{(x,\theta) : x\in E_0, \theta\in G_2(x)\},
	\end{equation*}
	and which we will use later on to prove the energy estimate \eqref{eq:E0energy-est}.

	Recall that for every $x\in E_0$ there exists a family of ``very good'' intervals $\cG_2(x)$ defined in Subsection \ref{subsec:verygooddirs}. Since the conical energy estimates are passed on to nearby points (see \lemref{lem:init-inter}), it is natural to consider families of good intervals corresponding to different scales.
	\begin{definition}\label{def:Gxk}
		For every $x\in E$ and $k\ge 0$ we define $\cG'(x,k)\subset\Delta(J_0)$ as the family of triadic intervals satisfying the following: $I\in \cG'(x,k)$ if and only if there exists $y\in E_0\cap B_I(x, 10\rho^k)$ such that $I\in\cG_2(y)$.
		
		We also define $\cG(x,k)$ as the collection of maximal intervals from $\cG'(x,k)$, and write
		\begin{equation*}
		G(x,k)=\bigcup_{I\in\cG(x,k)}I.
		\end{equation*}
	\end{definition}
	
 	We are going to use the collections $\cG(x,k)$ to define a tree-like structure $\T$ consisting of subsets of $E$. 
 	\begin{remark}
 		To be more precise, one should see $\T$ as a collection of triples $(Q,J_Q,k_Q)$, where $Q\subset E$ is a subset of $E$, $J_Q\in\Delta(J_0)$ is the interval of directions corresponding to $Q$, and $k_Q\in\mathbb{N}$ is the generation of $Q$. The shape of each cube $Q$ will be roughly $B_{J_Q}(x_Q,\rho^{k_Q})$. For simplicity, we will usually keep $J_Q$ and $k_Q$ implicit, so that statements like ``$Q,P\in\T,\, Q\neq P$'' should be understood as ``$(Q,J_Q,k_Q)\neq (P,J_P,k_P)$'', and the underlying sets $Q$ and $P$ might coincide.
 	\end{remark}

 	Recall that
 	\begin{equation*}
 	\ve\coloneqq c_\ve M^{-1}A^{-1},
 	\end{equation*}
 	where $0<c_\ve<1$ is a small absolute constant. We set
 	\begin{equation}
 		\E_2\coloneqq A\E_1.
 	\end{equation}	
 	
 	\begin{prop}\label{prop:tree}
 		Assume that the absolute constant $0<c_\ve<1$ is small enough. Then, there exists a collection $\T$ of subsets of $E$ satisfying the following properties:
 		\begin{enumerate}[label=(T{\arabic*})]
 			\item $\T=\bigcup_{k\ge 0}\T_k$, and for each $Q\in\T_k$ there is $x_Q\in Q$ and $J_Q\in\Delta(J_0)$ such that
 			\begin{equation*}
 			B(x_Q,\HH(J_Q)\rho^{k+3})\cap E\subset Q\subset B_{J_Q}(x_Q, 4\rho^k)\cap E,
 			\end{equation*}\label{list:Tbasic}
 			\item for every $Q\in\T_k,\ k\ge 1$, there exists $x\in E_0\cap Q$ with $J_Q\cap G_2(x)\neq\varnothing$, and
 			\begin{equation*}
 			\int_Q \HH(J_Q\cap G(x,k))\, d\mu(x)\ge (1-\ve)\HH(J_Q)\mu(Q),
 			\end{equation*}\label{list:big-good}
 			\item for every $Q\in\T_k$, $k\ge l\ge 0$, there is a unique $P\in\T_{l}$ such that $Q\subset P$ and $J_Q\subset J_P$; this introduces a partial order on $\T$, and we denote it by $Q\prec P$,\label{list:Torder}
 			\item if $Q\in\T_k,$ $P\in \T_l$, $k\ge l\ge 0$, then either $Q\prec P$, or $Q\cap P=\varnothing$, or $J_Q\cap J_P=\varnothing$,\label{list:disj}
 			\item \label{list:shat}there is a family $\Roots\subset\T$ satisfying
 			\begin{equation*}
 			\sum_{R\in\Roots}\HH(J_R)\mu(R)\lesssim \ve^{-1}\HH(J_0)\mu(E),
 			\end{equation*}
 			\item \label{list:Tstruc} there is a partition
 			\begin{equation*}
 			\T = \bigcup_{R\in\Roots}\T(R)
 			\end{equation*}
 			such that each $\T(R)$ is a subtree of $\T$ with root $R$, and
 			\begin{equation*}
 			\T(R)=\{Q\in\T\, :\, Q\prec R,\, J_Q= J_R \}.
 			\end{equation*}
% 			if $Q\in\T$ and $R\in\Roots$ is the minimal $\prec$-ancestor of $Q$ from $\Roots$, then $J_Q=J_R$. On the other hand, if $R\in\Roots$ and $Q\in\T\setminus\{R\}$ satisfy $R\prec Q$, then $J_R\subsetneq J_{Q}$.
 			%		, and $Q\in\DD_k(R,J_R)$.
 			\item for every $x\in E_0$, $J\in\cG_2(x)$, and $k\ge 0$, there is a unique $Q\in\T_k$ such that $x\in Q$ and $J\subset J_Q$. In particular, for every $k\ge0$ we have
 			\begin{equation*}
 			\{ (x,\theta)  : x\in E_0,\, \theta\in G_2(x)\}\subset \bigcup_{Q\in\T_k}Q\times J_Q.
 			\end{equation*}
 			\label{list:coverG2}
 			\item \label{list:ener-bdd-tree}for every $Q\in\T_k$ and $x\in Q$ we have
 			\begin{equation*}
 				\Bad(x,0.8J_Q,0,k)\lesssim \E_2,
 			\end{equation*}
 			where 
 			\begin{equation*}
 				\Bad(x,0.8J_Q,0,k)=\{0\le j\le k\, :\, X(x,0.8J_Q,\rho^{j+1},\rho^j)\cap E\neq\varnothing\}
 			\end{equation*}
 		\end{enumerate}
 	\end{prop} 	
	
	The construction of $\T$ above occupies the rest of this section. The properties \ref{list:Tbasic}--\ref{list:ener-bdd-tree} are proven in Section \ref{sec:prop-tree}.
	
	Before we begin the construction in earnest we describe the general idea. At the beginning of Subsection \ref{subsec:shattering} we discussed how the usual stopping time arguments are used to construct trees of cubes satisfying desired properties. Our construction is rather different because there is no fixed dyadic lattice on $E$ that works well for our purposes, but rather we have to use cubes corresponding to different directions at different scales. This could be seen as a stopping time argument moving in two directions (generation $k\in\mathbb{N}$ and directions $J\in\Delta(J_0)$), while normally there is only one direction (generation $k\in\mathbb{N}$). 
	
	Our algorithm is roughly the following.
	\vspace{0.5em}
	
	\emph{Step 1.} We define the $0$-th generation $\T_0 \subset \bD_0(E,J_0)$ consisting of cubes of sidelength $1$ corresponding to directions in $J_0$. This choice is somewhat arbitrary.
		
			\vspace{0.5em}
	\emph{Step 2.}	 The key properties of $\T$ we need are \ref{list:shat}--\ref{list:ener-bdd-tree}. This leads us to the following desired properties that the cubes from $\T$ should satisfy: 
		\begin{enumerate}[label=(S{\arabic*})]
			\item for every $Q\in\T$ there exists $x\in E_0\cap Q$ with $G_2(x)\cap J_Q\neq\varnothing$,\label{list:S1}
			\item  \label{list:S2} if $Q=(Q,J_Q,k)\in\T$ then
			\begin{equation*}
			\int_Q \HH(J_Q\cap G(y,k))\, d\mu(y)\ge (1-\ve)\HH(J_Q)\mu(Q).
			\end{equation*}
		\end{enumerate} 
		The condition \ref{list:S1} comes from \ref{list:coverG2}: if a cube fails \ref{list:S1} then it is useless to us. The condition \ref{list:S2} will be used to derive \ref{list:shat}--\ref{list:ener-bdd-tree}. Note that if \ref{list:S2} holds, then for most of $y\in Q$ the interval $J_Q$ is mostly covered by $G(y,k)$, the intervals good for $y$ at scale $k$.
	
		\vspace{0.5em}
	\emph{Step 3.}	
		 We add cubes to $\T$ inductively. Recall that in Remark \ref{rem:chsh} we introduced the notation $\Ch(Q,J_Q)$ and $\Sh(Q,J)$ corresponding to different partitions of $Q$. 
		 
		 If $Q=(Q,J_Q,k)\in\T$, then for every ``dyadic child'' $Q'\in\Ch(Q,J_Q)=\bD_{k+1}(Q,J_Q)$ we check whether $(Q', J_Q,k+1)$ satisfies the desired conditions \ref{list:S1} and \ref{list:S2}. If yes, add $(Q', J_Q, k+1)$ to $\T$. If not, there are two cases. 
		
		If \ref{list:S1} fails, then the construction stops at $Q'$. 
		
		If \ref{list:S1} holds, but \ref{list:S2} fails, then we cannot simply stop the construction; if we did, then our tree would not satisfy \ref{list:coverG2}. What we do instead is perform a ``shattering procedure'': for every triadic child $J\in\Ch_\Delta(J_Q)$ and every ``shattered subcube'' $P\in\Sh(Q',J)=\bD_{k+1}(Q',J)$ we check whether $(P,J,k+1)$ satisfies \ref{list:S1} and \ref{list:S2p}, a modified version of \ref{list:S2} which we discuss below. If yes, we set $J_P=J$ and add $(P,J_P,k+1)$ to $\T$. If \ref{list:S1} fails, we stop the construction at $P$. Finally, if \ref{list:S1} holds but \ref{list:S2p} fails, then we shatter $P$ again, and continue in a similar way with the shattered subcubes of $P$.
		
		This concludes the informal description of our algorithm.
				
		\begin{remark}
			The family $\Roots$ consists of $\T_0$ and the cubes $Q\in\T\setminus\T_0$ that arise from the shattering procedure. The packing condition \ref{list:shat} means that the shattering does not occur too often in $\T$. 
			
			In the construction above instead of \ref{list:S2} it is tempting to consider a simpler condition:
			\begin{enumerate}[label=(S{\arabic*}')]
				 \setcounter{enumi}{1}
				\item \label{list:S2p} if $Q\in\T$ then there exists $x\in E_0\cap Q$ with $J_Q\in\cG_2(x)$. 
			\end{enumerate}
			It is easy to see that any cube satisfying \ref{list:S2p} also satisfies \ref{list:S2}, see \lemref{lem:big-good}. Moreover, the tree obtained using \ref{list:S1} and \ref{list:S2p} would satisfy all the properties \ref{list:Tbasic}--\ref{list:ener-bdd-tree} except for \ref{list:shat}. Unfortunately, the packing condition for $\Roots$ is crucial to our arguments later on, and the condition \ref{list:S2p} is too restrictive for this to hold -- shattering procedure induced by the failure of \ref{list:S2p} may occur arbitrarily many times. This forces us to work with the more cumbersome condition \ref{list:S2}, although \ref{list:S2p} is still used within the shattering procedure itself.
		\end{remark}
%	\end{itemize}
	\subsection{Construction of $\T$}
	We define families $\T$ and $\Roots\subset\T$ one generation at a time. 
	
	First, let $m_0\in\mathbb{N}$ be the unique integer satisfying 
	\begin{equation*}
	\HH(J_0)\rho^2< 5\rho^{m_0}\le \HH(J_0)\rho.
	\end{equation*}
	We plug $k=0,$ $J=J_P=J_0$, $\cP=\bD_{m_0}$ and $l=0$ into \lemref{lem:gen-gen} to get a family $\bD_0(E,J_0)$ forming a partition of $E$. This can be seen as the cubes on $E$ of $0$-th generation relative to the metric $d_{J_0}$. We set 
	\begin{equation*}
	\T_0\coloneqq\{Q\in \bD_0(E,J_0) : E_0\cap Q\neq\varnothing\},
	\end{equation*}
	and $\T_0\cap\Roots \coloneqq \T_0$. For all $Q\in\T_0$ set $J_Q=J_0$. 
	
	Now assume that 
	\begin{equation*}
		\T^{k}\coloneqq \bigcup_{i=0}^{k}\T_i
	\end{equation*}
	and $\Roots\cap \T^{k}$ have already been defined, and that every $Q\in\T_{k}$ is a union of cubes from $\bD_{m(Q)}$, where $m(Q)$ satisfies $\HH(J_Q)\rho^{k+2}< 5\rho^{m(Q)}\le \HH(J_Q)\rho^{k+1}$. For every $Q\in\T_{k}$ we will define its $\T$-children, denoted by $\Ch(Q,\T),$ and then we will set $\T_{k+1}=\bigcup_{Q\in\T_{k}}\Ch(Q,\T)$. Every $P\in\Ch(Q,\T)$ will satisfy $P\subset Q$ and $J_P\subset J_Q$.
	
%	Let $R$ be the minimal $\prec$-ancestor of $Q$ from $\Sh$, so that $J_Q=J_R$ and there is no $P\in\Sh$ with $Q\prec P\prec R$, so that $Q\in\bD_{k-1}(R,J_R)$. Note that it may happen that $R=Q$. 
	
	\vspace{0.5em}
	Fix $Q\in\T_{k}$, and let $\cQ\subset\bD_{m(Q)}$ be such that $Q=\bigcup_{S\in\cQ}S$. Consider the family $\Ch(Q,J_Q)=\bD_{k+1}(Q,J_Q)$ obtained from \lemref{lem:gen-gen} by plugging in $k=k,$ $J_P=J=J_Q,$ $\cP=\cQ,$ and $l=1$. For every $P\in\Ch(Q,J_Q)$ we set $J_P=J_Q$. We divide $\Ch(Q,J_Q)$ into three subfamilies.
	\begin{itemize}
		\item If $P\in \Ch(Q,J_Q)$ fails \ref{list:S1}, that is,
		\begin{equation*}
			P\cap E_0=\varnothing,\quad\text{or}\quad G_2(x)\cap J_P=\varnothing\quad\text{for all $x\in P\cap E_0$,}
		\end{equation*}
		then we write $P\in\End_0(Q)$.
		\item If $P\in \Ch(Q,J_Q)\setminus\End_0(Q)$ fails \ref{list:S2}, that is,
		\begin{equation}\label{eq:shatter}
			\int_P \HH(J_P\cap G(x,k+1))\, d\mu(x)< (1-\ve)\HH(J_P)\mu(P),
		\end{equation}
		then we write $P\in\Sh_0(Q)$. 
		\item If $P\in \Ch(Q,J_Q)\setminus\End_0(Q)$ satisfies \ref{list:S2}, so that
		\begin{equation}\label{eq:no-shatter}
			\int_P \HH(J_P\cap G(x,k+1))\, d\mu(x)\ge (1-\ve)\HH(J_P)\mu(P),
		\end{equation}
	then we write $P\in\Good_0(Q)$.
	\end{itemize}
%	Note that for every $P\in\Good_0(Q)\cup\Sh_0(Q)$ there exists $x\in P\cap E_0$ with $G_2(x)\cap J_P\neq\varnothing$.
	
	The cubes from $\Good_0(Q)$ will be added to $\T_{k+1}$ and $\Ch(Q,\T)$. The cubes from $\End_0(Q)$ are not added to $\T$, this is where the tree ends. Finally, the cubes from $\Sh_0(Q)$ require further processing -- they need to be shattered.
	
	\subsection{Shattering procedure} The shattering procedure is defined inductively. Assume that $\Good_j(Q)$, $\End_j(Q)$ and $\Sh_j(Q)$ have already been defined for some $j\ge 0$, and if $j\ge 1$ then every $S\in\Good_j(Q)\cup\End_j(Q)\cup\Sh_j(Q)$ satisfies $S\in\Div(S',J_{S})$ for some $S'\in\Sh_{j-1}(Q)$ and $J_{S}\in\Ch_{\Delta}(J_{S'})$; for $j=0$ we have  $S\in\Ch(Q,J_{Q})$ and $J_S=J_Q$. 
	
	Assume further that they satisfy the following:
	\begin{itemize}
		\item If $S\in\End_j(Q)$, then it fails \ref{list:S1}: either $S\cap E_0=\varnothing$, or for all $x\in S\cap E_0$ we have $G_2(x)\cap J_S=\varnothing$.
		\item If $S\in\Good_j(Q)$ and $j\ge 1$, then it satisfies \ref{list:S2p}, that is, there exists $x\in E_0\cap S$ with $J_S\in\cG_2(x)$.
		\item If $S\in\Sh_j(Q)$, then it satisfies \ref{list:S1} but fails \ref{list:S2p}, that is, $S\notin\Good_j(Q)$ and there exists $x\in S\cap E_0$ with $G_2(x)\cap J_S\neq\varnothing$.
	\end{itemize}
	If $\Sh_j(Q)=\varnothing$, there is nothing to do and the shattering procedure ends. Suppose that $\Sh_j(Q)\neq\varnothing$, and let $S\in\Sh_j(Q)$.
%	
%	 Let $S'\in\Sh_{j-1}(Q)$ be as above (and if $j=0$ let $S'=Q$). Let $x\in S\cap E_0$ be such that $G_2(x)\cap J_S\neq\varnothing$. Note that $J_S\notin \cG_2(x)$, because if that was the case then we would have $S\in\Good_j(Q)$ (in the case $j=0$ if $J_S\in \cG_2(x)$, then \eqref{eq:no-shatter} holds). 
	
	Let $J\in\Ch_{\Delta}(J_S)$. We are going to partition the family $\Sh(S,J)=\bD_{k+1}(S,J)$ obtained using \lemref{lem:gen-gen}. Fix $P\in\Sh(S,J)$ and set $J_P\coloneqq J$. There are 3 possible outcomes:
	\begin{itemize}
		\item If $P$ fails \ref{list:S1}, i.e. $P\cap E_0=\varnothing$ or for all $x\in P\cap E_0$ we have $G_2(x)\cap J_P=\varnothing$, then we write $P\in\End_*(S,J)$.
		\item If $P$ satisfies \ref{list:S2p}, i.e. there exists $x\in E_0\cap P$ with $J_P\in \cG_2(x)$, then we write $P\in\Good_*(S,J)$. 
		\item If $P$ satisfies \ref{list:S1} but fails \ref{list:S2p}, i.e. $P\notin\Good_*(S,J)$ and there exists $x\in E_0\cap P$ with $J_P\cap G_2(x)\neq\varnothing$, then we write $P\in\Sh_*(S,J)$.
	\end{itemize}
	Finally, we define
	\begin{align*}
		\Good_{j+1}(Q) &\coloneqq \bigcup_{S\in\Sh_j(Q)}\bigcup_{J\in\Ch_{\Delta}(J_S)}\Good_*(S,J),\\
		\End_{j+1}(Q) &\coloneqq \bigcup_{S\in\Sh_j(Q)}\bigcup_{J\in\Ch_{\Delta}(J_S)}\End_*(S,J),\\
		\Sh_{j+1}(Q) &\coloneqq \bigcup_{S\in\Sh_j(Q)}\bigcup_{J\in\Ch_{\Delta}(J_S)}\Sh_*(S,J).\\
	\end{align*}
	This completes the description of the shattering procedure. 
		\begin{remark}\label{rem:finite-sh}
		It is not difficult to see that $\Sh_j(Q)=\varnothing$ for $j$ large enough. In other words, the shattering procedure ends after a finite number of ``shatterings'', and our algorithm never gets stuck in an infinite loop.
		
		Indeed, in \remref{rem:lowerbdintervals} we assumed that for every $x\in E_0$ and $I\in\cG_1(x)$ we have $\HH(I)\ge 3^{-N}$ for some huge $N\ge 1$. This implies that for every $I\in\cG_2(x)$ we have $\HH(I)\ge 3^{-N-1}$. At the same time, for every $P\in\Sh_j(Q)$ we have $\HH(J_P)=3^{-j}\HH(J_Q)$ because $J_P\subset J_Q$ is $j$ generations of triadic intervals below $J_Q$. It follows that $\Sh_{N+1}(Q)=\varnothing$.
	\end{remark}
	We define
	\begin{align*}
	\Ch(Q,\T)&\coloneqq\bigcup_{j\ge 0}\Good_j(Q),\\
	\Roots(Q)&\coloneqq \bigcup_{j\ge 1}\Good_j(Q).
	\end{align*}

%	\vspace{0.5em}
	Note that $\Roots(Q)$ are preciesly the $\T$-children of $Q$ that arise from the shattering procedure. Finally, we set
	\begin{align*}
		\T_{k+1} &\coloneqq \bigcup_{Q\in\T_{k}}\Ch(Q,\T),\\
		\T_{k+1}\cap\Roots &\coloneqq \bigcup_{Q\in\T_{k}}\Roots(Q),\\
		\T&\coloneqq\bigcup_{k\ge 0}\T_k,\\
		\Roots &\coloneqq \bigcup_{k\ge 0}\T_k\cap\Roots.
	\end{align*}
	This completes the construction of $\T$. 
	
		We define an auxiliary collection $\wt{\T}$ consisting of all the sets appearing in the construction above: $\wt{\T}_0=\T_0$ and
	\begin{align*}
	\wt{\T}_{k+1}&\coloneqq \bigcup_{Q\in\T_{k}}\bigcup_{j\ge 0}\Good_j(Q)\cup\End_j(Q)\cup\Sh_j(Q),\\
	\wt{\T}&\coloneqq\bigcup_{k\ge 0}\wt{\T}_k
	\end{align*}
	\begin{remark}\label{rem:constrT}
			It follows directly from the construction above that for every $k\ge 1$ and $Q\in \T_k$ there exists $P\in\T_{k-1}$ and $j\ge 0$ such that $Q\in\Good_j(P)$. Moreover, if $j=0$ then $Q\in\Ch(P,J_P)$ and $J_Q=J_P$, and if $j\ge 1$, then there exists a sequence of cubes $S_0,\dots, S_{j-1}$ with $S_i\in\Sh_i(P)$ and triadic intervals $J_0,\dots,J_{j-1}$ with $J_i=J_{S_i}$ such that
			\begin{itemize}
				\item $J_0=J_P$ and $S_0\in\Sh_0(P)\subset\Ch(P,J_P)$,
				\item $J_{i+1}\in\Ch_\Delta(J_{i})$ and $S_{i+1}\in\Sh_*(S_i,J_{i+1})\subset \Sh(S_i,J_{i+1})$
				\item $J_Q\in\Ch_\Delta(J_{j-1})$ and $Q\in\Good_*(S_{j-1},J_Q)\subset \Sh(S_{j-1},J_Q)$.
			\end{itemize}
	\end{remark}

	\section{Properties of the tree}\label{sec:prop-tree}
	In this section we establish properties \ref{list:Tbasic}--\ref{list:ener-bdd-tree} stated in \propref{prop:tree}.
	\subsection{Basic properties of $\T$}
	\vspace{0.5em} 
	
	For each $k\ge 0$ and $Q\in{\T}_k$ we set
	\begin{equation*}
	\ell(Q)=\rho^k
	\end{equation*}
	and
	\begin{equation*}
	B_Q\coloneqq B_{J_Q}(x_Q,4\ell(Q)).
	\end{equation*}
	\begin{lemma}\label{lem:Tbasic}
		Property \ref{list:Tbasic} holds: for each $Q\in\wt{\T}$ there is $x_Q\in Q$ and $J_Q\in\Delta(J_0)$ such that
		\begin{equation*}
			B(x_Q,\rho^3\HH(J_Q)\ell(Q))\cap E\subset Q\subset B_Q\cap E,
		\end{equation*}
	\end{lemma}
	\begin{proof}
		Let $Q\in\wt{\T}_k$. By the construction, we have $Q\in\bD_{k}(P,J_Q)$ for some subset $P\subset E$, and $Q$ is a union of cubes from a family $\cQ\subset\bD_{m(Q)}$, where $\HH(J_Q)\rho^{k+2}\le 5\rho^{m(Q)}\le \HH(J_Q)\rho^{k+1}$. By \lemref{lem:gen-gen}, there exists $S\in\cQ$ such that for $x_Q\coloneqq x_S$ we have
		\begin{equation*}
		B_{J_Q}(x_Q,0.5\rho^{k})\cap P\subset Q\subset B_{J_Q}(x_Q,4\rho^{k})\cap P.
		\end{equation*}
		Since $B_{J_Q}(x_Q,4\rho^{k})\cap P\subset B_Q\cap E$, we get $Q\subset B_Q\cap E$.
		
		On the other hand, by \lemref{lem:dyadic cubes} c) we have $E\cap B(x_S,\rho^{m(Q)})\subset E\cap B(S)\subset S$. Since $x_S=x_Q$ and
		\begin{equation*}
		\rho^3\HH(J_Q)\ell(Q) = \HH(J_Q)\rho^{k+3}\le 5\rho^{m(Q)+1}\le \rho^{m(Q)},
		\end{equation*}
		we get $B(x_Q,\rho^3\HH(J_Q)\ell(Q))\cap E\subset S\subset Q$.		
	\end{proof}
\begin{lemma}\label{lem:big-good}
	For every $Q\in\T$ we have \ref{list:S1}: there exists $x\in E_0\cap Q$ with $J_Q\cap G_2(x)\neq\varnothing$. 
	
	Moreover, if $Q\in\wt{\T}_k$ satisfies \ref{list:S2p}, i.e. there exists $x\in E_0\cap Q$ with $J_Q\in\cG_2(x)$, then $J_Q\subset G(y,k)$ for all $y\in Q$. In particular, it satisfies \ref{list:S2}.
\end{lemma}
\begin{proof}
	The fact that there exists $x\in E_0\cap Q$ with $J_Q\cap G_2(x)\neq\varnothing$ follows immediately from the construction because no cube failing \ref{list:S1} was added to $\T$.
	
	Now suppose $Q\in\wt{\T}_k$ and there exists $x\in E_0\cap Q$ with $J_Q\in\cG_2(x)$. By the definition of $\cG'(y,k)$, for all $y\in E\cap B_{J_Q}(x,10\rho^k)$ we have $J_Q\in\cG'(y,k)$, and in particular $J_Q\subset G(y,k)$. Using \lemref{lem:Tbasic} and the triangle inequality, we have
	\begin{equation*}
	Q\subset B_Q\cap E=B_{J_Q}(x_Q,4\rho^k)\cap E\subset B_{J_Q}(x,10\rho^k)\cap E,
	\end{equation*}
	and so $J_Q\subset G(y,k)$ for all $y\in Q$.
\end{proof}

	\begin{lemma}
		Property \ref{list:big-good} holds: for every $Q\in\T_k,\ k\ge 1$, and
		\begin{equation}\label{eq:blaa}
		\int_Q \HH(J_Q\cap G(x,k))\, d\mu(x)\ge (1-\ve)\HH(J_Q)\mu(Q).
		\end{equation}
	\end{lemma}
	\begin{proof}		
		By the definition of $\T_k$ we have $Q\in\Good_j(P)$ for some $j\ge 0$ and $P\in\T_{k-1}$.  If $j=0$, then \eqref{eq:blaa} follows immediately from the property \ref{list:S2}
	 defining $\Good_0(P)$.
		
		If $j\ge 1$, then $Q$ satisfies \ref{list:S2p} and this is even stronger than \eqref{eq:blaa} by \lemref{lem:big-good}.
	\end{proof}

	The following lemma prepares ground for establishing \ref{list:Torder} and \ref{list:disj}.
	\begin{lemma}\label{lem:Ch-cool}
		If $P\in\T$ and $Q\in\Ch(P,\T)$, then $Q\subset P$ and $J_Q\subset J_P$. If $Q\in\Ch(P,\T)\cap\Roots$, then $J_Q\subsetneq J_P$. 
		
		Moreover, if $Q,Q'\in\Ch(P,\T)$ satisfy $Q\neq Q'$, then either $Q\cap Q'=\varnothing$ or $J_Q\cap J_{Q'}=\varnothing$.
	\end{lemma}
	\begin{proof}
%		Since $Q\in \Ch(P,\T)$, we have $Q\in\Good_j(P)$ for some $j\ge 0$. If $j=0$, then $J_Q=J_P$ and $Q\subset P$, by definition of $\Good_0(P)$. 
%		
%		If $j\ge 1$, then $Q\in \Ch(P,\T)\cap \Roots$, and we have $Q\in\Good_*(S_{j-1},J_Q)$ for some $S_{j-1}\in\Sh_{j-1}(P)$ and $J_Q\in\Ch_{\Delta}(J_{S_{j-1}})$. Observe that for each $0\le i\le j-2$ there is $S_i\in\Sh_{i}(P)$ and such that $S_{i+1}\in\Sh_*(S_{i},J_{S_{i+1}})$, where $J_{S_{i+1}}\in\Ch_{\Delta}(J_{S_i})$. In particular,
%		\begin{gather*}
%		Q\subset S_{j-1}\subset\dots\subset S_0\subset P,\\
%		J_Q\subsetneq J_{S_{j-1}}\subsetneq J_{S_1}\subsetneq J_{S_0}=J_P.
%		\end{gather*}
		Let $j\ge 0$ be such that $Q\in\Good_j(P)$. If $j=0$, then $J_Q=J_P$ and $Q\in\Ch(P,J_P)$ so that $Q\subset P$. 
		
		If $j\ge 1$ (which happens if and only if $Q\in\Ch(P,\T)\cap \Roots$), let $S_i$ and $J_i$ be as in \remref{rem:constrT}. Clearly, we have $J_Q\subsetneq J_{j-1}\subsetneq\dots\subsetneq J_0=J_P$. At the same time, $Q\in\Sh(S_{j-1},J_Q)$, $S_{i+1}\in\Sh(S_i,J_{i+1})$, and $S_0\in\Ch(P,J_P)$. This implies
		\begin{equation*}
		Q\subset S_{j-1}\subset\dots\subset S_0\subset P,
		\end{equation*}
		and so we get the first assertion of the lemma.
		
	Now suppose that $Q,Q'\in\Ch(P,\T)$ satisfy $Q\neq Q'$. If $Q, Q'\in\Good_0(P)$, then we have $Q\cap Q'=\varnothing$ because they are distinct cubes from $\Ch(P,J_P)$. If $Q'\in\Good_0(P)$ and $Q\in\Good_j(P)$ for $j\ge 1$, then $Q\subset S_0$ for some $S_0\in\Sh_0(P)$. Since $Q'$ and $S_0$ are distinct cubes from $\Ch(P,J_P)$, we also get $Q\cap Q'=\varnothing$. 
	
	Finally, suppose that $Q\in\Good_j(P)$ and $Q'\in\Good_{j'}(P)$ for some $j,j'\ge 1$, and without loss of generality assume $j\ge j'$. If $J_Q\cap J_{Q'}=\varnothing$, we are done, so assume the contrary. Since both $J_Q$ and $J_{Q'}$ are triadic intervals, and $\HH(J_Q)=3^{-j}\HH(J_P),$ $\HH(J_{Q'})=3^{-j'}\HH(J_P),$ this implies $J_Q\subset J_{Q'}$. Let $S_i'\in\Sh_{i}(P)$ be the sequence of cubes associated to $Q'$ given by \remref{rem:constrT}. Since $J_Q\subset J_{Q'}$, and the intervals $\{J_{S_i}\}$ and $\{J_{S_i'}\}$ are uniquely defined by $J_Q$ and $J_{Q'}$, respectively, we get that $J_{S_i}=J_{S_i'}$ for all $0\le i\le j'-1$. 
	
	If $S_i=S_i'$ for all $0\le i\le j'-1$, then we get that either
	\begin{itemize}
		\item $j=j'$ and $Q,Q'\in\Good_*(S_{j'-1},J_{Q'})$, in which case $Q\cap Q'=\varnothing$ because they are distinct cubes from $\Sh(S_{j'-1},J_{Q'})$,
		\item $j>j'$, $Q'\in\Good_*(S_{j'-1},J_{Q'})$, and $S_{j'}\in\Sh_*(S_{j'-1},J_{Q'})$, in which case $Q\cap Q'=\varnothing$ because $Q\subset S_{j'}$ and $Q', S_{j'}$ are distinct cubes from $\Sh(S_{j'-1},J_{Q'})$.
	\end{itemize}
	Finally, suppose that there exists $0\le i\le j'-1$ with $S_i\neq S_i'$, and let $i$ be the smallest such integer. We have that either
	\begin{itemize}
		\item $i=0$, in which case $S_0, S_0'$ are distinct cubes from $\Ch(P,J_P)$. Hence, $S_0\cap S_0'=\varnothing$, and since $Q\subset S_0$ and $Q'\subset S_0'$, we get $Q\cap Q'=\varnothing$.
		\item $i>0$, in which case $S_i, S_i'$ are distinct cubes from $\Sh(S_{i-1},J_{S_i})$. But this again implies $S_i\cap S_i'=\varnothing$, so that $Q\cap Q'=\varnothing$.
	\end{itemize}
	Since we have covered all the cases, this shows that $J_Q\cap J_{Q'}\neq\varnothing$ implies $Q\cap Q'=\varnothing$.
	\end{proof}
	
	\begin{lemma}
		Properties \ref{list:Torder} and \ref{list:disj} hold: for every $Q\in\T_k$, $k\ge l\ge 0$, there is a unique $P\in\T_{l}$ such that $Q\subset P$ and $J_Q\subset J_P$. For every other $P'\in\T_{l}$ we have either $Q\cap P'=\varnothing$ or $J_Q\cap J_{P'}=\varnothing$.
	\end{lemma}
	\begin{proof}
		First, we cover the case $k=l$. We only need to show that if $P,Q\in\T_k$, $Q\cap P\neq\varnothing$ and $J_Q\cap J_P\neq\varnothing$, then $Q=P$. If $k=0$, then this is clear because $\T_0 \subset \bD_0(E,J_0)$. Now suppose that $k\ge 1$, and the claim is already known for $k-1$. Suppose that distinct cubes $P,Q\in\T_k$ satisfy $Q\cap P\neq\varnothing$ and $J_Q\cap J_P\neq\varnothing$. Let $Q',P'\in\T_{k-1}$ be such that $Q\in\Ch(Q',\T)$ and $P\in\Ch(P',\T)$. By \lemref{lem:Ch-cool}, we get that $Q'\neq P'$. But now, by the inductive assumption, we have either $Q'\cap P'=\varnothing$ or $J_{Q'}\cap J_{P'}=\varnothing$. By \lemref{lem:Ch-cool} we have $Q\subset Q',$ $J_Q\subset J_{Q'}$ and $P\subset P',$ $J_P\subset J_{P'}$, and so either $Q\cap P=\varnothing$ or $J_Q\cap J_P=\varnothing$. This finishes the proof in the case $k=l$.
		
		Now suppose that $k\ge l\ge 0$. It suffices to prove the lemma for $k=l+1$, and the general claim follows by induction.
		
		If $Q\in\T_{l+1}$, then there exists $P\in\T_{l}$ such that $Q\in\Ch(P,\T)$. By \lemref{lem:Ch-cool} we have $Q\subset P$ and $J_Q\subset J_P$. Now suppose that $P'\in\T_l$ is different from $P$. By the first part of the proof (case $k=l$), we have either $P\cap P'=\varnothing$, or $J_P\cap J_{P'}=\varnothing$. Hence, we get $Q\cap P'=\varnothing$ or $J_Q\cap J_{P'}=\varnothing$. This finishes the proof.
	\end{proof}

	We introduce the following notation: for every $Q\in\T_k, k\ge 1$ we denote by $Q^1\in\T_{k-1}$ its $\T$-parent: the unique cube from $\T_{k-1}$ such that $Q\prec Q^1$. Note that $Q\in\Ch(Q^1,\T)$. This is well-defined by \ref{list:Torder}.
	\vspace{0.5em}

	It will be convenient to have the following strengthening of the property \ref{list:S1}.
	\begin{lemma}\label{lem:intG2inT}
		For every $Q\in\T$ there exists $x\in E_0\cap Q$ and $J\in\cG_2(x)$ with $J\subset J_Q$.
	\end{lemma}
	\begin{proof}
		If $Q\in\T_0$ we have $J_Q=J_0$ and the result is trivial. Suppose $Q\in\T\setminus\T_0$, and let $P\in\T$ be such that $Q\in\Ch(P,\T)$. If $Q\in\Good_j(P)$ with $j\ge 1$, then by the construction $Q$ satisfies \ref{list:S2p} so that there exists $x\in E_0\cap Q$ with $J_Q\in\cG_2(x)$, and we are done. Assume that $Q\in\Good_0(P)$. 
		
		We know by property \ref{list:S1} that there exists $x\in E_0\cap Q$ with $J_Q\cap G_2(x)\neq\varnothing$. If there exists $J\in\cG_2(x)$ with $J\subset J_Q$, then we are done. Otherwise, we get that there exists $J\in\cG_2(x)$ with $J_Q\subsetneq J$ (because $\cG_2(x)\subset\Delta(J_0)$ and $J_Q\in\Delta(J_0)$). We will show that this leads to a contradiction.
		
		Let $k\ge 1$ be such that $Q\in\T_k$, and denote by $Q_i\in\T_i$ the unique cube such that $Q\prec Q_i$. This is well defined by \ref{list:Torder}. Since $J\subset J_0=J_{Q_0}$ and $J_Q=J_{Q_k}\subsetneq J$, we can find a cube $Q_i$ such that 
		\begin{equation*}
		J_{Q_i}\subsetneq J\subset J_{Q_{i-1}}.
		\end{equation*}
		It follows that $Q_i$ was defined using the shattering procedure. Since $J_{Q_i}\subsetneq J\subset J_{Q_{i-1}}$, there exists some $j\ge 0$ and $S\in\Sh_j(Q_{i-1})$ such that $J_S=J$ and $Q_i\prec S$. Recall that 
		\begin{equation*}
		x\in Q\cap E_0\subset Q_i\cap E_0\subset S\cap E_0
		\end{equation*}
		and $J_S=J\in\cG_2(x)$. But this implies $S\in \Good_j(Q_{i-1})$ (for $j\ge 1$ directly by the definition of $\Good_j(Q_{i-1})$, for $j=0$ by the second half of \lemref{lem:big-good}). This contradicts the fact that $S\in\Sh_j(Q_{i-1})$.
		\end{proof}

		We will need the following measure estimates for $Q\in\T$.
	\begin{lemma}\label{lem:cube-meas}
		For every $Q\in\T$ and $C\ge 1$ we have
		\begin{equation}\label{eq:meas-upper-bd}
		\mu(CB_Q)\lesssim CM\HH(J_Q)\ell(Q)
		\end{equation}
		and
		\begin{equation}\label{eq:meas-lower-bd}
		\mu(Q)\gtrsim A^{-1}\,\HH(J_Q)\ell(Q).
		\end{equation}
		In particular,
		\begin{equation}\label{eq:doubling}
		\mu(CB_Q)\lesssim CAM\,\mu(Q).
		\end{equation}
	\end{lemma}	
	\begin{proof}
		Suppose that $Q\in\T_k$. By \lemref{lem:intG2inT} there exists $x\in Q\cap E_0$ and $J\in\cG_2(x)$ such that $J\subset J_Q$.
		
		By \lemref{lem:find-bdd-dir}, there exists $\theta\in 3J$ such that $\mu_{\theta}^\perp(y)\le M$. Since $J\subset J_Q$, we also have $\theta\in 3J_Q$. Since $x\in B_Q$,
		we have $CB_Q\subset B_{J_Q}(x,5C\ell(Q))$.
		\lemref{lem:bound-rectang} gives
		\begin{equation*}
		\mu(CB_Q)\le\mu(B_{J_Q}(x,5C\ell(Q)))\lesssim CM\HH(J_Q)\ell(Q)
		\end{equation*}
		This concludes the proof of \eqref{eq:meas-upper-bd}.
		
		Concerning \eqref{eq:meas-lower-bd}, by \ref{list:Tbasic} and lower Ahlfors regularity of $E$ we have
		\begin{equation*}
		\mu(Q)\ge\mu(B(x_Q,\rho^3\HH(J_Q)\ell(Q))\cap E)\gtrsim A^{-1}\HH(J_Q)\ell(Q).
		\end{equation*}
	\end{proof}

	\subsection{Coronization of $\T$}
	The properties \ref{list:shat} and \ref{list:Tstruc} can be seen as a ``coronization'' of $\T$, in the spirit of \cite{david1991singular}. We prove them in this subsection.
	
	Given $R\in\Roots$, let $\Next(R)$ denote the family of cubes $P\in\Roots\setminus\{R\}$ such that $P\prec R$ and there is no other $S\in\Roots\setminus\{P,R\}$ satisfying $P\prec S\prec R$. We can decompose $\Roots$ into ``layers'': set $\Roots_0 = \T_0$, and then inductively
	\begin{equation}\label{eq:Sh-layer}
		\Roots_{k+1}=\bigcup_{P\in\Roots_k}\Next(P),
	\end{equation}
	so that
	\begin{equation}\label{eq:shlay}
		\Roots = \bigcup_{k=0}^\infty\Roots_k,
	\end{equation}
	and the union above is disjoint.
	
	Given $R\in\Roots$ denote by $\T(R)$ the cubes $Q\in\T$ such that $Q\prec R$ and there is no $P\in\Next(R)$ such that $Q\prec P\prec R$. In particular, $\Next(R)$ is not part of $\T(R)$.	
	We get a decomposition
	\begin{equation}\label{eq:tree-decomp}
		\T = \bigcup_{R\in\Roots}\T(R),
	\end{equation}
	where the union above is disjoint.
	
	\begin{lemma}\label{lem:Tstruc}
		The property \ref{list:Tstruc} holds: if $R\in\Roots$ then
		\begin{equation}\label{eq:TRstruc}
		\T(R)=\{Q\in\T\, :\, Q\prec R,\, J_Q= J_R \}.
		\end{equation}
%		On the other hand, if $R\in\Roots$ and $R'\in\T\setminus\{R\}$ satisfy $R\prec R'$, then $J_R\subsetneq J_{R'}$
	\end{lemma}
	\begin{proof}
		We begin by proving ``$\subset$'' in \eqref{eq:TRstruc}. Suppose that $Q\in\T(R)$. We need to show that $J_Q=J_R$. If $Q=R$ then there is nothing to prove, so suppose that $Q\neq R$.	
		 Assume that $Q\in\T_k$, and let $Q=Q^0\prec Q^1\prec\dots\prec Q^l=R$ be the sequence of cubes such that $Q^i\in\T_{k-i}$ (it exists and it is unique by \ref{list:Torder}). Since none of the $Q^0,\dots, Q^{l-1}$ belong to $\Roots$, it follows directly from the definition of the family $\Roots$ that $Q^{i}\in\Good_0(Q^{i+1})$ for all $i\in\{0,\dots, l-1\}$. But going back to the definition of $\Good_0(Q^{i+1})$, we see that $J_{Q^i}=J_R$. This proves $\T(R)\subset\{Q\in\T\, :\, Q\prec R,\, J_Q= J_R \}$.
		 
		 Now assume that $Q\prec R$ and $J_Q=J_R$. If $Q\notin\T(R)$, then there exists $R'\in\Roots\setminus\{R\}$ with $Q\prec R'\prec R$, and in particular $J_Q\subset J_{R'}\subset J_R$. Since $J_Q=J_R$, we get
		 \begin{equation}\label{eq:rgrg}
		 J_{R'}=J_R.
		 \end{equation}
		 Let $P\in\T$ be such that $R'\in\Ch(P,\T)$, so that $R'\prec P\prec R$ (it may happen that $P=R$). Since $R\in\Ch(P,\T)\cap\Roots$, by \lemref{lem:Ch-cool} we have $J_R\subsetneq J_P\subset J_R$, but this contradicts \eqref{eq:rgrg}. Hence, $Q\in\T(R)$.
	\end{proof}

	Observe that since all $Q\in\T(R)$ satisfy $J_Q=J_R$, the family $\T(R)$ is a tree of generalized dyadic cubes in the ``usual sense'', with partial order given by the inclusion relation.
%	, and it has the property: if $Q\in\Tree(R)$, $P\in\bD(R,J_R)$, and $Q\subset P\subset R$, then $P\in\Tree(R)$.

	For $R\in\Roots$ we define
	\begin{equation}\label{eq:stop-struct}
	\cS(R)\coloneqq\bigcup_{Q\in\T(R)}\Sh_0(Q)\cup\End_0(Q).
	\end{equation}
	Observe that for every $P\in\cS(R)$ there exists $P^1\in\T(R)$ such that $P\in\Ch(P^1,J_R)$, but $P\notin\T(R)$. Note also that $\cS(R)\subset\wt{\T}\setminus\T$.
	
	We may further divide
	\begin{equation}\label{eq:stop-sh}
	\cS_{\Sh}(R)\coloneqq\bigcup_{Q\in\T(R)}\Sh_0(Q)
	\end{equation}
	and
	\begin{equation*}
	\cS_{\End}(R)\coloneqq\bigcup_{Q\in\T(R)}\End_0(Q),
	\end{equation*}
	so that $\cS(R)=\cS_{\Sh}(R)\cup \cS_{\End}(R)$. It follows from the definitions of $\Sh_0(Q)$ and $\End_0(Q)$ that $\cS_{\Sh}(R)\cap \cS_{\End}(R)=\varnothing$.
%	\begin{lemma}
%		We have
%		
%		and in particular $\cS(R)\cap\T = \varnothing$.
%	\end{lemma}
%	\begin{proof}
%		 Note that if $S\in\cS(R)$, then $S\in\Ch(S^1,J_R)\setminus\Good_0(S^1) = \Sh_0(S^1)\cup\End_0(S^1)$. Since $S^1\in\T(R)$, we have $S\in \bigcup_{Q\in\T(R)}\Sh_0(Q)\cup\End_0(Q)$. This shows $\Stop(R)=\bigcup_{Q\in\T(R)}\Sh_0(Q)\cup\End_0(Q)$. 
%		 
%		 To see the converse, observe that if $S\in\bigcup_{Q\in\T(R)}\Sh_0(Q)\cup\End_0(Q)$, then $S^1\in\T(R)$, and $S\notin\T$. In particular, $S\notin\T(R)$, and so $S\in\Stop(R)$.
%	\end{proof}

	\begin{lemma}\label{lem:stop-pack}
		The cubes in $\cS(R)$ are pairwise disjoint, and we have
		\begin{equation}\label{eq:stop-pack}
			\sum_{Q\in\cS(R)}\mu(Q)\le \mu(R).
		\end{equation}
	\end{lemma}
	\begin{proof}
		Suppose that $Q,P\in\cS(R)$, $Q\neq P$, and $Q\cap P\neq\varnothing$. Let $Q^1,P^1\in\T(R)$ be such that $Q\in\Ch(Q^1,J_R)$ and $P\in\Ch(P^1,J_R)$. Then, $Q^1\cap P^1\neq\varnothing$ and by \ref{list:disj} we have either $Q^1\subset P^1$ or $P^1\subset Q^1$; without loss of generality assume the former. Then, there exists $\wt{Q}\in\Good_0(P^1)$ such that $Q^1\subset\wt{Q}$. At the same time, since $P\in\Sh_0(P^1)\cup\End_0(P^1)$ and $\Good_0(P^1),$ $\Sh_0(P^1),$ $\End_0(P^1)$ are disjoint subfamilies of $\Ch(P^1,J_R)$, we get $P\cap \wt{Q}=\varnothing$, which is a contradiction with $P\cap Q\neq\varnothing$. This shows that the cubes in $\cS(R)$ are pairwise disjoint.
%		Then, either $Q\subsetneq P$ or $P\subsetneq Q$; without loss of generality assume the former. Since $Q^1\in\T(R)$ and $Q^1\subset P\subset R$, we get $P\in\T(R)$. But this contradicts with $P\in\cS(R)$. 
		
		The estimate \eqref{eq:stop-pack} follows from the disjointedness and the fact that $Q\subset R$ for all $Q\in\cS(R)$.
	\end{proof}

	We move on to the proof of the packing estimate for $\Roots$, that is, property \ref{list:shat}. We start with some preliminary lemmas.
	\begin{lemma}\label{lem:sh-divide}
		If $R\in\Roots$, then
		\begin{equation*}
			\Next(R) = \bigcup_{S\in\cS_\Sh(R)}\{Q\in\Roots(S^1):Q\subset S\},
		\end{equation*}
		and the union above is disjoint.
	\end{lemma}
	\begin{proof}
		First, let us show that if $P\in\T$, then
		\begin{equation}\label{eq:bla}
			\Roots(P) = \bigcup_{S_0\in\Sh_0(P)}\{Q\in\Roots(P) : Q\subset S_0\},
		\end{equation}
		and the unions above are disjoint.
		
		Recall that $\Roots(P)= \bigcup_{j\ge 1}\Good_j(P) = \Ch(P,\T)\setminus\Good_0(P)$. Recalling \remref{rem:constrT}, for every $Q\in\Roots(P)$ there is a unique sequence of cubes $Q\prec S_k\prec \dots\prec S_0$ such that $S_i\in\Sh_i(P)$ and $Q\in\Good_*(S_k,J_Q)$. Since $\Sh_0(P)$ consists of pairwise disjoint cubes, we get \eqref{eq:bla}.
		
		Observe that $\Next(R) = \bigcup_{P\in\T(R)}\Roots(P),$ and that $\Roots(P)$ is non-empty if and only if $\Sh_0(P)$ is non-empty. Note also that if $S\in\End_0(P)$, then $\{Q\in\Roots(P):Q\subset S\}=\varnothing.$ Recalling \eqref{eq:stop-sh} we get
		\begin{multline*}
			\Next(R) = \bigcup_{P\in\T(R)}\bigcup_{S_0\in\Sh_0(P)}\{Q\in\Roots(P):Q\subset S_0\}\\
			= \bigcup_{S\in\cS_\Sh(R)}\{Q\in\Roots(S^1):Q\subset S\}.
		\end{multline*}
	\end{proof}
%	To prove \ref{list:shat}, we will need the following auxiliary estimate.
	The following estimate is key to obtaining \ref{list:shat}.
	\begin{lemma}\label{lem:sh-decrease}
		If $P\in\T$ and $S_0\in\Sh_0(P)$, then
		\begin{equation}\label{eq:sh-decrease}
			\sum_{S\in\Roots(P) : S\subset S_0}\HH(J_S)\mu(S)\le (1-\ve)\HH(J_P)\mu(S_0).
		\end{equation}
	\end{lemma}
	\begin{proof}
		%		By \lemref{lem:sh-divide} we have
		%		\begin{equation*}
			%		\sum_{S\in\Sh(P)}\HH(J_S)\mu(S) = \sum_{S_0\in\Sh_0(P)}\sum_{S\in\Sh(P) : S\subset S_0}\HH(J_S)\mu(S).
			%		\end{equation*}
		Fix $S_0\in\Sh_0(P)$. Let $k\in \mathbb{N}$ be such that $P\in\T_{k-1}$, so that for all $S\in\Roots(P)$ we have $S\in\T_k$.
		By the construction, for every $S\in\Roots(P)$ we have \ref{list:S2p}: there exists $x\in E_0\cap S$ with $J_S\in\cG_2(x)$.  Then, by \lemref{lem:big-good} for all $y\in S$ we have $J_S\subset G(y,k)$. It follows that
		\begin{multline*}
			\sum_{S\in\Roots(P):S\subset S_0}\HH(J_S)\mu(S) = \sum_{S\in\Roots(P):S\subset S_0}\int_S \HH(J_S\cap G(y,k))\, d\mu(y)\\
			= \sum_{S\in\Roots(P):S\subset S_0}\int_S\int_{J_S} \one_{G(y,k)}(t)\,  dt\, d\mu(y).
		\end{multline*}	
		Observe that \lemref{lem:Ch-cool} can be rephrased by saying that $J_Q\times Q\subset J_P\times P$ for all $Q\in\Ch(P,\T)$, and $\{J_Q\times Q\ :\ Q\in\Ch(P,\T)\}$ is a family of pairwise disjoint sets. Since $\Roots(P)\subset \Ch(P,\T)$, the same is true for $\{J_S\times S\ :\ S\in\Roots(P)\}$. Recalling also that $J_{S_0}=J_P$, and that cubes from $\Sh_0(P)$ fail \ref{list:S2}, we get
		\begin{multline*}
			\sum_{S\in\Roots(P):S\subset S_0}\int_S\int_{J_S} \one_{G(y,k)}(t)\,  dt\, d\mu(y) \le \int_{S_0}\int_{J_P} \one_{G(y,k)}(t)\,  dt\, d\mu(y)\\
			=\int_{S_0} \HH(J_P\cap G(y,k))\, d\mu(y) \overset{\eqref{eq:shatter}}{<} (1-\ve)\HH(J_P)\mu(S_0).
		\end{multline*}
		%		Putting together the estimates above and summing over $S_0\in\Sh_0(P)$, we arrive at
		%		\begin{equation*}
			%		\sum_{S\in\Sh(P)}\HH(J_S)\mu(S) \le\sum_{S_0\in\Sh_0(P)}(1-\ve)\HH(J_P)\mu(S_0)\le (1-\ve)\HH(J_P)\mu(P).
			%		\end{equation*}
	\end{proof}
	
	\begin{lemma}\label{lem:sh-decr}
		For every $R\in\Roots$ we have
		\begin{equation}\label{eq:sh-decr}
			\sum_{P\in\Next(R)}\HH(J_P)\mu(P)\le (1-\ve)\HH(J_R)\mu(R).
		\end{equation}
	\end{lemma}
	\begin{proof}
		By \lemref{lem:sh-divide},
		\begin{align*}
			\sum_{P\in\Next(R)}\HH(J_P)\mu(P)=&\sum_{S\in\cS_\Sh(R)}\sum_{P\in\Roots(S^1):P\subset S}\HH(J_P)\mu(P)\\
			= &\sum_{Q\in\T(R)}\sum_{S\in\Sh_0(Q)}\sum_{P\in\Roots(Q):P\subset S}\HH(J_P)\mu(P)\\
			\overset{\eqref{eq:sh-decrease}}{\le}& (1-\ve)\sum_{Q\in\T(R)}\sum_{S\in\Sh_0(Q)}\HH(J_Q)\mu(S)\\
			 =&(1-\ve) \HH(J_R)\sum_{S\in\cS_{\Sh}(R)}\mu(S)\le (1-\ve)\HH(J_R)\mu(R),
		\end{align*}
		where in the last line we used the fact that $J_Q=J_R$ for $Q\in\T(R)$ and that cubes from $\cS(R)$ are pairwise disjoint.
	\end{proof}

	\begin{lemma}\label{lemma:Sh-packs}
		Property \ref{list:shat} holds: the family $\Roots$ satisfies a packing estimate
		\begin{equation}\label{eq:Sh-packs}
			\sum_{Q\in\Roots}\HH(J_Q)\mu(Q)\le \ve^{-1} \HH(J_0)\mu(E).
		\end{equation}
	\end{lemma}
	\begin{proof}
		Recall that we had a decomposition of $\Roots$ into layers $\Roots_k$, see \eqref{eq:Sh-layer} and \eqref{eq:shlay}. It follows that
		\begin{multline*}
			\sum_{Q\in\Roots}\HH(J_Q)\mu(Q) = \sum_{k\ge 0}\sum_{Q\in\Roots_k}\HH(J_Q)\mu(Q)\\
			 = \sum_{Q\in\Roots_0}\HH(J_Q)\mu(Q) + \sum_{k\ge 1}\sum_{Q\in\Roots_k}\HH(J_Q)\mu(Q).
		\end{multline*}
		If $k\ge 1$ then
		\begin{equation*}
			\sum_{Q\in\Roots_k}\HH(J_Q)\mu(Q) = \sum_{R\in\Roots_{k-1}}\sum_{Q\in\Next(R)}\HH(J_Q)\mu(Q)\overset{\eqref{eq:sh-decr}}{\le}(1-\ve)\sum_{R\in\Roots_{k-1}}\HH(J_R)\mu(R).
		\end{equation*}	
		Iterating this estimate, we get
		\begin{equation*}
			\sum_{Q\in\Roots_k}\HH(J_Q)\mu(Q) \le(1-\ve)^k\sum_{R\in\Roots_{0}}\HH(J_R)\mu(R).
		\end{equation*}	
		Recalling that $\Roots_0=\T_0\subset\bD_0(E,J_0)$, and that the cubes $R\in\bD_0(E,J_0)$ are pairwise disjoint with $J_R=J_0$, we finally arrive at
		\begin{multline*}
			\sum_{Q\in\Roots}\HH(J_Q)\mu(Q)\le \sum_{k\ge 0}(1-\ve)^k\sum_{R\in\Roots_{0}}\HH(J_R)\mu(R) \\
			\le \HH(J_0)\mu(E)\sum_{k\ge 0}(1-\ve)^k = \ve^{-1}\HH(J_0)\mu(E).
		\end{multline*}
	\end{proof}

	\subsection{Relating $\T$ with good directions}
	In this subsection we prove \ref{list:coverG2} and \ref{list:ener-bdd-tree}, which show how the tree $\T$ relates to the good directions.
	\begin{lemma}\label{lem:coverG2}
		Property \ref{list:coverG2} holds: for every $x\in E_0$, $J\in\cG_2(x)$, and $k\ge 0$, there is a unique $Q\in\T_k$ such that $x\in Q$ and $J\subset J_Q$.
	\end{lemma}
	\begin{proof}
		Fix $x\in E_0$ and $J\in\cG_2(x)$. Observe that by \ref{list:disj} there exists at most one $Q\in\T_k$ such that $x\in Q$ and $J\subset J_Q$. Thus, we only need to prove the existence of such $Q\in\T_k$ for every $k\ge 0$.
		
		We prove the lemma by induction on $k$. First, observe that for $k=0$ the claim is certainly true: since $\T_0\subset\bD_0(E,J_0)$ consists of all the cubes in $\bD_0(E,J_0)$ intersecting $E_0$, there exists a unique cube $Q\in\T_0$ that contains $x$, and it satisfies $J_Q=J_0\supset J$ (recall that $\cG_2(x)\subset\Delta(J_0)$).
		
		Now suppose that the lemma is true for $k-1$, so that there exists $P\in\T_{k-1}$ with $x\in P$ and $J\subset J_P$. We will find $Q\in\Ch(P,\T)$ such that $J\subset J_Q$. This will close the induction because $\Ch(P,\T)\subset\T_k$.
		
		We need to trace the algorithm defining $\Ch(P,\T)$. First, we consider the cubes in $\Ch(P,J_P)$. Since $\Ch(P,J_P)$ is a partition of $P$, there exists a unique $S_0\in\Ch(P,J_P)$ such that $x\in S_0$. There are three possibilities.
		
		If $S_0\in\Good_0(P)$, then $S_0\in\Ch(P,\T)\subset\T_k$ and $J_{S_0}=J_P\supset J$, so we can take $Q=S_0$ and we are done. 
		
		If $S_0\in\End_0(P)$, then it fails \ref{list:S1}: either $S_0\cap E_0=\varnothing$, or $G_2(x)\cap J_P=\varnothing$. The former is false because $x\in S_0\cap E_0$. The latter is false because $J\subset G_2(x)\cap J_P$. We conclude that $S_0\notin\End_0(P)$.
		
		The remaining alternative is that $S_0\in\Sh_0(P)$. We claim that this implies $J\subsetneq J_P$. Indeed, if $J=J_P$, then by \lemref{lem:big-good} we have $J_P=J_{S_0}\subset G(y,k)$ for all $y\in S_0$ (recall that $x\in S_0\cap E_0$ and $J\in\cG_2(x)$). But this implies \ref{list:S2}:
		\begin{equation*}
		\int_{S_0} \HH(J_{S_0}\cap G(y,k))\, d\mu(y)=\HH(J_{S_0})\mu(S_0)\ge (1-\ve)\HH(J_{S_0})\mu(S_0),
		\end{equation*}
		and so $S_0\in\Good_0(P)$, which is a contradiction. Thus, $J\subsetneq J_P$. 
		
		Since $J$ and $J_P$ are triadic intervals and $J\subsetneq J_P$, we get that there is a unique $J_1\in\Ch_\Delta(J_P)$ such that $J\subset J_1$. We partition $S_0$ into $\Div(S_0,J_1)$, so that there is a unique $S_1\in\Div(S_0,J_1)$ such that $x\in S_1$. As before, there are three cases to consider.
		
		If $S_1\in\Good_*(S_0,J_1)$, then $S_1\in\Good_1(P)\subset\Ch(P,\T)$, and we may take $Q=S_1$ because $x\in S_1$ and $J\subset J_1=J_{S_1}$.
		
		If $S_1\in\End_*(S_0,J_1)$, then either $S_1\cap E_0=\varnothing$, or $G_2(x)\cap J_1=\varnothing$; both are false, so this case does not occur.
		
		Finally, it may happen that $S_1\in\Sh_*(S_0,J_1)$. In this case we have $J\subsetneq J_1$ because $J=J_1$ would imply $S_1\in\Good_*(S_0,J_1)$, as before. Hence, there is a unique $J_2\in\Ch_{\Delta}(J_1)$ and $S_2\in\Div(S_1,J_2)$ such that $x\in S_2$ and $J\subset J_2$. 
		
		Arguing this way we obtain a sequence $\{S_i\}\subset\wt{\T}_k$ such that $S_i\supset S_{i+1}$, $J_{S_i}\supsetneq J_{S_{i+1}}$, $x\in S_i$, and $J\subset J_{S_i}$. If there is $j\ge 0$ such that $S_j\in\Good_j(P)$, then we can take $Q=S_j$ and we are done because $Q\in\Ch(P,\T)$. 
		
		If there is no such $j\ge 0$, then the sequence $\{S_i\}$ is infinite and we have $S_i\in\Sh_i(P)$ for all $i\in\mathbb{N}$. But this is impossible by \remref{rem:finite-sh} (or simply by recalling that $J\subset J_{S_i}$ and $\HH(J_{S_i})= 3^{-i}\HH(J_P)$).
		
		In conclusion, we have found $Q\in\T_k$ such that $x\in Q$ and $J\subset J_Q$. This closes the induction.
	\end{proof}

	We will need the following corollary. It can be seen as a converse to \lemref{lem:intG2inT}.
	\begin{lemma}\label{lem:int-contain}
		If $x\in E_0$, $J\in\cG_2(x)$, and $x\in Q\in\T$ satisfies $J\cap J_Q\neq\varnothing$, then $J\subset J_Q$. 
	\end{lemma}
	\begin{proof}
		Recall that both $J$ and $J_Q$ are triadic intervals. Since $J\cap J_Q\neq\varnothing$ we have either $J\subset J_Q$ or $J_Q\subsetneq J$. Suppose it is the latter.
		
		Let $k\ge 0$ be such that $Q\in\T_k$. By \lemref{lem:coverG2}, there exists a unique $P\in\T_k$ such that $x\in P$ and $J\subset J_P$. Since $J_Q\subsetneq J$, we get that $J_Q\subsetneq J_P$. Taking into account that $x\in Q\cap P$, \ref{list:disj} implies $Q\prec P$. Then, by \ref{list:Torder} (applied with $k=l$, and using the uniqueness assertion) we get that $Q=P$. This contradicts $J_Q\subsetneq J_P$.
	\end{proof}
%	\subsection{Conical energy on the tree}
	We move on to the proof of \ref{list:ener-bdd-tree}, which we recall below.
	\begin{lemma}\label{lem:ener-bdd-tree}
		Property \ref{list:ener-bdd-tree} holds: for every $Q\in\T_k$ and $x\in Q$ we have
		\begin{equation}\label{eq:ener-bdd-tree}
		\Bad(x,0.8J_Q,0,k)\lesssim \E_2,
		\end{equation}
		where 
		\begin{equation*}
			\Bad(x,0.8J_Q,0,k)=\{0\le j\le k\, :\, X(x,0.8J_Q,\rho^{j+1},\rho^j)\cap E\neq\varnothing\}
		\end{equation*}
		and 
		$\E_2=A\E_1$.
		%		In particular, there exists $x\in Q$ where \eqref{eq:ener-bdd-tree} holds.
	\end{lemma}
	We split the proof into several steps.
	
	\begin{lemma}
		For every $Q\in\T_k$ there exists $x_0\in Q$ such that
		\begin{equation}\label{eq:xnotbad}
		\HH(J_Q\cap G(x_0,k))\ge (1-2\ve)\HH(J_Q).
		\end{equation}
	\end{lemma}
	\begin{proof}
		By the property \ref{list:big-good} of $\T$ we know that
		\begin{equation}\label{eq:dd}
		\int_Q \HH(J_Q\cap G(x,k))\, d\mu(x)\ge (1-\ve)\HH(J_Q)\mu(Q).
		\end{equation}
		At the same time, if $Q_B\coloneqq \{x\in Q\,:\, \HH(J_Q\cap G(x,k))\le (1-2\ve)\HH(J_Q)  \}$ then
		\begin{multline*}
		\int_Q \HH(J_Q\cap G(x,k))\, d\mu(x) \le \int_{Q_B}(1-2\ve)\HH(J_Q)\, d\mu(x) + \int_{Q\setminus Q_B}\HH(J_Q)\, d\mu(x)\\
		= \HH(J_Q)((1-2\ve)\mu(Q_B) + \mu(Q\setminus Q_B))=\HH(J_Q)\mu(Q) -2\ve\HH(J_Q)\mu(Q_B).
		\end{multline*}
		Comparing this to \eqref{eq:dd} we get
		\begin{equation*}
		(1-\ve)\HH(J_Q)\mu(Q)\le \HH(J_Q)\mu(Q) -2\ve\HH(J_Q)\mu(Q_B),
		\end{equation*}
		which gives $\mu(Q_B)\le \mu(Q)/2$. In particular, there exists $x_0\in Q\setminus Q_B$, so that \eqref{eq:xnotbad} holds.
	\end{proof}
		Recall that by \eqref{eq:init-inter} for every $x\in E_0$, $I\in\cG_2(x)$ and $y\in B_I(x,10\rho^k)$ we have the pointwise energy estimate
	\begin{equation}\label{eq:ener-Gyk}
	\int_{\rho^k}^1 \frac{\mu(X(y, 1.5I, \rho r, r))}{r}\frac{dr}{r}\lesssim \E_1\HH(I).
	\end{equation}
	In particular, by the definition of $\cG(y,k)$ (Definition \ref{def:Gxk}) the estimate above holds for every $y\in E$ and $I\in\cG(y,k)$.	

		Fix $x_0\in Q$ satisfying \eqref{eq:xnotbad}.	
		We claim that
		\begin{equation}\label{eq:des1}
			\#\Bad(x_0,0.85J_Q,0,k)\lesssim A \E_1=\E_2.
		\end{equation}
		
		Recall that $G(x_0,k)=\bigcup_{I\in\cG(x_0,k)}I$, where $\cG(x_0,k)$ is a family of disjoint triadic intervals, and $J_Q$ is also a triadic interval. There are two cases to consider: either for every $I\in\cG(x_0,k)$ with $I\cap J_Q=\varnothing$ we have $I\subset J_Q$, or there exists (a unique) $I\in\cG(x_0,k)$ with $J_Q\subset I$. We consider each case in a separate lemma.
		\begin{lemma}\label{lem:bad-scales-est}
			Assume that for every $I\in\cG(x_0,k)$ with $I\cap J_Q=\varnothing$ we have $I\subset J_Q$. If $0<c_\ve<1$ is chosen small enough, then we have
			\begin{equation}\label{eq:dess}
				\int_{\rho^k}^{1}\frac{\mu(X(x_0,0.9J_Q,\rho r,r))}{r}\, \frac{dr}{r}\lesssim \E_1\HH(J_Q),
			\end{equation}
			and consequently \eqref{eq:des1} holds.
		\end{lemma}
		\begin{proof}
		We begin the proof of \eqref{eq:dess}. First, we apply \eqref{eq:ener-Gyk} to each $I\in\cG(x_0,k)$ satisfying $I\cap J_Q=\varnothing$ to get
		\begin{multline}\label{eq:suppl-est}
		\int_{\rho^k}^1 \frac{\mu(X(x_0,J_Q\cap G(x_0,k), \rho r, r))}{r}\frac{dr}{r}=
		\sum_{I\in\cG(x_0,k), I\subset J_Q}\int_{\rho^k}^1 \frac{\mu(X(x_0, I, \rho r, r))}{r}\frac{dr}{r}\\
		\lesssim \sum_{I\in\cG(x_0,k), I\subset J_Q}\E_1\HH(I)
		\le \E_1\HH(J_Q),
		\end{multline}
		where in the last inequality we used the fact that the intervals in $\cG(x_0,k)$ are disjoint.
		
		We are in a similar situation as in the proof of \lemref{lem:init-inter2}: we are going to apply \lemref{lem:ener-interior} to transfer this estimate to the conical energy associated with $0.9J_Q$. 
		
		We take $J=J_Q$, $\delta=\ell(Q)=\rho^k$, and we define $\cI$ as the family of maximal triadic intervals contained in $J_Q\setminus G(x_0,k)$. We need to check assumptions (b) and (c) of \lemref{lem:ener-interior}.
		
		The assumption (c) is straightforward to check. For $H\coloneqq\bigcup_{I\in\cI}I= J_Q\setminus G(x_0,k)$ we have
		\begin{equation*}
		\HH(H)= \HH(J_Q)-\HH(J_Q\cap G(x_0,k))\overset{\eqref{eq:xnotbad}}{\le}2\ve\HH(J_Q).
		\end{equation*}
		Recalling that $\ve=c_\ve M^{-1}A^{-1}$, we get \eqref{eq:smallH} assuming $c_\ve$ is chosen small enough. 
		
		Now we are going to check assumption (b): for every $I\in\cI$ there exists $\theta_I\in 15 I$ and $y_I\in B_I(x_0,C\delta)$ with $\mu^\perp_{\theta_I}(y_I)\le M$. Fix $I\in\cI$. By the maximality of $I$, we get that $I^1$ intersects $G(x_0,k)$, which implies that it contains some $J\in \cG(x_0,k)$. By the definition of $\cG(x_0,k)$, there exists $y\in E_0$ such that $J\in\cG_2(y)$ and $y\in B_{J}(x_0,10\rho^k)$. Furthermore, by \lemref{lem:find-bdd-dir}, there exists $\theta_J\in 3J$ such that $\mu_{\theta_J}^\perp(y)\le M$. Note that $\theta_J\in 15I$ and $y\in B_I(x_0, C\rho^k)$. Thus, setting $\theta_I\coloneqq\theta_J$ and $y_I\coloneqq y$ we get the desired property of $\cI$ for some $\alpha\sim 1$.
		
		We have checked all the assumptions of \lemref{lem:ener-interior}, and so we get that for all $r\ge \rho^k$
		\begin{equation*}
		\mu(X(x_0,0.9J_Q,r,2r))\lesssim \mu(X(x_0,J_Q\cap G(x_0,k),r/2,4r)).
		\end{equation*}
		Fix $r'\in [\rho^{k-1},1]$. Applying the estimate above with $r_j=\{2^j\rho r'\}_{j\in\mathbb{N}}$, and summing over $j\in\{0,1,\dots, \lceil-\log_2\rho\rceil\}$ yields
		\begin{equation*}
		\mu(X(x_0,0.9J_Q,\rho r',r'))\lesssim \mu(X(x_0,J_Q\cap G(x_0,k),\rho r'/2,4r')).
		\end{equation*}
		Integrating over $r'\in [2\rho^{k-1},1/4]$ yields
		\begin{multline}\label{eq:blaaa}
		\int_{2\rho^{k-1}}^{1/4}\frac{\mu(X(x_0,0.9J_Q,\rho r,r))}{r}\, \frac{dr}{r}\lesssim\int_{2\rho^{k-1}}^{1/4}\frac{\mu(X(x_0,J_Q\cap G(x_0,k),\rho r/2,4r))}{r}\, \frac{dr}{r}\\
		\lesssim\int_{\rho^{k-1}}^{1}\frac{\mu(X(x_0,J_Q\cap G(x_0,k),\rho r,r))}{r}\, \frac{dr}{r}\overset{\eqref{eq:suppl-est}}{\lesssim}\E_1\HH(J_Q).
		\end{multline}
		We also have
		\begin{multline}\label{eq:similar-stuff}
		\int_{\rho^k}^{2\rho^{k-1}}\frac{\mu(X(x_0,0.9J_Q,\rho r,r))}{r}\, \frac{dr}{r}+\int_{1/4}^{1}\frac{\mu(X(x_0,0.9J_Q,\rho r,r))}{r}\, \frac{dr}{r}\\
		\lesssim \frac{\mu(B_{J_Q}(x_0,C\rho^{k-1}))}{\rho^k} + \mu(B_{J_Q}(x_0,C))\le \frac{\mu(C'B_Q)}{\rho^k} + \mu(C'\rho^{-k}B_Q)\overset{\eqref{eq:meas-upper-bd}}{\lesssim}M\HH(J_Q).
		\end{multline}
		Since $M\le\E_1$, the estimate above together with \eqref{eq:blaaa} gives \eqref{eq:dess}.
		
		It remains to derive \eqref{eq:des1} from \eqref{eq:dess}. This follows immediately from estimate \eqref{eq:scales1} in \lemref{lem:scales}:
		\begin{equation*}
			\#\Bad(x_0,0.85J_Q,0,k)\lesssim A\HH(J_Q)^{-1}\int_{\rho^k}^{1}\frac{\mu(X(x, 0.9J_Q, \rho r, r))}{r}\frac{dr}{r} + 1\lesssim A\E_1.
		\end{equation*}
	\end{proof}
	Now we settle the second case.
	\begin{lemma}
		Assume that there exists $I\in\cG(x_0,k)$ such that $J_Q\subset I$. Then, \eqref{eq:des1} holds.
	\end{lemma}
	\begin{proof}
		By \eqref{eq:ener-Gyk} we have
		\begin{equation}\label{eq:ener-Gxk}
			\int_{\rho^k}^1 \frac{\mu(X(x_0, 1.5I, \rho r, r))}{r}\frac{dr}{r}\lesssim \E_1\HH(I).
		\end{equation}
		By the estimate \eqref{eq:scales1} in \lemref{lem:scales} this implies
		\begin{equation*}
			\#\Bad(x_0,I,0,k)\lesssim A\HH(I)^{-1}\int_{\rho^k}^{1}\frac{\mu(X(x, 1.5I, \rho r, r))}{r}\frac{dr}{r} + 1\lesssim A\E_1.
		\end{equation*}
		Since $J_Q\subset I$ we have $\Bad(x_0,J_Q,0,k)\subset \Bad(x_0,I,0,k)$, and so the estimate above implies \eqref{eq:des1}.
	\end{proof}
	Now that we have established \eqref{eq:des1} we are ready to finish the proof of \lemref{lem:ener-bdd-tree}.
	\begin{proof}[Proof of \lemref{lem:ener-bdd-tree}]
		Let $x\in Q$, and let $x_0\in Q$ be such that \eqref{eq:des1} holds. By \lemref{lem:cone-in-cone}
		\begin{equation}\label{eq:ciii}
		X(x,0.8J_Q,\rho r,r)\subset X(x_0,0.85J_Q,\rho^2 r, \rho^{-1}r)\quad\text{if $r>Cd_{J_Q}(x,x_0)$,}
		\end{equation}
		where $C\ge 1$ is an absolute constant. 
		
		Recall that $d_{J_Q}(x,x_0)\le \diam_{{J_Q}}(Q)\le 8\rho^k$ by \ref{list:Tbasic}. We claim that there is an absolute constant $C'\ge 1$ such that for every $j\in \Bad(x,0.8J_Q,0,k)$ with $C'\le j\le k-C'$ we have
		\begin{equation}\label{eq:scal}
			\{j-1, j, j+1\}\cap \Bad(x_0,0.85J_Q,0,k)\neq\varnothing.
		\end{equation}
		Indeed, for any such $j$ we have 
		\begin{equation*}
		\rho^j\ge \rho^{k-C'}\ge \frac{\rho^{-C'}}{8}d_{J_Q}(x,x_0)\ge Cd_{J_Q}(x,x_0)
		\end{equation*}
		assuming $C'=C'(\rho,C)$ is large enough. Thus, applying \eqref{eq:ciii} with $r=\rho^j$ yields
		\begin{equation*}
			\varnothing\neq E\cap X(x,0.8J_Q,\rho^{j+1},\rho^j)\subset E\cap X(x_0,0.85J_Q,\rho^{j+2}, \rho^{j-1}).
		\end{equation*}
		This implies \eqref{eq:scal}. Hence,
		\begin{equation*}
			\#\Bad(x,0.8J_Q,0,k)\lesssim \#\Bad(x_0,0.85J_Q,0,k) + 1\overset{\eqref{eq:des1}}{\lesssim} A\E_1.
		\end{equation*}
	\end{proof}
	This concludes the proof of \propref{prop:tree}.

%	\begin{lemma}\label{lem:int-contain2}
%		If $x\in Q\in\T$, $\ell(Q)=\rho^k$, and $J\in \cG(x,k)$ satisfies $J\cap J_Q\neq\varnothing$, then $J\subset J_Q$. 
%	\end{lemma}
%	\begin{proof}
%		Recall that both $J$ and $J_Q$ are triadic intervals. Since $J\cap J_Q\neq\varnothing$ we have either $J\subset J_Q$ or $J_Q\subsetneq J$. Suppose it is the latter.
%		
%		By \lemref{lem:coverG2}, there exists a unique $P\in\T_k$ such that $x\in P$ and $J\subset J_P$. Since $J_Q\subsetneq J$, we get that $J_Q\subsetneq J_P$. Taking into account that $x\in Q\cap P$, \ref{list:disj} implies $Q\prec P$. Then, by \ref{list:Torder} (applied with $k=l$, and using the uniqueness assertion) we get that $Q=P$. This contradicts $J_Q\subsetneq J_P$.
%	\end{proof}

	\section{From conical energies to $\Bad$ cubes}\label{sec:con-to-bad}
	Recall that our goal is proving the energy estimate from \propref{prop:E0energy-est}, that is,
	\begin{equation}\label{eq:E0energy-est2}
	\int_{E_0}\int_0^1 \frac{\mu(X(x,\wt{G_2}(x),r))}{r}\, \frac{dr}{r}d\mu(x)
	\lesssim (AM)^{11}\E_1\,\HH(J_0)\mu(E),
	\end{equation}
	where 
	\begin{equation*}
	\wt{G_2}(x)=\bigcup_{J\in\cG_2(x)}15J.
	\end{equation*}
	Observe that by \lemref{lem:coverG2} for any $x\in E_0$ and $k\ge 0$ we have
	\begin{equation*}
	\wt{G_2}(x)\subset \bigcup_{Q\in\T_k :\, x\in Q} 15J_Q.
	\end{equation*}	
	This leads us to the following definition of family $\Bad\subset\T$: $Q\in\Bad$ if there exists $x\in Q$ such that
	\begin{equation*}
		X(x,15 J_Q, \rho\,\ELL(Q), \ELL(Q))\cap E\neq\varnothing.
	\end{equation*}
	Later on we will show the following packing estimate for $\Bad$.
	\begin{prop}\label{prop:Bad-pack}
		We have
		\begin{equation*}
			\sum_{Q\in\Bad} \HH(J_Q)\mu(Q) \lesssim (AM)^{10}\E_2\, \HH(J_0)\mu(E).
		\end{equation*}
	\end{prop}
	Now we show how \propref{prop:Bad-pack} implies the energy estimate \eqref{eq:E0energy-est2}.

	\begin{proof}[Proof of \propref{prop:E0energy-est} using \propref{prop:Bad-pack}]
		Let $x\in E_0$. By the definition of $\wt{G_2}(x)$,
		\begin{multline}\label{eq:G11toG2}
		\int_0^1 \frac{\mu(X(x,\wt{G_2}(x),r))}{r}\, \frac{dr}{r} = \int_0^1 \frac{\mu(X(x,\bigcup_{J\in\cG_2(x)}15J,r))}{r}\, \frac{dr}{r}\\
		\le\sum_{J\in\cG_2(x)}\int_0^1 \frac{\mu(X(x,15J,r))}{r}\, \frac{dr}{r}.
		\end{multline}
		Now fix $J\in\cG_2(x)$. By \eqref{eq:standard2}
		\begin{equation}\label{eq:sum1}
		\int_0^1 \frac{\mu(X(x,15J,r))}{r}\, \frac{dr}{r}\lesssim \int_0^1 \frac{\mu(X(x,15J,\rho r, r))}{r}\, \frac{dr}{r}.
		\end{equation}
		Recall that
		\begin{equation*}
		\Bad(x,15J) = \{k\ge 0 : E\cap X(x, 15J, \rho^{k+1}, \rho^k)\neq\varnothing \}.
		\end{equation*}
		We are going to use estimate \eqref{eq:scales2} from \lemref{lem:scales}. By \lemref{lem:find-bdd-dir} the assumption involving the maximal function is satisfied with $y=x$, and so we may apply \eqref{eq:scales2} with arbitrarily large $j$. Letting $j\to\infty$ yields
		\begin{equation*}
		\int_0^1 \frac{\mu(X(x,15J,\rho r, r))}{r}\, \frac{dr}{r}\lesssim M\HH(J)\cdot\#\Bad(x,15J).
		\end{equation*}
		
		Let $k\in\Bad(x,15J)$. By \ref{list:coverG2}, there exists a unique $Q\in\T_k$ such that $x\in Q$ and $J\subset J_Q$. Since $15J\subset 15J_Q$ and $\ell(Q)=\rho^k$, it follows that
		\begin{equation*}
		E\cap X(x, 15J, \rho^{k+1}, \rho^k)\subset E\cap X(x, 15J_Q, \rho\ell(Q), \ell(Q)),
		\end{equation*}
		which gives $Q\in\Bad$. Thus, for every $k\in\Bad(x,J)$ there is a unique $Q\in \Bad\cap\T_k$ such that $x\in Q$ and $J\subset J_Q$. Hence,		
		\begin{equation*}
		\int_0^1 \frac{\mu(X(x,15J,\rho r, r))}{r}\, \frac{dr}{r} \lesssim M \HH(J)\cdot \#\{Q\in\Bad : x\in Q,\, J\subset J_Q\}.
		\end{equation*}
		Observe further that, by \lemref{lem:int-contain}, for any $Q\in \Bad$ such that $x\in Q$ we have either $J\cap J_Q=\varnothing$, or $J\subset J_Q$. Therefore,
		\begin{equation*}
		\HH(J)\cdot\#\{Q\in\Bad : x\in Q,\, J\subset J_Q\} = \sum_{Q\in\Bad:\, x\in Q} \HH(J\cap J_Q).
		\end{equation*}
		Putting this together with the previous estimate, summing over all $J\in\cG_2(x)$, and recalling that the intervals in $\cG_2(x)$ are pairwise disjoint yields
		\begin{multline*}
		\int_0^1 \frac{\mu(X(x,\wt{G_2}(x),r))}{r}\, \frac{dr}{r} \overset{\eqref{eq:G11toG2},\eqref{eq:sum1}}{\le} \sum_{J\in\cG_2(x)}\int_0^1 \frac{\mu(X(x,15J,\rho r, r))}{r}\, \frac{dr}{r}\\
		 \lesssim M\sum_{J\in\cG_2(x)} \sum_{Q\in\Bad:\, x\in Q} \HH(J\cap J_Q)=M \sum_{Q\in\Bad:\, x\in Q} \HH(G_2(x)\cap J_Q)\\
		\le  M \sum_{Q\in\Bad} \one_{Q}(x) \HH(J_Q).
		\end{multline*}
		Integrating over $x\in E_0$ and using \propref{prop:Bad-pack} we get
		\begin{multline*}
		\int_{E_0}\int_0^1 \frac{\mu(X(x,\wt{G_2}(x),r))}{r}\, \frac{dr}{r}\, d\mu(x)\lesssim M \sum_{Q\in\Bad} \HH(J_Q)\mu(Q\cap E_0)\\
		\lesssim A^{10}M^{11}\E_2\,\HH(J_0)\mu(E) = (AM)^{11}\E_1\,\HH(J_0)\mu(E)
		\end{multline*}
		and so \eqref{eq:E0energy-est2} holds.
	\end{proof}
	
	It remains to prove \propref{prop:Bad-pack}. 
 Recalling that the family $\Roots$ satisfies a packing condition \ref{list:shat}, to get \propref{prop:Bad-pack} it suffices to prove the following.
	\begin{lemma}\label{lem:Bad-pack}
		For every $R\in\Roots$ we have
		\begin{equation*}
			\sum_{Q\in \T(R)\cap \Bad}\mu(Q)\lesssim (AM)^9 \E_2\mu(R).
		\end{equation*}
	\end{lemma}
	\begin{proof}[Proof of \propref{prop:Bad-pack} using \lemref{lem:Bad-pack}]		
		Using the decomposition of $\T$ given by \ref{list:Tstruc}, as well as the packing estimates from \lemref{lem:Bad-pack} and \ref{list:shat} we get
		\begin{multline}
		\sum_{Q\in\Bad} \HH(J_Q)\mu(Q)= \sum_{R\in\Roots}\sum_{Q\in\Bad\cap\T(R)} \HH(J_Q)\mu(Q)\\
		= \sum_{R\in\Roots}\HH(J_R)\sum_{Q\in\Bad\cap\T(R)}\mu(Q)\lesssim (AM)^9 \E_2\sum_{R\in\Roots}\HH(J_R)\mu(R)\\
		\le \ve^{-1}(AM)^9 \E_2\,\HH(J_0)\mu(E).
		\end{multline}
		Recalling that $\ve=c_\ve A^{-1}M^{-1}$ finishes the proof.
	\end{proof}

	The rest of the article is dedicated to the proof of \lemref{lem:Bad-pack}.
	\subsection{Trees of David-Mattila cubes}
	In the proof of \lemref{lem:Bad-pack} we will need to use the small boundaries property from \lemref{lem:dyadic cubes}, which we cannot guarantee for our \textit{ad hoc} cubes comprising $\T$. For this reason to every $R\in\Roots$ we associate an auxiliary family of cubes $\Tree(R)$ which consists of David-Mattila cubes contained in a fixed dyadic lattice $\DD(E,J_R)$.
	
	\begin{definition}\label{def:tree(R)}
		Let $R\in\Roots$. We define $\Tree(R)\subset \DD(E,J_R)$ to be the family of cubes $Q\in \DD(E,J_R)$ for which there exists $P_Q\in\T(R)$ with $\ell(P_Q)=\ell(Q)$ and $P_Q\cap Q\neq\varnothing$.
	\end{definition}

	For the sake of brevity, for every $R\in\Roots$ we set
	\begin{equation*}
	d_R\coloneqq d_{J_R}.
	\end{equation*}
	
	\begin{remark}
		In the proof of Lemma \ref{lem:Bad-pack} we may assume without loss of generality that $\T(R)$ is a finite family of cubes, simply by truncating it at an arbitrarily small generation, and then proving the packing estimate which does not depend on the truncation. From now on we assume that is the case. This implies that $\Tree(R)$ is also a finite family of cubes.
	\end{remark}
	
	In the following lemma we show that $\Tree(R)$ consists of a bounded number of trees.
	\begin{lemma}\label{lem:prop-tree}
		The family $\Tree(R)$ has the following properties:
		\begin{enumerate}[label={\alph*})]
			\item For every $Q\in\Tree(R)$ we have $Q\subset B_Q\subset CB_R$ for some absolute constant $C\ge 1$.
			\item There exist pairwise disjoint cubes $R_1,\dots, R_k\in\Tree(R)$ with $\ell(R_i)=\ell(R)$, $R_i\cap R\neq\varnothing$, and such that for every $Q\in\Tree(R)$ we have $Q\subset R_i$ for some $1\le i\le k$. Moreover, $k\lesssim 1$.
			\item If $Q,S\in\DD(E,J_R)$ satisfy $Q\subset S$, $Q\in\Tree(R)$, and $\ell(S)\le\ell(R)$, then we have $S\in\Tree(R)$.		
		\end{enumerate}  
	\end{lemma}
	\begin{proof}
		We start with a). Suppose that $Q\in\Tree(R)$, and let $P_Q\in\T(R)$ be such that $\ell(P_Q)=\ell(Q)$ and $P_Q\cap Q\neq\varnothing$. Recall that there exists $x_Q\in Q$ such that $Q\subset B_Q=B_{J_R}(x_Q, 50r(Q))$ with $r(Q)\sim \ell(Q)$, and also that there exists $x_R\in R$ such that $R\subset B_R=B_{J_R}(x_R,4\ell(R))$. Since $P_Q\subset R$ we have $Q\cap R\neq\varnothing$ and so
		\begin{equation*}
		d_{R}(x_Q,x_R)\le \diam_{d_R}(B_Q)+\diam_{d_R}(B_R)\lesssim \ell(Q)+\ell(R)\lesssim \ell(R).
		\end{equation*}
		Hence, $Q\subset B_Q\subset CB_R$.
		\vspace{0.5em}
		
		We move on to b). Let $\{R_i\}_{i\in \cI}\in\DD(E,J_R)$ be the family of all cubes satisfying $\ell(R_i)=\ell(R)$ and $R_i\cap R\neq\varnothing$. Clearly, all $R_i$ belong to $\Tree(R)$, so that $B_{R_i}\subset CB_R$ by a). The $R_i$'s are disjoint because they belong to the same generation of $\DD(E,J_R)$. 
		
		Using \lemref{lem:dyadic cubes} c) we get that the $d_R$-balls $5B(R_i)$ are disjoint, which together with the fact that $5B(R_i)\subset B_{R_i}\subset CB_R$ and $\ell(R_i)=\ell(R)$ gives that the number of $R_i$'s is bounded, i.e. $\#\cI\lesssim 1$.
		
		Let $Q\in\Tree(R)$, so that there exists $P_Q\in\T(R)$ with $\ell(P_Q)=\ell(Q)$ and $Q\cap P_Q\neq\varnothing$. Let $S\in\DD(E,J_R)$ be the unique cube such that $Q\subset S$ and $\ell(S)=\ell(R)$. Then, $S\cap P_Q\neq\varnothing$, and in particular $S\cap R\neq\varnothing$. This means that $S=R_i$ for some $i\in\cI$. This concludes the proof of b).
		\vspace{0.5em}		
		
		Regarding c), let $Q,S\in\DD(E,J_R)$ satisfy $Q\subset S$, $\ell(S)\le\ell(R)$ and $Q\in\Tree(R)$. Let $P_Q\in\T(R)$ be such that $Q\cap P_Q\neq\varnothing$ and $\ell(Q)=\ell(P_Q)$. Let $P_S\in\T$ be the unique $\T$-ancestor of $P_Q$ satisfying $\ell(P_S)=\ell(S)$. In particular, $J_R=J_{P_Q}\subset J_{P_S}$ and $P_Q\subset P_S$. 
		
		We claim that $P_S\in\T(R)$. Since $\ell(P_S)\le\ell(R)$, $P_S\cap P_R\neq\varnothing$, and $J_R\cap J_{P_S}\neq\varnothing$, it follows from \ref{list:disj} that $P_S\prec R$. Moreover, since $J_R\subset J_{P_S}$ we also have $J_{P_S}=J_R$. By \ref{list:Tstruc} this gives $P_S\in\T(R)$.		
		Together with the fact that
		\begin{equation*}
		\varnothing \neq Q\cap P_Q\subset S\cap P_S
		\end{equation*}
		and $\ell(S)=\ell(P_S)$, this gives $S\in\Tree(R)$.
	\end{proof}
	For any $Q\in\DD_k(E,J_R)$ with $k\ge 1$ denote by $Q^1$ the dyadic parent of $Q$. We define $\Stop(R)\subset\DD(E,J_R)$ as the family of cubes satisfying $Q^1\in\Tree(R)$ but $Q\notin\Tree(R)$. It is easy to see that the cubes in $\Stop(R)$ are pairwise disjoint, and that they are all contained in $CB_R$. Moreover, since $\Tree(R)$ is a finite family we have
	\begin{equation}\label{eq:stopR}
		\bigcup_{Q\in\Tree(R)}Q = \bigcup_{Q\in\Stop(R)}Q=\bigcup_{i=1}^k R_i,
	\end{equation}
	where $R_i$ are as in \lemref{lem:prop-tree} b). We set
	\begin{equation}\label{eq:wtR}
		\wt{R}\coloneqq \bigcup_{i=1}^k R_i.
	\end{equation}
	Note that
	\begin{equation}\label{eq:wtRest}
		\mu(\wt{R})=\sum_{Q\in\Stop(R)}\mu(Q)\le\mu(CB_R)\overset{\eqref{eq:doubling}}{\lesssim}AM\mu(R).
	\end{equation}
	
	We have the following counterpart of \lemref{lem:cube-meas}.
	\begin{lemma}\label{lem:cube-meas2}
		For every $Q\in\Tree(R)\cup\Stop(R)$ and $C\ge 1$ we have
		\begin{equation}\label{eq:meas-upper-bd2}
		\mu(CB_Q)\lesssim CM\HH(J_R)\ell(Q)
		\end{equation}
		and
		\begin{equation}\label{eq:meas-lower-bd2}
		\mu(Q)\gtrsim A^{-1}\,\HH(J_R)\ell(Q).
		\end{equation}
		In particular,
		\begin{equation}\label{eq:doubling2}
		\mu(CB_Q)\lesssim CAM\,\mu(Q).
		\end{equation}
	\end{lemma}	
	\begin{proof}
		Suppose that $Q\in\Tree(R)$ and let $P_Q\in\T(R)$ be the cube such that $\ell(Q)=\ell(P_Q)$ and $Q\cap P_Q\neq\varnothing$. Since $CB_Q\subset C'CB_{P_Q}$, the estimate \eqref{eq:meas-upper-bd2} follows from \eqref{eq:meas-upper-bd}. If $Q\in\Stop(R)$, then $Q^1\in\Tree(R)$ and then $CB_Q\subset CB_{Q^1}$ so that 
		\begin{equation*}
			\mu(CB_Q)\lesssim \mu(CB_{Q^1})\lesssim CM\HH(J_R)\ell(Q^1)\sim CM\HH(J_R)\ell(Q).
		\end{equation*}
		
		To get \eqref{eq:meas-lower-bd2}, recall that
		\begin{equation*}
		B_{J_R}(x_Q,\ell(Q))\cap E\subset Q
		\end{equation*}
		by \lemref{lem:dyadic cubes} c). Since $B(x_Q,c\HH(J_R)\ell(Q))\subset B_{J_R}(x_Q,\ell(Q))$, the estimate  \eqref{eq:meas-lower-bd2} follows from Ahlfors regularity of $E$.
	\end{proof}
	Let $\Bad(R)\subset\Tree(R)$ be the family of cubes $Q\in\Tree(R)$ for which there exists $x\in Q$ with
	\begin{equation*}
	X(x,15J_R, \rho\ELL(Q), \ELL(Q))\cap E\neq\varnothing.
	\end{equation*}
	Recall that our goal is \lemref{lem:Bad-pack}. We reduce the packing estimate for $\Bad\cap\T(R)$ to a packing estimate for $\Bad(R)$.
	\begin{lemma}\label{lem:badQ1}
		For every $P\in\Bad\cap\T(R)$ there exists $Q\in\Bad(R)$ with $\ell(Q)=\ell(P)$ and $Q\cap P\neq\varnothing$.
	\end{lemma}
	\begin{proof}
		If $P\in\Bad\cap\T(R)$, then there exists $x\in P$ with $X(x,15J_R, \rho\ELL(P), \ELL(P))\cap E\neq\varnothing$. Observe that the unique cube $Q\in\DD(E,J_R)$ with $\ell(Q)=\ell(P)$ and $x\in Q$ satisfies $Q\in\Tree(R)$, and so $Q\in\Bad(R)$.
	\end{proof}
	Given $Q\in\Bad(R)$ set
	\begin{equation*}
	\Bad_Q\coloneqq \{P\in\Bad\cap\T(R)\,:\,\ell(P)=\ell(Q),\, Q\cap P\neq\varnothing\}.
	\end{equation*}
	Our previous lemma states that for every $P\in\Bad\cap\T(R)$ there exists $Q\in\Bad(R)$ with $P\in\Bad_Q$. 
%	\begin{lemma}\label{lem:badQ2}
%		For every $Q\in\Bad(R)$ we have
%		\begin{equation}\label{eq:badQ2}
%		\#\Bad_Q\lesssim AM.
%		\end{equation}
%	\end{lemma}
%	\begin{proof}
%		If $\Bad_Q=\varnothing$ there is nothing to do. Fix some $P\in\Bad_Q$. It follows from the triangle inequality that all $S\in\Bad_Q$ satisfy
%		\begin{equation*}
%		B_S\subset \alpha B_P
%		\end{equation*}
%		for some absolute $\alpha\ge 1$. Since we have also $\ell(S)=\ell(P)$, it follows that $\Bad_Q\subset \Peer(P,\alpha)$, where $\Peer(P,\alpha)$ was defined in \eqref{eq:peer}. Thus, \lemref{lem:bdd-overlap} gives the desired estimate.
%	\end{proof}

	\begin{lemma}\label{lem:badtobadR}
		We have
		\begin{equation*}
		\sum_{P\in\Bad\cap\T(R)}\mu(P)\lesssim AM \sum_{Q\in\Bad(R)}\mu(Q).
		\end{equation*}
	\end{lemma}
	\begin{proof}
		By \lemref{lem:badQ1} for every $P\in\Bad\cap\T(R)$ there exists $Q\in\Bad(R)$ with $P\in\Bad_Q$. Since every $P\in\Bad_Q$ satisfies $P\subset CB_Q$, and $\Bad_Q$ is a family of disjoint cubes, we have
		\begin{equation*}
		\sum_{P\in\Bad_Q}\mu(P)\le\mu(CB_Q)\overset{\eqref{eq:doubling2}}{\lesssim}AM\mu(Q).
		\end{equation*}
		Thus,
		\begin{equation*}
		\sum_{P\in\Bad\cap\T(R)}\mu(P)\le \sum_{Q\in\Bad(R)}\sum_{P\in\Bad_Q}\mu(P)\lesssim AM\sum_{Q\in\Bad(R)}\mu(Q)
		\end{equation*}
	\end{proof}
	In the light of \lemref{lem:badtobadR}, to prove \lemref{lem:Bad-pack} it suffices to prove the packing estimate for $\Bad(R)$.
	\begin{lemma}\label{lem:BADDD}
		For every $R\in\Roots$ we have
		\begin{equation}\label{eq:BADDD}
		\sum_{Q\in \Bad(R)}\mu(Q)\lesssim (AM)^8 \E_2\mu(R).
		\end{equation}
	\end{lemma}

	We divide $\Bad(R)$ into two subfamilies: ``interior bad cubes'' $\Bi(R)$ and ``exterior bad cubes'' $\Be(R)$. 	
	We write 
	$Q\in\mathsf{B}_{int}(R)$ if there exists $x\in Q$ such that
	\begin{equation*}
		X(x,0.5J_R, \rho\ELL(Q), \ELL(Q))\cap E\neq\varnothing,
	\end{equation*}
	and $Q\in \Be(R)$ if there exists $x\in Q$ such that
	\begin{equation*}
		X(x,15J_R\setminus 0.5J_R, \rho\ELL(Q), \ELL(Q))\cap E\neq\varnothing.
	\end{equation*}
	
	We prove the packing estimate for $\Bi(R)$ in \lemref{lem:Bi-pack}, and for $\Be(R)$ in \lemref{lem:exterior-est}. Together they will imply \eqref{eq:BADDD}.

%	We further divide $\Bad(R)$ into two subfamilies: $Q\in\mathsf{B}_{int}(R)$ if there exists $x\in Q$ such that
%	\begin{equation*}
%		X(x,0.5 J_Q, A\ELL(Q), A^2\ELL(Q))\cap E\neq\varnothing,
%	\end{equation*}
%	and $\Be(R) \coloneqq \Bad(R)\setminus \Bi(R)$. In particular, for every $Q\in\Be(R)$ there exists $x\in Q$ such that
%	\begin{equation*}
%		X(x,3J_Q\setminus 0.5 J_Q, A\ELL(Q), A^2\ELL(Q))\cap E\neq\varnothing,
%	\end{equation*}
	\section{Interior bad cubes pack}\label{sec:inter}
	In this section we prove the packing estimate for $\Bi(R)$.
	\begin{lemma}\label{lem:Bi-pack}
		For every $R\in\Roots$ we have
		\begin{equation}\label{eq:bad-est}
		\sum_{Q\in \Bi(R)}\mu(Q)\lesssim \E_3 \mu(R),
		\end{equation}
		where $\E_3=A^2M^2\E_2$.
	\end{lemma}

%
%		\begin{lemma}\label{lem:Bi-pack}
%		The family $\Bi(R)$ has bounded intersection:
%		\begin{equation}\label{eq:bddinter}
%			\sum_{Q\in \Bi(R)}\one_Q \lesssim \E_2.
%		\end{equation}
%		In consequence,
%		\begin{equation*}
%			\sum_{Q\in \Bi(R)}\mu(Q) \lesssim \E_2\mu(R).
%		\end{equation*}
%	\end{lemma}
%	\begin{proof}
%		Suppose \eqref{eq:bddinter} fails, so that there is $x\in R$ such that $\sum_{Q\in \Bi(R)}\one_Q(x)\ge C\E_2$ for some large $C>1$.
%		%		Let $Q\in\Bi(R)$ be the minimal  cube containing $x$. Then, we have for all $y\in Q$ 
%		%		\begin{equation*}
%			%		\sum_{P\in \Bi(R)}\one_P(y)\ge C(M)\E_0
%			%		\end{equation*}
%		%		Fix $y\in Q$. 
%		Set 
%		\begin{equation*}
%			S\coloneqq \{j\ge 0\, :\, x\in Q\in\Bi(R)\cap\T_j\},
%		\end{equation*}
%		so that for 
%		By \lemref{lem:Bint-bad}, there exist $\ge C(M)\E_0$ cubes $\{P_i\}\subset\Bi(R)$ such that $y\in P_i$ and
%		\begin{equation*}
%			X(y,0.6 J_R, A\ELL(P_i)/2, 2A^2\ELL(P_i))\cap E\neq\varnothing.
%		\end{equation*}
%		In consequence,
%		\begin{equation*}
%			\int_{A\ELL(Q)}^{A\ELL(R)}\frac{\mu(X(y,0.9J_R,r,2r))}{r}\, \frac{dr}{r}\gtrsim C(M)\HH(J_R)\E_0.
%		\end{equation*}
%		This estimate holds for all $y\in Q$, but this leads to contradiction with \lemref{lem:ener-bdd-tree}.
%	\end{proof}
%	\begin{proof}[Proof of \lemref{lem:Bi-pack}]
%		This follows from \lemref{lem:Bi-Bi*-Stop}, \lemref{lem:Bi*-pack}, \lemref{lem:stop-pack}, and doubling.
%	\end{proof}

		The estimate \eqref{eq:bad-est} is a consequence of property \ref{list:ener-bdd-tree} of $\T$. However, in order to use \ref{list:ener-bdd-tree} it would be convenient if the following were true: if $Q\in \Bi(R)$, then \emph{for all} $x\in Q$
		\begin{equation*}
		X(x,0.5J_R, \rho\ELL(Q), \ELL(Q))\cap E\neq\varnothing.
		\end{equation*}
		Unfortunately, the definition of $\Bi(R)$ only asserts that there \emph{exists} some $x\in Q$ for which this holds. To overcome this minor issue, we define an auxiliary family of bad cubes. 
		
		\begin{definition}
			Let $\gamma=\rho^{-k_\gamma}>\rho^{-2}$ be a large absolute constant to be chosen in \lemref{lem:Bint-bad}. Here $k_\gamma\in\mathbb{N}$.
			
			We define $\Bi^*(R)\subset\Tree(R)$ as the family of  
			$Q\in\Tree(R)$ such that there exists $x\in B_Q\cap E$ with
			\begin{equation}\label{eq:bad-cone}
			X(x,0.5J_R, \gamma\ELL(Q), \gamma^2\ELL(Q))\cap E\neq\varnothing.
			\end{equation}
		\end{definition}
			
		First, we show that every $Q\in\Bi(R)$ contains a relatively large sub-cube from $\Bi^*(R)\cup\Stop(R)$.
		\begin{lemma}\label{lem:Bi-Bi*-Stop}
			Suppose that $Q\in\Bi(R)$. Then, there exists $P\in\Bi^*(R)\cup\Stop(R)$ with $P\subset Q$ and $\ELL(P)\ge \gamma^{-2}\ELL(Q)$.
		\end{lemma}
		\begin{proof}
			Let $\ell(Q)=\rho^k$. If $Q$ contains $P\in\Stop(R)$ with $\ELL(P)\ge \gamma^{-2}\ELL(Q)$ then we are done. Assume it does not. 
			
			It follows that for all $m\ge k$ with $\rho^m\ge \gamma^{-2}\rho^k$ the cube $Q$ can be partitioned into cubes from $\Tree(R)\cap\DD_m(E,J_R)$.
			Let $m\ge k$ be the smallest integer satisfying $\gamma\rho^m\le \rho^{k+1}$. Then, $\gamma\rho^{m-1}> \rho^{k+1}$ and so $\gamma^2\rho^m\ge \rho^k$ because $\gamma\ge \rho^{-2}$.

			Since $Q\in\mathsf{B}_{int}(R)$, there exists $x\in Q$ such that $X(x,0.5J_R, \rho^{k+1}, \rho^k)\cap E\neq\varnothing$.  Let $P\in\Tree(R)\cap\DD_m(E,J_R)$ be the unique cube containing $x$. Since $\gamma\rho^m\le \rho^{k+1}$ and  $\gamma^2\rho^m\ge \rho^k$ we have
			\begin{equation*}
			X(x,0.5J_R, \gamma\rho^m, \gamma^2\rho^m)\cap E\subset X(x,0.5J_R, \rho^{k+1}, \rho^k)\cap E\neq\varnothing,
			\end{equation*}
			and so $P\in\Bi^*(R)$.
		\end{proof}
		Our strategy is the following: we are going to show that the cubes in $\Bi^*(R)$ pack. Together with the packing of $\Stop(R)$, the lemma above, and the doubling property of $\mu$ on $\Tree(R)$ \eqref{eq:doubling2}, this will imply the packing estimate for $\Bi(R)$.
		\vspace{0.5em}

	First, we show that the existence of a single point $x\in B_Q\cap E$ satisfying \eqref{eq:bad-cone} implies a similar property for all $y\in B_Q$ (with a slightly larger cone).
	\begin{lemma}\label{lem:Bint-bad}
		If $\gamma$ is chosen large enough, then for every $Q\in\Bi^*(R)$ and $y\in B_Q$ we have
		\begin{equation}\label{eq:cone-inters}
		X(y,0.6 J_R, \ELL(Q), \gamma^3\ELL(Q))\cap E\neq\varnothing.
		\end{equation}
	\end{lemma}
	\begin{proof}
		Let $x\in B_Q\cap E$ be such that \eqref{eq:bad-cone} holds. We claim that for every $y\in B_Q$
		\begin{equation}\label{eq:coneincone}
		X(x,0.5J_R, \gamma\ELL(Q), \gamma^2\ELL(Q))\subset X(y,0.6 J_R, \ELL(Q), \gamma^3\ELL(Q)),
		\end{equation}
		assuming $\gamma$ is chosen large enough. Indeed, by \lemref{lem:cone-in-cone} we get that \eqref{eq:coneincone} holds as soon as $\gamma\ell(Q)\ge Cd_{R}(x,y)=Cd_{Q}(x,y)$, where $C\ge 4$ is absolute. Since $x,y\in B_Q$, we have
		\begin{equation*}
		Cd_{Q}(x,y)\le CC'\ell(Q)\le \gamma\ell(Q)
		\end{equation*}
		assuming $\gamma$ is large enough. This shows \eqref{eq:coneincone}, and then \eqref{eq:cone-inters} follows from \eqref{eq:bad-cone}.	
	\end{proof}

	\begin{lemma}\label{lem:Bi*-pack}
		The family $\Bi^*(R)$ has bounded intersection:
		\begin{equation}\label{eq:bddinter}
		\sum_{Q\in \Bi^*(R)}\one_Q \lesssim_{\gamma} \E_2.
		\end{equation}
		In consequence,
		\begin{equation}\label{eq:Bipack}
			\sum_{Q\in \Bi^*(R)}\mu(Q) \lesssim_{\gamma} \E_2\mu(CB_R).
		\end{equation}
	\end{lemma}
	\begin{proof}
		Suppose \eqref{eq:bddinter} fails, so that there is a decreasing sequence of cubes
		\begin{equation*}
		Q_1\supset Q_2\supset\dots \supset Q_N\quad \text{with}\quad \{Q_i\}_{i=1}^N\subset \Bi^*(R)		\end{equation*}
		with $N\ge C_0\E_2$ and $C_0=C_0(\gamma)\ge 100$ very large. For every $1\le i\le N$ let $\ell(Q_i)=\rho^{k_i}$.

		By \lemref{lem:Bint-bad} for all $x\in Q_i$ we have
		\begin{equation}\label{eq:skhs}
		X(x,0.6 J_{R}, \rho^{k_i}, \gamma^3\rho^{k_i})\cap E\neq\varnothing.
		\end{equation}	
		Since $\{Q_i\}$ is a nested family of cubes, we get that for every $x\in Q_N$ \eqref{eq:skhs} holds . 
		
		Let $P_{N}\in\T(R)$ be a cube such that $\ell(P_N)=\ell(Q_N)=\rho^{k_N}$ and $Q_N\cap P_N\neq\varnothing$ (such cube exists by the definition of $\Tree(R)$, Definition \ref{def:tree(R)}). By \ref{list:ener-bdd-tree} applied to $P_N$ we have
		\begin{equation}\label{eq:fkjd}
			\Bad(x,0.8J_R,0, k_N)\lesssim \E_2\quad\text{for all $x\in P_N$.}
		\end{equation}
		Fix $x_0\in Q_N\cap P_N$. Observe that both \eqref{eq:fkjd} and \eqref{eq:skhs} (for all $i\in\{1,\dots,N\}$) hold for $x_0$.
		
		We claim that \eqref{eq:skhs} implies that for every $3k_\gamma\le i\le N$ we have
		\begin{equation}\label{eq:scalesint}
			[k_i-3k_\gamma, k_i)\cap \Bad(x_0,0.8J_R,0, k_N) \neq\varnothing,
		\end{equation}
		where $\gamma=\rho^{-k_\gamma}$. Indeed, by \eqref{eq:skhs} we have 
		\begin{equation*}
			X(x,0.6 J_{R}, \rho^{j+1}, \rho^j)\cap E\neq\varnothing
		\end{equation*}
		for some $j$ with $\rho^{k_i}\le \rho^{j+1} \le \rho^{j}\le \gamma^3\rho^{k_i} = \rho^{k_i-3k_\gamma}$. In particular, $k_i-3k_\gamma\le j< k_i$. Assuming $3k_\gamma\le i\le N$ we also get that $0\le j\le k_N$. This gives \eqref{eq:scalesint}. 
		
		Observe that the intervals $[k_i-3k_\gamma, k_i)$ can overlap at most $\le 3k_\gamma$ many times. Thus, \eqref{eq:scalesint} implies
		\begin{equation*}
			\#\Bad(x,0.8J_R,0, k_N)\ge \frac{N-3k_\gamma}{3k_\gamma}\gtrsim N\ge C_0\E_2.
		\end{equation*}
		Comparing this to \eqref{eq:fkjd} and assuming that $C_0$ is large enough leads to contradiction. This concludes the proof of \eqref{eq:bddinter}.
		 
	Recalling that all the cubes in $\Tree(R)$ are contained in $CB_R$, and integrating \eqref{eq:bddinter} over $CB_R$ gives \eqref{eq:Bipack}.
	\end{proof}
	\begin{proof}[Proof of \lemref{lem:Bi-pack}]
		For every $Q\in\Bi(R)$ let $S_Q\in\Bi^*(R)\cup\Stop(R)$ be the cube from \lemref{lem:Bi-Bi*-Stop}, so that $S_Q\subset Q$ and $\ell(S_Q)\ge \gamma^{-2}\ell(Q)$. 
		
		It follows from the doubling property \eqref{eq:doubling2} that
		\begin{equation}\label{eq:doubl}
			\mu(Q)\lesssim \gamma^2 A M \mu(S_Q).
		\end{equation}
		For a fixed $S\in\Bi^*(R)\cup\Stop(R)$ we define
		\begin{equation*}
			\cQ(S) = \{Q\in\Bi(R) : S_Q=S\}.
		\end{equation*}
		We claim that $\#\cQ(S)\lesssim_{\gamma} 1$. 
		
		First, note that every $Q\in\cQ(S)$ satisfies $S\subset Q$ and $\ell(S)\le \ell(Q)\le \gamma^2\ell(S)$. By \lemref{lem:dyadic cubes} c) we know that for a fixed generation $k$ the balls
		\begin{equation*}
			\{5B(Q) : Q\in\cQ(S)\cap \DD_k(E,J_R)\}
		\end{equation*}
		are disjoint. Since they are all contained in $C(\gamma)B_S$ and satisfy $r(Q)\sim_{\gamma}\ell(S)$, we get that there are at most $\lesssim_{\gamma} 1$ many cubes in $\cQ(S)\cap \DD_k(E,J_R)$. At the same time, there are $\lesssim_{\gamma} 1$ integers $k$ such that $\cQ(S)\cap \DD_k(E,J_R)$ is non-empty. Thus,
		\begin{equation}\label{eq:overlaps}
			\#\cQ(S)\lesssim_{\gamma} 1.
		\end{equation}
		Using the packing estimate from \lemref{lem:Bi*-pack} and the fact that the cubes in $\Stop(R)$ are disjoint allows us to estimate as follows:
		\begin{multline*}
			\sum_{Q\in\Bi(R)}\mu(Q) = \sum_{S\in\Bi^*(R)\cup\Stop(R)}\sum_{Q\in\cQ(S)}\mu(Q)\\
			\overset{\eqref{eq:doubl}}{\lesssim_{\gamma}} A M\sum_{S\in\Bi^*(R)\cup\Stop(R)}\sum_{Q\in\cQ(S)} \mu(S)\overset{\eqref{eq:overlaps}}{\lesssim_{\gamma}} A M\sum_{S\in\Bi^*(R)\cup\Stop(R)}\mu(S)\\
			\overset{\eqref{eq:Bipack},\eqref{eq:stopR}}{\lesssim_{\gamma}} A M\big(\E_2\mu(CB_R)+\mu(\wt{R})\big)\overset{\eqref{eq:doubling},\eqref{eq:wtRest}}{\lesssim}(AM)^2\E_2\mu(R).
		\end{multline*}
		This finishes the proof of \lemref{lem:Bi-pack}.
	\end{proof}
%	\begin{lemma}
%		For every $Q\in\Tree(R)$ we have
%		\begin{equation*}
%			\int_{Q} \frac{\mu(X(x,0.9J_Q, A\ELL(Q)/2, 2A\ELL(R)))}{den}
%		\end{equation*}
%	\end{lemma}

	\section{Exterior bad cubes pack}
	Recall that $\Be(R)$ consists of cubes $Q\in\Tree(R)$ such that
	\begin{equation*}
	X(x,15J_R\setminus 0.5J_R, \rho\ELL(Q), \ELL(Q))\cap E\neq\varnothing.
	\end{equation*}
	In this section we prove a packing estimate for $\Be(R)$.
	\begin{lemma}\label{lem:exterior-est}
		For every $R\in\Roots$ we have
		\begin{equation*}
			\sum_{Q\in \Be(R)}\mu(Q)\lesssim (AM)^6\E_3\,\mu(R).
		\end{equation*}
	\end{lemma}
	
	\subsection{Strategy of the proof}
	We are going to use the key geometric lemma stated below, which is a refinement of \cite[Lemma 8.4]{dabrowski2022quantitative}. First we introduce some notation.
	
	\begin{definition}
		Given an open set $U\subsetneq\R$, the \emph{Whitney decomposition} of $U$ is a partition of $U$ into dyadic intervals with the property $3I\subset U$ and $3I^1\not\subset U$, where $I^1$ is the dyadic parent of $I$.
	\end{definition}
	It is easy to see that for any open $U\subsetneq\R$ a Whitney decomposition exists: simply take the maximal dyadic intervals with the property $3I\subset U$.
	
	Whitney decompositions have many nice properties (e.g. bounded overlaps between enlarged intervals) that have been used in analysis for decades, see e.g. Appendix J in \cite{grafakos2014classical} or Chapter VI.1 in \cite{stein1970singular}. In our application we won't need these more sophisticated properties. All we will use is that Whitney decomposition of $U$ forms a partition of $U$, and that if $I\subset U$ is an interval with $CI\not\subset U$, then there exists a Whitney interval $J$ such that $\HH(I)\sim\HH(J)$ and $I\subset C_0 CJ$ where $C_0\sim 1$ is absolute
	
		Recall that $\pi_{J_R} = \pi_\theta$ where $\theta$ is the mid-point of $J_R$. For brevity, we set
	\begin{equation*}
		\pi_R\coloneqq\pi_{J_R}.
	\end{equation*}	
		Given $P\in\Tree(R)$, denote by $\cW(P)$ the Whitney decomposition of
		\begin{equation*}
			U(P)\coloneqq \R\setminus \pi_{R}^\perp(\overline{P}).
		\end{equation*}
		Note that $U(P)$ is an open set, so the Whitney decomposition is well-defined.
		
		For $r>0$ we set
		\begin{equation*}
			\cW(P,r) = \{I\in\cW(P) : C^{-1}\lambda\HH(J_R)r\le \HH(I)\le C\lambda\HH(J_R)r\}
		\end{equation*}
	where $C>2$ is an absolute constant to be chosen later on, and $\lambda=c_\lambda M^{-1}A^{-1}$ is the constant from \lemref{lem:key geometric lemma} below.

	For any $Q\in\Tree(R)$ set 
	\begin{equation*}
		I_Q\coloneqq \pi_R^\perp(B_Q)\subset\R.
	\end{equation*}
	Note that $I_Q$ is an interval of length
	\begin{equation*}
	\HH(I_Q)\sim \HH(J_R)\ell(Q).
	\end{equation*}
	
	The following is our key geometric lemma.
	\begin{lemma}\label{lem:key geometric lemma}
		There exist small absolute constants $0<c_J<1$, $0<c_\lambda<1$ and a large absolute constant $\rho^{-1}<\Lambda\ll c_\lambda^{-1}$ such that the following holds. 
		
		Let $\lambda=c_\lambda M^{-1}A^{-1}$. 
		 Assume that
		\begin{enumerate}[label=({\roman*})]
			\item $R\in\Roots$ and $P,Q\in\Tree(R)$ satisfy $Q\subset P$ and
			\begin{equation}\label{eq:empty-cone}
			X(z,0.5J_R,\lambda \ell(Q), \Lambda \ell(P))\cap P=\varnothing\quad\text{for all $z\in \Lambda B_Q\cap E$.}
			\end{equation}
			\item there exists $x\in Q$ such that
			\begin{equation*}
			X(x,15 J_R\setminus 0.5J_R,\rho\ell(Q), \ell(Q))\cap E\neq\varnothing.
			\end{equation*}
		\end{enumerate}
		Then, there exists a Whitney interval $I\in \cW(P,\ell(Q))$ such that $I_Q\subset \lambda^{-1}\Lambda I$.
	\end{lemma}
	
	The proof of \lemref{lem:key geometric lemma} is deferred to Section \ref{sec:KGL}.
	\vspace{0.5em}
	
	We outline the proof of \lemref{lem:exterior-est}. We would like to apply \lemref{lem:key geometric lemma}. Recall that for every $Q\in\Tree(R)\setminus\Bi(R)$ we have
	\begin{equation*}
	X(z,0.5 J_R, \rho\ELL(Q), \ELL(Q))\cap E=\varnothing \quad\text{for all $z\in Q$.}
	\end{equation*}
	This is similar to the assumption \eqref{eq:empty-cone} except the truncation parameters don't match, and we would like this to hold for all $z\in \Lambda B_Q\cap E$ and not just for $z\in Q$. 
	
	To overcome these issues we will use the packing estimate for $\Bi(R)$ established in \lemref{lem:Bi-pack} to get a partition
	\begin{equation*}
	\Tree(R)=\Evil(R)\cup\bigcup_{P\in\Top(R)}\Br(P)
	\end{equation*}
	such that each $\Br(P)$ is a sub-tree with root $P$, for each $Q\in\Br(P)$ we will have \eqref{eq:empty-cone}, and the families $\Top(R)$ and $\Evil(R)$ pack:
	\begin{equation*}
	\sum_{P\in\Top(R)\cup\Evil(R)}\mu(P) \lesssim (AM)^4\E_3\,\mu(R).
	\end{equation*}
	Then, we will apply \lemref{lem:key geometric lemma} inside each $\Br(P)$, and with a little more work this will give
	\begin{equation*}
	\sum_{Q\in\Be(R)\cap\Br(P)}\mu(Q)\lesssim (AM)^2 \mu(P).
	\end{equation*}
	Together with the packing estimates for $\Top(R)$ and $\Evil(R)$, this will establish \lemref{lem:exterior-est}.

	\subsection{Evil cubes}
	We begin by defining a new family of bad cubes called $\Evil(R)$. 
	
	For every $S\in\Tree(R)\cup\Stop(R)$ let
	\begin{align*}
	\Peer(S)&\coloneqq \{Q\in\Tree(R)\, :\, \Lambda^2 B_S\cap \Lambda^2 B_Q\neq\varnothing,\, \lambda \ELL(Q)\le \ELL(S)\le \lambda^{-1} \ELL(Q)\},\\
	\Near(S)&\coloneqq \{Q\in\Tree(R)\, :\, \Lambda^2 B_Q\cap S\neq\varnothing,\, \ELL(Q)\le \ELL(S),\, S\cap Q=\varnothing\}.
	\end{align*}
	The constants $\Lambda>1$ and $\lambda=c_\lambda M^{-1}A^{-1}$ are the same as in \lemref{lem:key geometric lemma}.
	\begin{remark}
		Without loss of generality we may assume that $\lambda=\rho^{k_\lambda}$ and $\Lambda = \rho^{-k_\Lambda}$ for some $k_\lambda,k_{\Lambda}\in\mathbb{N}$. To do that, we slightly decrease $\lambda$ and increase $\Lambda$. It is clear that \lemref{lem:key geometric lemma} still holds with these modified parameters.
	\end{remark}
	
	We define
	\begin{equation*}
		{\Evil}(R) \coloneqq \{Q\in\Tree(R): \Lambda\ell(Q)>\ell(R)\}\cup \bigcup_{P\in\Bi(R)\cup\Stop(R)}\Peer(P)\cup\Near(P).
	\end{equation*}
	Our next objective is a packing estimate for $\Evil(R)$.

	\begin{lemma}
		For any $P\in\Tree(R)\cup\Stop(R)$ we have
		\begin{equation}\label{eq:peer-est}
		\sum_{Q\in\Peer(P)}\mu(Q)\lesssim (AM)^4\,\mu(P).
		\end{equation}
	\end{lemma}
	\begin{proof}
		Let $k\ge 0$ be such that $\ell(P)=\rho^{k}.$ Let $j\in\Z$ be such that $\lambda \rho^j\le \rho^k\le \lambda^{-1} \rho^j$. Recalling that $\lambda^{-1}\ge c_\lambda^{-1}\gg\Lambda$, we get that all $Q\in\Peer(P)\cap\DD_j(E,J_R)$ satisfy $Q\subset \lambda^{-2}B_P$. Hence,
		\begin{equation*}
			\sum_{Q\in\Peer(P)\cap\DD_j(E,J_R)}\mu(Q)\le \mu(\lambda^{-2}B_P)\overset{\eqref{eq:doubling2}}{\lesssim}AM\lambda^{-2}\mu(P).
		\end{equation*}
		Taking into account that the sum above is non-empty for $\lesssim |\log_\rho \lambda|\lesssim AM$ many integers $j$, and that $\lambda\sim A^{-1}M^{-1}$, we get the desired estimate.
	\end{proof}
	The reason we needed the small boundary property of David-Mattila cubes \eqref{eq:doubl-DM} is to get the following estimate. 
	\begin{lemma}
		For any $P\in\Tree(R)\cup\Stop(R)$ we have
		\begin{equation}\label{eq:near-est}
		\sum_{Q\in\Near(P)}\mu(Q)\lesssim AM\mu(P).
		\end{equation}
	\end{lemma}
	\begin{proof}
		Let $k\ge 0$ be such that $P\in\DD_k(E,J_R)$. Let $Q\in\Near(P)$. Since $\ell(Q)\le\ell(P)=\rho^k$, $\Lambda^2B_Q\cap P\neq\varnothing$ and $Q\cap P=\varnothing$, it follows that for every $x\in Q$
		\begin{equation*}
			\dist_{J_R}(x,P)\lesssim \Lambda^2\ell(Q).
		\end{equation*}
		Hence, if $l\in\mathbb{N}$ and $\ell(Q)\le c\Lambda^{-2}\rho^{k+l}$, then $Q\subset N_l^{ext}(P)$, where $0<c<1$ is an absolute constant and $N_l^{ext}(P)$ is defined above \eqref{eq:doubl-DM}.
		
		For $l\ge 0$ let
		\begin{equation*}
			\Near_l(P)=\{Q\in\Near(P)\, :\, c\Lambda^{-2}\rho^{k+l+1}< \ell(Q)\le c\Lambda^{-2}\rho^{k+l}\}.
		\end{equation*}
		Note that $\{\Near_l(P)\}_{l\ge 0}$ are disjoint families of cubes, each $\Near_l(P)$ consists of pairwise disjoint cubes, and for any $Q\in\Near_l(P)$ we have $Q\subset N_l^{ext}(P)$. Moreover, if $Q\in \Near_{big}(P)\coloneqq \Near(P)\setminus \bigcup_{l\ge 0}\Near_l(P)$, then $c\Lambda^{-2}\ell(P)\le\ell(Q)\le \ell(P)$. Thus,
		\begin{equation*}
			\sum_{Q\in \Near_{big}(P)}\mu(Q)\lesssim_{\Lambda}\mu(\Lambda^3B_P)\overset{\eqref{eq:doubling2}}{\lesssim_\Lambda}AM\mu(P).
		\end{equation*}
		We estimate
		\begin{multline*}
			\sum_{Q\in\Near(P)}\mu(Q) = \sum_{l\ge 0}\sum_{Q\in\Near_l(P)}\mu(Q) + \sum_{Q\in \Near_{big}(P)}\mu(Q)\\
			\lesssim_\Lambda \sum_{l\ge 0}\mu(N_l^{ext}(P)) + AM\mu(P)\overset{\eqref{eq:doubl-DM}}{\lesssim}\sum_{l\ge 0}\rho^{l/2}\mu(3B_P) + AM\mu(P)\lesssim AM\mu(P).
		\end{multline*}
	\end{proof}
	We are ready to prove the packing estimate for $\Evil(R)$.
	\begin{lemma}
		We have
		\begin{equation}\label{eq:evil-est}
			\sum_{Q\in\Evil(R)}\mu(Q)\lesssim \E_4\,\mu(R),
		\end{equation}
		where $\E_4\coloneqq (AM)^4\E_3$.
	\end{lemma}
	\begin{proof}
		We estimate
		\begin{multline*}
			\sum_{Q\in\Evil(R)}\mu(Q) \le \sum_{Q\in\Tree(R): \Lambda\ell(Q)>\ell(R)}\mu(Q)+  \sum_{P\in\Bi(R)}\sum_{Q\in\Peer(P)\cup\Near(P)}\mu(Q)\\
			\qquad\qquad\qquad\qquad\qquad\qquad\qquad\qquad\qquad + \sum_{P\in\Stop(R)}\sum_{Q\in\Peer(P)\cup\Near(P)}\mu(Q)\\
			\overset{\eqref{eq:peer-est},\eqref{eq:near-est}}{\lesssim}|\log_\rho\Lambda|\cdot\mu(\wt{R})+(AM)^4 \sum_{P\in\Bi(R)}\mu(P) + (AM)^4 \sum_{P\in\Stop(R)}\mu(P)\\
			\overset{\eqref{eq:bad-est},\eqref{eq:wtRest}}{\lesssim}(AM)^4\E_3\,\mu(R)+ (AM)^5\mu(R)\lesssim \E_4\,\mu(R).
		\end{multline*}
	\end{proof}

	Recall that by \eqref{eq:stopR}
	\begin{equation*}
		\wt{R}=\bigcup_{Q\in\Tree(R)}Q=\bigcup_{Q\in\Stop(R)}Q.
	\end{equation*}
	The following lemma will allow us to get \eqref{eq:empty-cone} inside $\Br(P)$. 
%	The definitions of $\Peer(S)$ and $\Near(S)$ were chosen precisely to work well with its proof.
	\begin{lemma}\label{lem:good-cone1}
		If $Q\in\Tree(R)\setminus\Evil(R)$ then
		\begin{equation*}
		X(z,0.5 J_R, \lambda\ELL(Q), \Lambda\ELL(Q))\cap E=\varnothing \quad\text{for all $z\in \Lambda^2 B_Q\cap \wt{R}$.}
		\end{equation*}
	\end{lemma}
	\begin{proof}
		Proof by contradiction. Suppose that there exists $z\in \Lambda^2 B_Q\cap \wt{R}$ such that 
		\begin{equation*}
		X(z,0.5 J_R, \lambda\ELL(Q), \Lambda\ELL(Q))\cap E\neq\varnothing.
		\end{equation*}
		Since $\lambda=\rho^{k_\lambda}$ and $\Lambda=\rho^{-k_\Lambda}$ for some integers $k_\lambda,k_\Lambda\ge 1$, there exists $k\ge 0$ such that
		\begin{equation}\label{eq:ellpellq}
		\lambda\rho^{-1}\ELL(Q)\le \rho^k\le \Lambda\ELL(Q)\le \ell(R)
		\end{equation}
		and
		\begin{equation}\label{Pbad}
		X(z,0.5 J_R, \rho^{k+1}, \rho^k)\cap E\neq\varnothing,
		\end{equation}
		where the last inequality in \eqref{eq:ellpellq} is derived from $Q\notin\Evil(R)$.
		
		Let $P\in\DD_k(E,J_R)$ be the unique cube with $z\in P$. We claim that $P\in\Tree(R)$.
		Indeed, if $P\notin\Tree(R)$, then there exists a cube $S\in\Stop(R)$ with $z\in S$ and $P\subset S$ (here we use that $z\in\wt{R}$, $\wt{R}$ is covered by $\Stop(R)$, and $\ell(P)\le\ell(R)$, recall \lemref{lem:prop-tree}). It follows that $z\in S\cap \Lambda^2 B_Q$, and $\ell(S)\ge \ell(P)=\rho^k\ge \lambda\rho^{-1}\ell(Q)$. There are two cases to consider:
		\begin{itemize}
			\item If $\ell(S)\le \ell(Q)$, then we have $Q\in\Peer(S)$, but this is a contradiction with $Q\notin\Evil(R)$.
			\item If $\ell(S)>\ell(Q)$, then clearly $S\not\subset Q$. At the same time $Q\not\subset S$ because $Q\in\Tree(R)$ and $S\in\Stop(R)$. Thus, $Q\cap S=\varnothing$. But this implies $Q\in\Near(S)$, which again contradicts $Q\notin\Evil(R)$.
		\end{itemize}
		Since we reached a contradiction in both cases, we get that $P\in\Tree(R)$. 
		
		Note that \eqref{Pbad} and $P\in\Tree(R)$ imply $P\in\Bi(R)$. Moreover, \eqref{eq:ellpellq} and $P\cap \Lambda^2 B_Q\neq\varnothing$ imply $Q\in\Peer(P)$. But this again means that $Q\in\Evil(R)$, which is a contradiction.
	\end{proof}
%	\begin{lemma}
%		If $Q\in\Tree(R)\setminus\Evil(R)$ then
%		\begin{equation}\label{eq:ddd}
%			X(z,0.5 J_R, \lambda\ELL(Q), \Lambda\ELL(Q))\cap \wt{R}=\varnothing \quad\text{for all $z\in \Lambda B_Q\cap E$.}
%		\end{equation}
%	\end{lemma}
%	\begin{proof}
%		Suppose that \eqref{eq:ddd} fails, so that there exist $z\in \Lambda B_Q\cap E$ and $x\in X(z,0.5 J_R, \lambda\ELL(Q), \Lambda\ELL(Q))\cap \wt{R}$.
%	\end{proof}
	\subsection{Dividing $\Tree(R)$ into sub-trees}
	Recall that for every $Q\in\DD_k(E,J_R), k\ge 1$ we write $Q^1$ to denote its dyadic parent: the unique cube $Q^1\in\DD_{k-1}(E,J_R)$ with $Q\subset Q^1$. 
	
	Set
	\begin{equation*}
	\Top(R)\coloneqq\{P\in\Tree(R)\setminus\Evil(R)\ :\ P^1\in\Evil(R)\}.
	\end{equation*}
%	where $R_i$ are as in \lemref{lem:prop-tree}.	
	We decompose $\Tree(R)\setminus\Evil(R)$ into sub-trees (or ``branches'') associated to cubes from $\Top(R)$. 
	\begin{definition}
		For every $P\in\Top(R)$ let $\Br(P)$ denote the cubes $Q\in\Tree(R)$ such that $Q\subset P$ and there is no $S\in\Evil(R)$ with $Q\subset S\subset P$.
	\end{definition}
	
	We have
	\begin{equation}\label{eq:branch-decomp}
		\Tree(R)\setminus\Evil(R)=\bigcup_{P\in\Top(R)}\Br(P),
	\end{equation}
	and the union above is disjoint. A packing estimate for $\Top(R)$ follows easily from the packing estimate for $\Evil(R)$.
	\begin{lemma}
		We have
		\begin{equation}\label{eq:top-est}
		\sum_{P\in\Top(R)}\mu(P)\lesssim \E_4\,\mu(R).
		\end{equation}
	\end{lemma}
	\begin{proof}
		Since any cube $P\in\DD(E,J_R)$ has only a bounded number of children (this follows from \lemref{lem:dyadic cubes} c)), it follows that for every $Q\in\Evil(R)$ there is at most a bounded number of cubes $P\in\Top(R)$ with $P^1=Q$. Thus,
		\begin{equation*}
			\sum_{P\in\Top(R)}\mu(P) \le \sum_{P\in\Top(R)}\mu(P^1)\lesssim \sum_{Q\in\Evil(R)}\mu(Q)
			\overset{\eqref{eq:evil-est}}{\lesssim}\E_4\,\mu(R)\lesssim \E_4\,\mu(R).
		\end{equation*}
	\end{proof}
	
%	\begin{lemma}\label{lem:good-cone2}
%		If $P\in\Top(R)$ and $Q\in\Br(P)$, then
%		\begin{equation}\label{eq:empcon}
%		X(z,0.5 J_R, \lambda \ELL(Q), \Lambda\ELL(P))\cap E=\varnothing \quad\text{for all $z\in \Lambda B_Q\cap P$.}
%		\end{equation}
%	\end{lemma}
%	\begin{proof}
%		Let $z\in \Lambda B_Q\cap P$. Let $Q=Q_0\subset Q_1\subset\dots\subset Q_k=P$ be a sequence of consecutive cubes,  so that $\ell(Q_i)=\rho^{-i}\ell(Q)$ and $z\in \Lambda B_{Q_i}\cap P$ for every $i=0,\dots, k$. By \lemref{lem:good-cone1} we have
%		\begin{equation}\label{eq:daa}
%			X(z,0.5 J_R, \lambda \ELL(Q_i), \Lambda\ELL(Q_i))\cap E=\varnothing,
%		\end{equation}
%		and at the same time $\Lambda\ELL(Q_i)>\lambda\ell(Q_{i+1})$. It follows that
%		\begin{equation*}
%			X(z,0.5 J_R, \lambda \ELL(Q), \Lambda\ELL(P))\subset \bigcup_{i=0}^k X(z,0.5 J_R, \lambda \ELL(Q_i), \Lambda\ELL(Q_i)),
%		\end{equation*}
%		which together with \eqref{eq:daa} gives $X(z,0.5 J_R, \lambda \ELL(Q), \Lambda\ELL(P))\cap E=\varnothing$.
%	\end{proof}
%	The property \eqref{eq:empcon} is almost the same as \eqref{eq:empty-cone}: we just need $P$ and $E$ to switch places. This follows from the symmetricity of cones and we show the details in the next lemma.
	We are going to apply \lemref{lem:key geometric lemma} with $P\in\Top(R)$ and $Q\in\Br(P)$. To do that, we need to verify that \eqref{eq:empty-cone} holds for $Q\in\Br(P)$.
	\begin{lemma}
		If $P\in\Top(R)$ and $Q\in\Br(P)$, then
		\begin{equation}\label{eq:empty cones4}
		X(z,0.5J_R,\lambda \ELL(Q), \Lambda\ELL(P))\cap P=\varnothing \quad\text{for all $z\in \Lambda B_Q\cap E$.}
		\end{equation}
	\end{lemma}
	\begin{proof}
		Suppose that \eqref{eq:empty cones4} fails, so that there exist $z\in \Lambda B_Q\cap E$ and $x\in X(z,0.5J_R,\lambda \ELL(Q), \Lambda\ELL(P))\cap P$. Let $S\in\Br(P)$ be such that $Q\subset S\subset P$ and 
		\begin{equation*}
		x\in X(z,0.5J_R,\lambda\ELL(S), \Lambda\ELL(S))\cap P.
		\end{equation*}
		Then, by the symmetricity of cones, we have
		\begin{equation}\label{eq:ff}
		z\in X(x,0.5J_R,\lambda\ELL(S), \Lambda\ELL(S))\cap E.
		\end{equation}
		At the same time, since $d_{J_R}(x,z)\lesssim \Lambda\ell(S)$ (this follows from \eqref{eq:ff} and \lemref{lem:coneinball}) and $z\in \Lambda B_Q\subset\Lambda B_S$, we have
		\begin{equation*}
			d_{J_R}(x,x_S)\le d_{J_R}(x,z)+d_{J_R}(z,x_S)\lesssim \Lambda\ell(S),
		\end{equation*}
		so that $x\in \Lambda^2 B_S\cap P$. Thus, we have found a point $x\in \Lambda^2 B_S\cap \wt{R}$ with $X(x,0.5J_R,\lambda\ELL(S), \Lambda\ELL(S))\cap E\neq\varnothing$. But this is a contradiction with \lemref{lem:good-cone1} since $S\in\Br(P)\subset\Tree(R)\setminus\Evil(R)$.
	\end{proof}

	We need two corollaries of \eqref{eq:empty cones4}. Recall that $I_Q=\pi_R^\perp(B_Q)$.
	\begin{lemma}\label{lem:graph}
		If $P\in\Top(R)$ and $Q_1,Q_2\in\Br(P)$ satisfy $\ell(Q_1)\le \ell(Q_2)$ and $I_{Q_1}\cap I_{Q_2}\neq\varnothing$, then $B_{Q_1}\subset CB_{Q_2}$, where $C> 1$ is an absolute constant.
	\end{lemma}
	\begin{proof}
		For $i\in \{1,2\}$ let $x_i\coloneqq x_{Q_i}\in Q_i$ and $y_i\in B_{Q_i}$ be such that $\pi_R^\perp(y_1)=\pi_R^\perp(y_2)$. Note that to show $B_{Q_1}\subset CB_{Q_2}$ it suffices to prove that $d_{J_R}(x_1,x_2)\lesssim \ell(Q_2)$.
		
		Recall that $B_{Q_1}$ and $B_{Q_2}$ are $d_{J_R}$-balls. By the definition of metric $d_{J_R}$ (see Subsection \ref{subsec:metrics})
		\begin{multline*}
			|\pi_R^\perp(x_1-x_2)|\le |\pi_R^\perp(x_1-y_1)|+|\pi_R^\perp(y_1-y_2)|+|\pi_R^\perp(y_2-x_2)|\\
			\lesssim \HH(J_R)\ell(Q_1)+0+\HH(J_R)\ell(Q_2)\lesssim \HH(J_R)\ell(Q_2).
		\end{multline*}
		Thus, to get $d_{J_R}(x_1,x_2)\lesssim \ell(Q_2)$ it remains to show $|\pi_R(x_1-x_2)|\le C \ell(Q_2).$
		
		Assume the opposite, so that $|\pi_R(x_1-x_2)|\ge C \ell(Q_2)$ for some large $C>1$. Then, we may use the previous computation to estimate
		\begin{equation*}
			|\pi_R^\perp(x_1-x_2)|\lesssim \HH(J_R)\ell(Q_2)\le \frac{\HH(J_R)}{C}|x_1-x_2|.
		\end{equation*}
		If $C$ is chosen large enough, we get $x_1\in X(x_2,0.5J_R)$. At the same time, since $x_1,x_2\in P$ we have $|x_1-x_2|\le\Lambda\ell(P)$. Thus,
		\begin{equation*}
			x_1\in X(x_2,0.5J_R,\lambda\ell(Q_2),\Lambda\ell(P))\cap P.
		\end{equation*}
		But this is a contradiction with \eqref{eq:empty cones4} because $x_2\in Q_2\subset \Lambda B_{Q_2}\cap E$.		
	\end{proof}
	For $k\ge 0$ let
	\begin{equation*}
	\Br_k(P)=\{Q\in\Br(P) : \ell(Q)=\rho^k\}.
	\end{equation*}

	\begin{lemma}
		For $P\in\Top(R)$ and any $k\ge 0$ the family of intervals $\{I_Q\}_{Q\in\Br_k(P)} $ has bounded overlaps:
		\begin{equation}\label{eq:int-bdd-ove}
		\sum_{Q\in\Br_k(P)}\one_{I_Q}\lesssim 1.
		\end{equation}
		In particular, for any interval $K\subset \R$ we have
		\begin{equation}\label{eq:bounded number of intervals}
			\#\big\{Q\in \Br_k(P)\ :\ I_Q\subset K\big\}\lesssim \frac{\HH(K)}{\HH(J_R)\rho^k}.
		\end{equation}
	\end{lemma}
	\begin{proof}
		Suppose that $Q,S\in \Br_k(P)$ satisfy $I_Q\cap I_S\neq\varnothing$. Then, by \lemref{lem:graph} we have $B_Q\subset CB_S$. At the same time, by \lemref{lem:dyadic cubes} c) the balls $5B(Q)$ with $Q\in\Br_k(P)$ are pairwise disjoint. Since $r(B(Q))\sim r(B_Q)= r(B_S)$, it follows that
		\begin{equation*}
			\#\{Q\in\Br_k(P) : B_Q\subset CB_S\}\lesssim_C 1.
		\end{equation*}
		This gives \eqref{eq:int-bdd-ove}.
		
		To see \eqref{eq:bounded number of intervals}, recall that $\HH(I_Q)\sim\HH(J_R)\rho^k$ for every $Q\in\Br_k(P)$. Thus,
		\begin{multline*}
			\#\big\{Q\in \Br_k(P)\ :\ I_Q\subset K\big\}\lesssim \frac{1}{\HH(J_R)\rho^k} \sum_{Q\in\Br_k(P)}\int_K \one_{I_Q}(t)\, dt\\
			=  \frac{1}{\HH(J_R)\rho^k} \int_K \sum_{Q\in\Br_k(P)} \one_{I_Q}(t)\, dt\overset{\eqref{eq:int-bdd-ove}}{\lesssim}\frac{\HH(K)}{\HH(J_R)\rho^k}.
		\end{multline*}
	\end{proof}
	We are ready to apply our key geometric lemma, Lemma \ref{lem:key geometric lemma}, to estimate the measure of exterior bad cubes in a fixed $\Br(P)$.
	\begin{lemma}
		For every $P\in\Top(R)$ we have
		\begin{equation}\label{eq:bad-in-branch}
		\sum_{S\in\Be(R)\cap\Br(P)}\mu(S)\lesssim A^2M^2\mu(P).
		\end{equation}
	\end{lemma}
	\begin{proof}
		Suppose that $S\in\Be(R)\cap\Br(P)$. 
		It follows from \lemref{lem:key geometric lemma} that there exists $I\in\cW(P,\ell(S))$ such that $I_S\subset \lambda^{-1}\Lambda I$. Observe that 
		\begin{equation*}
		\HH(I)\sim \lambda \HH(J_R)\ell(S)\le \lambda \HH(J_R)\ell(P)\sim\lambda \HH(I_P).
		\end{equation*}
		Since $I_P\cap \lambda^{-1}\Lambda I\neq\varnothing$, we get that $I\subset CI_P$ for $C=C(\Lambda)$.
		
		It follows that
		\begin{align*}
		\sum_{S\in\Be(R)\cap\Br(P)}\mu(S)&\lesssim \sum_{k\ge 0} \sum_{I\in\cW(P,\rho^k)}\sum_{S\in\Br_k(R): I_S\subset \lambda^{-1}\Lambda I}\mu(S)\\
		&\overset{\eqref{eq:meas-upper-bd2}}{\lesssim} M \sum_{k\ge 0} \sum_{I\in\cW(P,\rho^k)}\sum_{S\in\Br_k(R): I_S\subset \lambda^{-1}\Lambda I}\HH(J_R)\rho^k\\
		&\overset{\eqref{eq:bounded number of intervals}}{\lesssim} M \sum_{k\ge 0} \sum_{I\in\cW(P,\rho^k)}\frac{\HH(\lambda^{-1}\Lambda I)}{\HH(J_R)\rho^k}\HH(J_R)\rho^k
		\\
		&\lesssim M\lambda^{-1}\Lambda\sum_{I\in\cW(P):I\subset CI_P}\HH(I)\lesssim M\lambda^{-1}\Lambda\HH(J_R)\ell(P)\\
		&\overset{\eqref{eq:meas-lower-bd2}}{\lesssim} AM\lambda^{-1}\Lambda\mu(P).
		\end{align*}
		Recalling that $\Lambda$ is an absolute constant and $\lambda\sim M^{-1}A^{-1}$ finishes the proof.
	\end{proof}
	We are ready to conclude the proof of \lemref{lem:exterior-est}.
	\begin{proof}[Proof of \lemref{lem:exterior-est}]
		We want to show
		\begin{equation*}
		\sum_{Q\in \Be(R)}\mu(Q)\lesssim (AM)^6\E_3\,\mu(R).
		\end{equation*}
		Clearly,
		\begin{multline*}
		\sum_{Q\in \Be(R)}\mu(Q) = \sum_{Q\in \Be(R)\cap\Evil(R)}\mu(Q)+\sum_{Q\in \Be(R)\setminus\Evil(R)}\mu(Q)\\
		\overset{\eqref{eq:evil-est}}{\lesssim} (AM)^4\E_3\,\mu(R) + \sum_{Q\in \Be(R)\setminus\Evil(R)}\mu(Q),
		\end{multline*}
		so it suffices to estimate the last sum. Recalling \eqref{eq:branch-decomp}, we have
		\begin{multline*}
		\sum_{Q\in \Be(R)\setminus\Evil(R)}\mu(Q) = \sum_{P\in\Top(R)}\,\sum_{Q\in \Be(R)\cap\Br(P)}\mu(Q)\\
		\overset{\eqref{eq:bad-in-branch}}{\lesssim} A^2M^2 \sum_{P\in\Top(R)}\mu(P) \overset{\eqref{eq:top-est}}{\lesssim}A^2M^2\E_4\mu(R) = (AM)^6\E_3\mu(R).
		\end{multline*}
	\end{proof}

	Putting together the estimates from \lemref{lem:exterior-est} and \lemref{lem:Bi-pack} gives
	\begin{equation*}
	\sum_{Q\in\Bad(R)}\mu(Q)\lesssim (AM)^6\E_3\, \mu(R) = (AM)^8\E_2\,\mu(R),
	\end{equation*}
	so that \lemref{lem:BADDD} holds. Together with \lemref{lem:badtobadR}, we get
	\begin{equation*}
	\sum_{Q\in\Bad\cap\T(R)}\mu(Q)\lesssim (AM)^9\E_2\,\mu(R),
	\end{equation*}
	which establishes \lemref{lem:Bad-pack}. We showed earlier that \lemref{lem:Bad-pack} implies \propref{prop:Bad-pack}, and that \propref{prop:Bad-pack} implies \propref{prop:E0energy-est}. Together with the work done in Subsection \ref{subsec:Gstar}, this completes the proof of our Main Proposition, \propref{prop:main}.
	
	\section{Key geometric lemma}\label{sec:KGL}
	In this section we prove \lemref{lem:key geometric lemma} which was crucially used in the previous section. For the sake of possible future applications we will prove a more general, self-contained result stated below.
	
	\begin{lemma}\label{lem:key-lem}
		 Suppose that $\alpha> 1$ and $\rho\in (0,1)$. There exist constants $0<c_J<1$, $0<c_\lambda<1$, and $\rho^{-1}<\Lambda\ll c_\lambda^{-1}$, depending only on $\rho$ and $\alpha$, such that the following holds. 
		
		Suppose that $M\ge 1$, $A\ge 1$, and let $\lambda=c_\lambda M^{-1}A^{-1}$. Assume that
		\begin{enumerate}[label=({\roman*})]
			\item $J\subset \TT$ is an interval with $\HH(J)\le c_JM^{-1}A^{-1}$,
			\item $E$ is an Ahlfors regular set with constant $A$, and $\mu=\HH|_E$,
			\item for some $z_0\in E$ and $0<r<R$ we have a subset $F\subset E\cap B_J(z_0,R)$, and a $d_J$-ball $B_0\coloneqq B_J(z_0,r)$ satisfying
			\begin{equation}\label{eq:measure-assump}
				\mu(\Lambda B_0)\le M\HH(J)r
			\end{equation}
			and
			\begin{equation}\label{eq:empty-cone2}
				X(z,J,\lambda r, \Lambda R)\cap F=\varnothing\quad\text{for all $z\in \Lambda B_0\cap E$.}
			\end{equation}
			\item there exists $x\in B_0\cap F$ such that
			\begin{equation*}
				X(x,\alpha J\setminus J,\rho r, r)\cap E\neq\varnothing.
			\end{equation*}
		\end{enumerate}
		Then, there exists an interval $I\subset \R$ such that $\pi_J^\perp(F)\cap I = \varnothing,$ $\HH(I)\sim_{\alpha,\rho} \lambda \HH(J)r$, and $\pi_{J}^\perp(B_0)\subset \lambda^{-1}\Lambda I$.
	\end{lemma}
	First, let us show how \lemref{lem:key geometric lemma} follows from this more general statement.
	\begin{proof}[Proof of \lemref{lem:key geometric lemma}]
		Suppose that the assumptions of \lemref{lem:key geometric lemma} hold. Let $J=0.5J_R$, and note that 
		\begin{equation*}
		d_{J_R}(x,y)\le d_J(x,y)\le 2d_{J_R}(x,y).
		\end{equation*}
		Let $\alpha=30$, and let $\rho$ be the absolute constant we fixed in \lemref{lem:dyadic cubes}. Let $c_J'$, $c_\lambda'$ and $\Lambda'$ be the constants depending on $\alpha$ and $\rho$ given by \lemref{lem:key-lem}. Since $\alpha$ and $\rho$ were absolute constants,  $c_\lambda'$ and $\Lambda'$ are also absolute constants. 
		
		We define the absolute constants $c_J,c_\lambda$ and $\Lambda$, whose existence is asserted by \lemref{lem:key geometric lemma}, as
		\begin{align*}
		\Lambda &\coloneqq C\Lambda',\\
		c_J&\coloneqq c\Lambda^{-1}c_J',\\
		c_\lambda&\coloneqq c\, \Lambda^{-1} c_\lambda',
		\end{align*}
		where $0<c<1<C$ are absolute and will be chosen in the course of the proof.
		
		Let $M'=C\Lambda' M$ and $A'=A$. Let $r=\ell(Q)$, $R=C\ell(P)$ for some absolute $C\ge 10$, $z_0=x_Q$ where $x_Q$ is the center of $Q$ as in \lemref{lem:dyadic cubes}, and $F=P$. Let us verify the assumptions of \lemref{lem:key-lem} with these choices, and with constants $M'$ and $A'$.
		
		Assumption (i) holds because $J\subset J_0$ and by \propref{prop:main} (b)
		\begin{equation*}
		\HH(J_0)\le c_JM^{-1}A^{-1}\le c_J'(M')^{-1}A^{-1}.
		\end{equation*}
		Assumption (ii) is trivial.
		
		Regarding assumption (iii), it is easy to see that
		\begin{multline*}
		F=P\subset E\cap B_P = E\cap B_{J_R}(x_P,50r(P))\\
		\subset E\cap B_{J}(x_Q,C\ell(P)) = E\cap B_{J}(z_0,R),
		\end{multline*}
		where in the second inclusion we also used that $x_Q\in Q\subset P\subset B_P$ so that $d_{J_P}(x_Q,x_P)=d_{J_R}(x_Q,x_P)\lesssim \ell(P)$.
		
		Observe that $B_J(z_0,r)\subset B_{J_R}(x_Q,50r(Q))=B_Q$. Thus,
		\begin{equation*}
		\mu(\Lambda' B_0)\le\mu(\Lambda' B_Q)\overset{\eqref{eq:meas-upper-bd2}}{\le}C\Lambda' M\HH(J)r\le M'\HH(J)r
		\end{equation*}
		which shows \eqref{eq:measure-assump}. 
		
		The assumption \eqref{eq:empty-cone2} follows from \eqref{eq:empty-cone} because $\Lambda'R\le\Lambda\ell(P)$ and
		\begin{equation*}
		\lambda\ell(Q) = c_\lambda (AM)^{-1}\ell(Q) \le c_\lambda'(AM')^{-1}r=\lambda'r,
		\end{equation*}
		if the absolute constants denoted by $c$ are chosen small enough and by $C$ are chosen large enough. This verifies assumption (iii) of \lemref{lem:key-lem}.
		
		Finally, assumption (iv) holds by the assumption (ii) of \lemref{lem:key geometric lemma}.
		
		Thus, we may apply \lemref{lem:key-lem} to obtain an interval $I'\subset\R$ such that $\pi_R^\perp(P)\cap I'=\varnothing$, $\HH(I')\sim\lambda'\HH(J_R)\ell(Q)$, and
		\begin{equation*}
		\pi_{R}^\perp(B_0)\subset (\lambda')^{-1}\Lambda' I'.
		\end{equation*}
		Without loss of generality, assume $I'$ is a dyadic interval. Since $z_0=x_Q\in P$, we get that $\pi_{R}^\perp(x_Q)\in (\lambda')^{-1}\Lambda' I'\cap \pi_{R}^\perp(P)$ and so
		\begin{equation*}
		\dist(I',\pi_{R}^\perp(\overline{P}))\le \dist(I',\pi_{R}^\perp(P))\lesssim (\lambda')^{-1}\Lambda'\HH(I')\sim \Lambda'\HH(J_R)\ell(Q).
		\end{equation*}
		It is easy to show that, since $\pi_R^\perp(P)\cap I'=\varnothing$, we have
		\begin{equation*}
		\pi_R^\perp(\overline{P})\cap I'\subset\partial I',
		\end{equation*}
		and possibly $\pi_R^\perp(\overline{P})\cap I'=\varnothing$.
%		 $\pi_R^\perp(\overline{P})\cap I'=\varnothing$, or $$ is contained in the endpoints of $I'$. 
 		It follows that we may find a dyadic descendant $I''\subset I'$ with $\HH(I'')\sim\HH(I')\sim\lambda' \HH(J_R)\ell(Q)$ and such that
		\begin{equation}\label{eq:lkdshgds}
		3\,\HH(I'')\le \dist(I'',\pi_{R}^\perp(\overline{P}))\lesssim (\lambda')^{-1}\Lambda\HH(I'').
		\end{equation}
%		In particular, $3I''\cap \pi_{R}^\perp(\overline{P}) = \varnothing$.
		Furthermore, we may find another dyadic interval $I'''$ with $\HH(I''')=\HH(I'')$,
		\begin{equation}\label{eq:rgerh}
		\dist(I'',I''')\lesssim (\lambda')^{-1}\Lambda'\HH(I'')
		\end{equation}
		and such that
		\begin{equation}\label{eq:lkdshgds2}
		3\,\HH(I''')\le \dist(I''',\pi_{R}^\perp(\overline{P}))\le 4 \HH(I''').
		\end{equation}
		To do that, we consider dyadic intervals $\{I_k\}_{k\in\Z}$ with $I_k=I''+k\cdot\HH(I'')$. It follows from \eqref{eq:lkdshgds} that for some $|k|\lesssim (\lambda')^{-1}\Lambda'$ the interval $I_k$ satisfies \eqref{eq:lkdshgds2}, and we set $I'''\coloneqq I_k$.
		
		Note that $3I'''\cap \pi_{R}^\perp(\overline{P})=\varnothing$. We define $I$ to be the maximal dyadic interval with $I'''\subset I$ and $3I\cap  \pi_{R}^\perp(\overline{P})=\varnothing$, so that $I\in\cW(P)$ is a Whitney interval for $U(P)=\R\setminus \pi_{R}^\perp(\overline{P})$. It follows from \eqref{eq:lkdshgds2} that
		\begin{equation*}
		\HH(I)\lesssim \HH(I''')\sim \lambda'\HH(J_R)\ell(Q)
		\end{equation*}
		and at the same time $\HH(I)\ge\HH(I''')\sim \lambda'\HH(J_R)\ell(Q)$. Recalling that $\lambda'\sim\lambda$, we get that $I\in\cW(P,\ell(Q))$ if the constant $C>1$ in the definition of $\cW(P,r)$ is chosen large enough.
		
		It remains to show that 
		\begin{equation}\label{eq:ehleh}
		I_Q=\pi_R^\perp(B_Q)\subset \lambda^{-1}\Lambda I.
		\end{equation}
		Recall that $\pi_{R}^\perp(x_Q)\in\pi_{R}^\perp(B_0)\subset (\lambda')^{-1}\Lambda' I'$, and observe that $\pi_{R}^\perp(x_Q)$ is the center of $I_Q$. We have
		\begin{align*}
		\dist(\pi_{R}^\perp(x_Q),I)\le \dist(\pi_{R}^\perp(x_Q), I''')
		&\le \dist(\pi_{R}^\perp(x_Q), I'') + \dist(I'',I''')+\HH(I'')\\
		&\overset{\eqref{eq:rgerh}}{\lesssim} \dist(\pi_{R}^\perp(x_Q), I') + \HH(I')+(\lambda')^{-1}\Lambda'\HH(I')\\
		&\lesssim (\lambda')^{-1}\Lambda'\HH(I')\sim (\lambda')^{-1}\Lambda'\HH(I).
		\end{align*}
		Taking into account that
		\begin{equation*}
		\lambda^{-1} = c_{\lambda}^{-1}AM=c^{-1}C (c_{\lambda}')^{-1}A'M'=c^{-1}C(\lambda')^{-1}
		\end{equation*}
		and $\Lambda = C\Lambda'$, it follows that $\dist(\pi_{R}^\perp(x_Q),I)\lesssim cC^{-1}\lambda^{-1}\Lambda\HH(I)$. Since we also have 
		\begin{equation*}
		\HH(I_Q)\sim \HH(J_R)\ell(Q)\sim (\lambda')^{-1}\HH(I)\lesssim cC^{-1}\lambda^{-1}\HH(I),
		\end{equation*}
		we get \eqref{eq:ehleh} if $0<c<1$ is small enough and $C>1$ is large enough. This concludes the proof of \lemref{lem:key geometric lemma}.
	\end{proof}
	
		We begin the proof of \lemref{lem:key-lem}. Fix $0<r<R$, $F\subset E\cap B_J(z_0,R)$, $B_0=B_J(z_0,r),$ and $x\in B_0\cap F$ as in the assumptions. For brevity, set $\pi\coloneqq \pi_J$.
		
		Let $y\in X(x,\alpha J\setminus J,\rho r, r)\cap E$. Observe that $y\in CB_0$ for some $C=C(\alpha)>1$, and also
		\begin{align}
		|\pi(x-y)|&\sim_\rho r,\label{eq:dim1}\\
		|\pi^\perp(x-y)|&\sim_{\alpha,\rho} \HH(J)r\label{eq:dim2}.
		\end{align}
		We are going to find an interval $I\subset \R$ such that $\pi^\perp(F)\cap I = \varnothing,$ $\HH(I)\sim \lambda \HH(J)r$, and $\pi^\perp(x)\in \Lambda I$. After that it won't be too difficult to show that $\pi^\perp(B_0)\subset \lambda^{-1}\Lambda I$.
		
		Let 
		\begin{equation*}
		t_\cG\coloneqq \frac{\pi^\perp(x)+\pi^\perp(y)}{2}\in\R,\\
		\end{equation*}
		and
		\begin{equation*}
		\cG \coloneqq \bigg\{z\in \R^2 \ :\ |\pi^\perp(z)-t_{\cG}|\le 2|\pi^\perp(x-y)|,\
	|\pi(z-y)| \le \frac{|\pi(x-y)|}{2}\bigg\}.
		\end{equation*}
		Note that by \eqref{eq:dim1} and \eqref{eq:dim2} $\cG$ is a tube of dimensions roughly $\HH(J)r\times r$. Moreover, for $z\in\cG$
		\begin{equation}\label{eq:zbeatsy}
		|\pi^\perp(z-x)| = \left|\pi^\perp(z)-t_{\cG} - \frac{\pi^\perp(x-y)}{2}\right|\le 3|\pi^\perp(x-y)|.
		\end{equation}
	and
		\begin{equation}\label{eq:zbeatsx}
			|\pi^\perp(z-y)| = \left|\pi^\perp(z)-t_{\cG} - \frac{\pi^\perp(y-x)}{2}\right|\le 3|\pi^\perp(x-y)|.
		\end{equation}
		
		Let $N$ be a big integer satisfying
		\begin{equation}\label{eq:N-def}
		N= \lceil C_N AM\rceil,
		\end{equation}
		for some big constant $C_N=C_N(\alpha,\rho)$ whose precise value will be fixed later on.
		
		We divide $\cG$ into $2N+1$ subsets $\{\cG_i\}_{i=-N}^N$ defined as
		\begin{multline*}
		\cG_i = \bigg\{z\in \cG \ :\ 
		\frac{(2i-1)}{2(2N+1)}|\pi(x-y)|
		\le \pi(z-y) \le \frac{(2i+1)}{2(2N+1)}|\pi(x-y)|\bigg\}.
		\end{multline*}
		Each $\cG_i$ is a tube of dimensions roughly $\HH(J)r\times r/N$. Since $\HH(J)\le c_JM^{-1}A^{1}$, if we choose $c_J$ small enough (depending only on $C_N$) we have $\HH(J)\ll N^{-1}$. In particular, for each $\cG_i$ with $\cG_i\cap E\neq\varnothing$ we have
		\begin{equation}\label{eq:Gi-lower-bd}
		\mu(3\cG_i)\gtrsim_{\alpha,\rho} A^{-1}\HH(J)r.
		\end{equation}
		
		\begin{definition}
			For each $\cG_i$ with $\cG_i\cap E\neq\varnothing$ we define $z_i\in\cG_i\cap E$ to be a point satisfying
			\begin{equation*}
			|\pi^\perp(z_i-x)| = \inf_{z\in\cG_i\cap E}|\pi^\perp(z-x)|.
			\end{equation*}
			There may be many such points, and we just pick one.
		\end{definition}
		\begin{definition}
			If $-N\le i,j,\le N$ and $\cG_i\cap E\neq\varnothing$, we will say that $\cG_i$ \emph{beats} $\cG_j$ if either $\cG_j\cap E=\varnothing$, or
			\begin{equation*}
			|\pi^\perp(z_i-x)|\le |\pi^\perp(z_j-x)|.
			\end{equation*}
		\end{definition}
		\begin{definition}
			For $-N+1\le i\le N-1$ we will say that $\cG_i$ is \emph{nice} if $\cG_i\cap E\neq\varnothing$ and $\cG_i$ beats both $\cG_{i-1}$ and $\cG_{i+1}$.
		\end{definition}
		Observe that $\cG_i$ is nice if and only if $\cG_i\cap E\neq\varnothing$ and
		\begin{equation*}
			|\pi^\perp(z_i-x)| = \inf_{z\in(\cG_{i-1}\cup\cG_i\cup\cG_{i+1})\cap E}|\pi^\perp(z-x)|.
		\end{equation*}
		
		\begin{lemma}
%			We have either
%			\begin{equation}\label{eq:no-nice}
%			\bigcup_{i=-1}^1\cG_i\cap R = \varnothing,
%			\end{equation}
			If $C_N=C_N(\alpha,\rho)$ is chosen big enough, then there exists $-N+1\le i_*\le N-1$ such that $\cG_i$ is nice.
		\end{lemma}
		\begin{proof}
			Observe that $y\in \cG_0\cap E$, and so $z_0$ is well-defined.
			If $\cG_0$ is nice, we are done. Otherwise, either $\cG_{-1}$ or $\cG_{1}$ beats $\cG_0$. Without loss of generality assume $\cG_{1}$ beats $\cG_0$. 
			
			If $\cG_{1}$ is nice, we are done. Otherwise, $\cG_{2}$ beats $\cG_{1}$. Repeating this reasoning $N-1$ times, we either find a nice $\cG_j$, or we get that for each $j\in \{0,\dots, N\}$ the intersection $\cG_j\cap E$ is non-empty. Since $\{3\cG_j\}_{j=-N}^N$ have bounded overlaps, it follows that
			\begin{equation*}
			\mu(3\cG)\gtrsim\sum_{j=-N}^N\mu(3\cG_j)\overset{\eqref{eq:Gi-lower-bd}}{\gtrsim_{\alpha,\rho}} N\cdot A^{-1}\HH(J)r.
			\end{equation*}
			Hence,
			\begin{equation*}
			\mu(CB_0)\ge c NA^{-1}\HH(J)r
			\end{equation*}
			for some constants $0<c<1<C$ depending only on $\alpha$ and $\rho$. Recalling \eqref{eq:N-def}, we get
			\begin{equation*}
			\mu(CB_0)\ge cC_NM\HH(J)r.
			\end{equation*}
			But this is a contradiction with \eqref{eq:measure-assump}, assuming $C_N\ge c^{-1}$ and $\Lambda\ge C$.
		\end{proof}
	
		Let  $i_*\in\{-N+1, \dots, N-1\}$ be the nice tube $\cG_{i_*}$ found in the proof above, and set $z_*\coloneqq z_{i_*}$. Note that
		\begin{equation}\label{eq:z*y}
			|\pi^\perp(z_*-x)|\le |\pi^\perp(y-x)|,
		\end{equation}
		and the equality holds if and only if $i_*=0$ and $y=z_0$.
%		\begin{equation}\label{eq:no-nice2}
%		\cB = \bigcup_{i=-1}^1\cG_i
%		\end{equation}
%		and $z_*\coloneqq y\in\cB$. Otherwise, there exists $-N+1\le i\le N-1$ such that $\cG_i$ is nice, and we set
		
		We define
		\begin{equation}\label{eq:nice}
		\cB = \{z\in\cG_{i_*-1}\cup\cG_{i_*}\cup\cG_{i_*+1} : |\pi^\perp(z-x)|<|\pi^\perp(z_*-x)|\}.
		\end{equation}		
		\begin{lemma}
			We have 
%			\begin{equation}\label{eq:zbeatsx}
%			|\pi^\perp(z-x)|\le|\pi^\perp(z_*-x)|\quad\text{for $z\in\cB$},
%			\end{equation}
			\begin{equation}\label{eq:BcapR}
			\cB\cap E=\varnothing.
			\end{equation}
		\end{lemma}
		\begin{proof}
			If $z\in\cB\cap E$, then $|\pi^\perp(z-x)|<|\pi^\perp(z_*-x)|$. At the same time
			\begin{equation*}
			|\pi^\perp(z_*-x)| = \inf_{z'\in (\cG_{i_*-1}\cup\cG_{i_*}\cup\cG_{i_*+1})\cap E} |\pi^\perp(z'-x)|
			\end{equation*}
			because $\cG_{i_*}$ is nice. Since $\cB\subset \cG_{i_*-1}\cup\cG_{i_*}\cup\cG_{i_*+1}$, this leads to a contradiction.
		\end{proof}
		\begin{lemma}
			We have
			\begin{equation}\label{eq:z*x}
			|\pi^\perp(z_*-x)|\sim_{\alpha,\rho} \HH(J)r.
			\end{equation}
%			and
%			\begin{equation}\label{eq:z*x}
%				|\pi^\perp(z_*-y)|\le \HH(J)r
%			\end{equation}
		\end{lemma}
		\begin{proof}
			The upper bound is immediate because $z_*\in\cG$ and we have \eqref{eq:zbeatsy} and \eqref{eq:dim2}. Now suppose that the lower bound fails, so that $|\pi^\perp(z_*-x)|\le c \HH(J)r$ for some tiny $0<c<1$. Recalling that $|z_*-x|\sim\dist(\cG,x)\sim_\rho r$ we get
			\begin{equation*}
			|\pi^\perp(z_*-x)|\le c \HH(J)r\sim_\rho c \HH(J)|z_*-x|,
			\end{equation*}
			and so assuming $c$ is small enough we get $x\in X(z_*, J, \lambda r, \Lambda r)\cap F$, which is a contradiction with \eqref{eq:empty-cone2}.
		\end{proof}
	
		Let 
		\begin{equation*}
		t_\cY\coloneqq c_\cY\lambda\pi^\perp(x)+(1-c_\cY\lambda)\pi^\perp(z_*)\in\R,
		\end{equation*}
		for some small constant $0<c_\cY=c_\cY(\alpha,\rho)<1$, and consider the tube
		\begin{multline*}
		\cY=\{z\in\R^2 :  |\pi^\perp(z)-t_\cY|\le \tfrac{1}{2}c_\cY\lambda|\pi^\perp(z_*-x)|,\
		|\pi(z-z_*)|\le \Lambda R/2\}
		\end{multline*}
		Note that $\cY\subset B(z_*,\Lambda R)$, and also that for $z\in\cY$ we have
		\begin{multline}\label{eq:cYz}
		|\pi^\perp(z)-\pi^\perp(z_*)| = \left|\pi^\perp(z)-t_\cY + c_\cY\lambda(\pi^\perp(x)-\pi^\perp(z_*))\right|\\
		\le 2c_\cY\lambda |\pi^\perp(z_*-x)|\overset{\eqref{eq:z*x}}{\sim_{\alpha,\rho}}c_\cY\lambda\HH(J)r
		\end{multline}
		and
		\begin{multline}\label{eq:cYz2}
			|\pi^\perp(z-x)| = |\pi^\perp(z)-t_\cY + t_\cY-\pi^\perp(x)|\\
			\le \tfrac{1}{2}c_\cY\lambda|\pi^\perp(z_*-x)| + (1-c_\cY\lambda)|\pi^\perp(z_*-x)|<|\pi^\perp(z_*-x)|.
		\end{multline}
		\begin{lemma}
			If $0<c_\cY=c_\cY(\alpha,\rho)<1$ and $0<c_\lambda=c_\lambda(\alpha,\rho)< C_N^{-1}$ are chosen small enough, then $\cY\cap F=\varnothing$.
		\end{lemma}
		\begin{proof}
			We will show that
			\begin{equation*}
			\cY\subset \cB\cup X(z_*, J, \lambda r,\Lambda R).
			\end{equation*}
			Then, the claim follows from \eqref{eq:BcapR} and \eqref{eq:empty-cone2}. Suppose the contrary, so that there exists
			\begin{equation*}
			z\in\cY\setminus (\cB\cup X(z_*, J, \lambda r,\Lambda R)).
			\end{equation*}
			If $\lambda r\le|z-z_*|\le \Lambda R$, then $z\notin X(z_*, J)$, which implies
			\begin{equation*}
			|\pi^\perp(z)-\pi^\perp(z_*)|\ge c\HH(J)|z-z_*|\ge c\lambda\HH(J)r,
			\end{equation*}
			but this contradicts \eqref{eq:cYz}, assuming $c_\cY$ is small enough. So either $|z-z_*|\ge \Lambda R$ or $|z-z_*|\le \lambda r$. The former is not the case because $\cY\subset B(z_*,\Lambda R)$.
			
			So we have $|z-z_*|\le \lambda r$. We will show that this implies $z\in\cB$.			
			First, let's show that $z\in\cG$. Observe that
			\begin{align*}
				|t_\cY-t_{\cG}| &= |c_\cY \lambda \pi^\perp(x)+(1-c_\cY \lambda )\pi^\perp(z_*) - 0.5\pi^\perp(x) - 0.5\pi^\perp(y)|\\
				&= |c_\cY \lambda \pi^\perp(x-0.5x-0.5y)+(1-c_\cY \lambda )(\pi^\perp(z_*)-t_{\cG} )|\\
				&\le c_\cY \lambda \frac{|\pi^\perp(x-y)|}{2}+(1-c_\cY \lambda )|\pi^\perp(z_*)-t_{\cG}|
				\le (2-2.5c_\cY \lambda )|\pi^\perp(x-y)|,
			\end{align*}
			where in the last inequality we used the fact that $|\pi^\perp(z_*)-t_{\cG}|\le 2|\pi^\perp(x-y)|$ because $z_*\in\cG$. Hence,
			\begin{multline*}
				|\pi^\perp(z)-t_{\cG}| \le |\pi^\perp(z)-t_\cY| + |t_\cY-t_{\cG}|\\
				\le c_\cY \lambda |\pi^\perp(z_*-x)| + (2-2.5c_\cY \lambda )|\pi^\perp(x-y)|\overset{\eqref{eq:z*y}}{\le} 2|\pi^\perp(x-y)|.
			\end{multline*}
			This shows that $\pi^\perp(z)\in\pi^\perp(\cG)$. 
			
			Now suppose that $c_\lambda\ll C_N^{-1}$, so that $\lambda=c_\lambda A^{-1}M^{-1}\ll N^{-1}= \lceil C_N AM\rceil^{-1}$. Then,
			\begin{equation*}
				|z-z_*|\le \lambda r\le C\lambda |\pi(x-y)|\le \frac{1}{2N+1}|\pi(x-y)|.
			\end{equation*}
			Using also the fact that $z_*\in \cG_{i_*}$, we get
			\begin{multline*}
				\pi(z-y) = \pi(z-z_*)+\pi(z_*-y) \le \frac{1}{2N+1}|\pi(x-y)| + \frac{(2i_*+1)}{2(2N+1)}|\pi(x-y)|\\= \frac{(2(i_*+1)+1)}{2(2N+1)}|\pi(x-y)|,
			\end{multline*}
			and similarly 
			\begin{multline*}
				\pi(z-y) = \pi(z-z_*)+\pi(z_*-y) \ge \frac{(2i_*-1)}{2(2N+1)}|\pi(x-y)| - \frac{1}{2N+1}|\pi(x-y)|\\= \frac{(2(i_*-1)-1)}{2(2N+1)}|\pi(x-y)|.
			\end{multline*}
			This shows that 
			\begin{equation*}
				z\in \cG_{i_*-1}\cup\cG_{i_*}\cup\cG_{i_*+1}.
			\end{equation*}
			Together with \eqref{eq:cYz2}, we get that $z\in \cB$.
		\end{proof}
	
		\begin{lemma}
			There exists an interval $I\subset \R$ such that $\pi^\perp(F)\cap I = \varnothing,$ $\HH(I)\sim_{\alpha,\rho} \lambda \HH(J)r$, and $\pi^\perp(B_0)\subset \Lambda \lambda^{-1} I$.
		\end{lemma}
		\begin{proof}
			Observe that $I\coloneqq \pi^\perp(\cY)$ is an interval centered at $t_\cY$ of length 
			\begin{equation*}
				c_\cY  \lambda|\pi^\perp(z_*-x)|\sim_\alpha \lambda \HH(J)r
			\end{equation*}
			Since
			\begin{equation*}
				|\pi^\perp(x)-t_\cY| = |\pi^\perp(x)-c_\cY \lambda \pi^\perp(x)-(1-c_\cY \lambda )\pi^\perp(z_*)|\le |\pi^\perp(z_*-x)|,
			\end{equation*}
			we get that $\pi^\perp(x) \in C\lambda^{-1}I$. Since $\pi^\perp(B_0)$ is an interval of length $\sim\HH(J) r$ containing $\pi^\perp(x)$, this further implies that $\pi^\perp(B_0)\subset \lambda^{-1}\Lambda I$ assuming $\Lambda$ large enough.
			
			It remains to show that
			\begin{equation}\label{eq:nointersec}
			I\cap \pi^\perp(F)=\varnothing
			\end{equation}
			 Recall that $\cY\cap F=\varnothing$ and $F\subset B_J(z_0,R)$. Thus, to show \eqref{eq:nointersec} it suffices to prove
			 \begin{equation*}
			 (\pi^\perp)^{-1}(I)\cap B_J(z_0,R)\subset \cY.
			 \end{equation*}
			 Recall that $I= \pi^\perp(\cY)$ and $\cY=\pi^{-1}(\pi(\cY))\cap (\pi^\perp)^{-1}(I)$. We are going to show that $B_J(z_0,R)\subset \pi^{-1}(\pi(\cY))$.
			 
			 By the definition of $\cY$
			 \begin{equation*}
			 \pi^{-1}(\pi(\cY)) = \{z\in\R^2: |\pi(z-z_*)|\le\Lambda R/2\}.
			 \end{equation*}
			 Recall that $z_*\in\cG\subset CB_0=B_J(z_0,Cr)$. Thus, for any $z\in B_J(z_0,R)$ we have
			 \begin{equation*}
			 |\pi(z-z_*)|\le |\pi(z-z_0)|+|\pi(z_0-z_*)|\le R+Cr\le \Lambda R/2
			 \end{equation*}
			 assuming $\Lambda$ is large enough. This gives $B_J(z_0,R)\subset \pi^{-1}(\pi(\cY))$ and finishes the proof of \eqref{eq:nointersec}.			
		\end{proof}
		\appendix
			\section{Reduction to parallel segments}\label{app:paral}
		In this section we prove that it suffices to prove \thmref{thm:main} under the additional assumption that $E$ is a finite union of parallel segments. Recall that in \thmref{thm:main} we assume that
		\begin{equation}\label{eq:Eassump}
			\text{$E$ is Ahlfors regular with constant $A$,\quad $\Fav(E)\ge\kappa\HH(E)$.}
		\end{equation}
		\begin{lemma}\label{lem:parallel2}
			Suppose that there exist functions $\phi,\psi:(0,1)\times\R_+\to \R_+$ such that the following holds: if $E\subset [0,1]^2$ is a finite union of parallel segments and it satisfies \eqref{eq:Eassump}, then there exists a Lipschitz graph $\Gamma$ such that $\lip(\Gamma)\le \phi(\kappa,A)$ and 
			\begin{equation*}
				\HH(E\cap \Gamma)\ge \psi(\kappa,A)\cdot\HH(E).
			\end{equation*}
			
			Then, there is an absolute constant $C>1$ such that for every $E\subset [0,1]^2$ satisfying \eqref{eq:Eassump} there exists a Lipschitz graph $\Gamma$ such that $\lip(\Gamma)\lesssim \phi(C^{-1}A^{-1}\kappa,CA^2)$ and 
			\begin{equation*}
				\HH(F\cap \Gamma)\gtrsim A^{-2} \psi(C^{-1}A^{-1}\kappa,CA^2)\cdot\HH(E).
			\end{equation*}
		\end{lemma}
		We begin the proof of \lemref{lem:parallel2}.
		
		Suppose that $\phi,\psi:(0,1)\times\R_+\to \R_+$ are as in the assumptions. Let $E\subset [0,1]^2$ satisfy $\eqref{eq:Eassump}.$ We need to find a Lipschitz graph covering a big piece of $E$.
		
		Denote by $\mathfrak{D}(\R^2)$ the usual dyadic cubes in $\R^2$, and by $\mathfrak{D}_k(\R^2)$ the dyadic cubes with sidelength $2^{-k}$. For every $k\ge 0$ let
		\begin{equation*}
			\mathfrak{D}_k(E)\coloneqq \{Q\in \mathfrak{D}_k(\R^2) : Q\cap E\neq\varnothing\}.
		\end{equation*}
		Note that for any $k\ge 0$ the cubes in $\mathfrak{D}_k(E)$ are a covering of $E$.
		
		Set
		\begin{equation*}
			E_k\coloneqq \bigcup_{Q\in \mathfrak{D}_k(E)}\partial Q,
		\end{equation*}
		where $\partial Q$ denotes the boundary of a cube. For any $\theta\in\TT$ we have $\pi_\theta(Q)=\pi_\theta(\partial Q)$, and so
		\begin{equation}\label{eq:Favskeleton}
			\HH(\pi_\theta(E_k)) = \HH\left(\pi_\theta\bigg(\bigcup_{Q\in \mathfrak{D}_k(E)}Q\bigg)\right)\ge \HH(\pi_\theta(E)).
		\end{equation}
		In particular, we have $\Fav(E_k)\ge \Fav(E)\ge\kappa\HH(E)$.
		
		Observe that $E_k$ may be written as a disjoint union $E_k=F^1_k\cup F_k^2$, where each $F^i_k$ is a finite union of segments parallel to the $x_i$-axis. Taking into account that $\Fav(E_k)\le \Fav(F_k^1)+\Fav(F_k^2)$, we have $\Fav(F_k^i)\ge \tfrac{\kappa}{2}\HH(E)$ for some $i\in {1,2}$. Set $F_k\coloneqq F_k^i$.
		
		\begin{lemma}\label{lem:ARskeleton}
			The set $F_k$ is Ahlfors regular with constant $\sim A^2$, and we have $A^{-1}\HH(E)\lesssim \HH(F_k)\lesssim A\HH(E)$.
		\end{lemma}
		\begin{proof}
			We begin by proving Ahlfors regularity. Let $x\in F_k$. We will show that $A^{-2}r \lesssim \HH(F_k\cap B(x,r))\lesssim A^2r$.
			
			Assume first $0<r<C2^{-k}$. Since $F_k$ is a union of segments of length $2^{-k}$, and the ball $B(x,r)$ intersects at most $\lesssim 1$ such segments, we immediately get
			\begin{equation*}
				\HH(F_k\cap B(x,r))\sim r.
			\end{equation*}
			
			Now assume $C2^{-k}\le r\le \diam(F_k)$. Set
			\begin{align*}
				\cQ&\coloneqq \{Q\in\mathfrak{D}_k(E):Q\cap B(x,r)\neq\varnothing\},\\
				\cQ'&\coloneqq \{Q\in\mathfrak{D}_k(E):Q\subset B(x,r)\}.
			\end{align*}
			Note that if $C>1$ is chosen large enough, then every $Q\in\cQ$ is contained in $B(x,2r)$, and also every $Q\in\mathfrak{D}_k(E)$ intersecting $B(x,r/2)$ is in $\cQ'$. 
			
			Observe that for every $Q\in\mathfrak{D}_k(E)$ we have 
			\begin{equation*}
				A^{-1}2^{-k}\lesssim\HH(3Q\cap E)\lesssim A2^{-k}.
			\end{equation*}
			Since the family $\{3Q\}_{Q\in\mathfrak{D}_k(E)}$ has bounded intersections, it follows that
			\begin{equation*}
				\#\cQ \lesssim A2^{k}\sum_{Q\in\cQ}\HH(3Q\cap E)\lesssim A2^{k}\HH(E\cap B(x,2r))\lesssim A^{2}2^{k}r
			\end{equation*}
			and
			\begin{equation*}
				\#\cQ' \gtrsim A^{-1}2^{k}\sum_{Q\in\cQ'}\HH(3Q\cap E)\gtrsim A^{-1}2^{k}\HH(E\cap B(x,r/2))\gtrsim A^{-2}2^{k}r.
			\end{equation*}
			Thus,
			\begin{equation*}
				\HH(F_k\cap B(x,r))\le \sum_{Q\in\cQ}\HH(\partial Q)\lesssim A^22^kr\cdot 2^{-k}=A^2r
			\end{equation*}
			and
			\begin{equation*}
				\HH(F_k\cap B(x,r))\ge \frac{1}{2}\sum_{Q\in\cQ'}\HH(\partial Q)\gtrsim A^{-2}2^kr\cdot 2^{-k}=A^{-2}r.
			\end{equation*}
			Thus, $F_k$ is Ahlfors regular with constant $\sim A^2$.
			
			Concerning the length estimate, observe that
			\begin{equation*}
				\HH(E)=\sum_{Q\in\mathfrak{D}_k(E)}\HH(E\cap Q)\le A\sum_{Q\in\mathfrak{D}_k(E)}2^{-k}\sim A\HH(F_k)
			\end{equation*}
			and 
			\begin{equation*}
				\HH(E)\sim\sum_{Q\in\mathfrak{D}_k(E)}\HH(E\cap 3Q)\ge A^{-1}\sum_{Q\in\mathfrak{D}_k(E)}2^{-k}\sim A^{-1}\HH(F_k).
			\end{equation*}
		\end{proof}
		If follows from the lemma above that
		\begin{equation*}
			\Fav(F_k)\ge\frac{\kappa}{2}\HH(E)\gtrsim A^{-1}\kappa\HH(F_k).
		\end{equation*}
		
		Since each $F_k$ is an Ahlfors regular finite union of parallel segments with large Favard length, we may use our assumptions on functions $\phi$ and $\psi$ to obtain a family of Lipschitz graphs $\Gamma_k$ with $\lip(\Gamma_k)\le \phi(C^{-1}A^{-1}\kappa, CA^2)$ and 
		\begin{multline}\label{eq:Fklarge}
			\HH(F_k\cap \Gamma_k)\ge \psi(C^{-1}A^{-1}\kappa, CA^2)\cdot\HH(F_k)\\
			\gtrsim \psi(C^{-1}A^{-1}\kappa, CA^2)\cdot A^{-1}\HH(E).
		\end{multline}
		By choosing a subsequence and slightly increasing the Lipschitz constants, we may assume that every $\Gamma_k$ is a graph of a Lipschitz function $f_k:\R\to\R$ over a fixed line $\ell$ passing through the origin, and by applying a rotation, we may assume that
		\begin{equation*}
			\Gamma_k = \{(y,f_k(y)):y\in\R\}.
		\end{equation*}
		Since $F_k\subset [0,1]^2$, we may also assume that $f_k$ are bounded and compactly supported.
		
		By applying the Arzelà-Ascoli theorem, we may choose a subsequence of $f_k$ converging uniformly to another Lipschitz function $f:\R\to\R$ satisfying $\lip(f)\lesssim \phi(C^{-1}A^{-1}\kappa, CA^2)$. Let $\Gamma$ be the graph of $f$. To complete the proof of \lemref{lem:parallel2} it suffices to show that $E\cap\Gamma$ is large.
		\begin{lemma}
			We have 
			\begin{equation*}
				\HH(\Gamma\cap E)\gtrsim A^{-2}\psi(C^{-1}A^{-1}\kappa, CA^2)\cdot \HH(E).
			\end{equation*}
		\end{lemma}
		\begin{proof}
			Fix $m\ge 0$ and set			
			\begin{equation*}
				G_m\coloneqq \{x\in E : \dist(x,\Gamma)\le 2^{-m} \}.				
			\end{equation*} 
			Let $k_0\ge m+1$ be so large that $\|f-f_k\|_{\infty}\le 2^{-m-1}$ for all $k\ge k_0$.
			For any $k\ge k_0$ set 
			\begin{equation*}
				\cF_k\coloneqq \{Q\in\mathfrak{D}_k(E):F_k\cap \Gamma_k\cap\partial Q\neq\varnothing\}.
			\end{equation*}
			Fix $Q\in\cF_k$ and $x\in 3Q$. Observe that
			\begin{equation*}
				\dist(x,\Gamma_k)\le \diam(3Q)\lesssim 2^{-k}.
			\end{equation*}
			Let $y\in\R$ be such that
			\begin{equation*}
				\dist(x,\Gamma_k) = |x-(y,f_k(y))|\lesssim 2^{-k}.
			\end{equation*}
			Since $\|f-f_k\|_{\infty}\le 2^{-m-1}$, we get
			\begin{multline*}
				\dist(x,\Gamma) \le |x-(y,f(y))|\le |x-(y,f_k(y))|+|f(y)-f_k(y)|\\
				\le C2^{-k}+2^{-m-1}\le 2^{-m},
			\end{multline*}
			assuming $k\ge k_0$ is large enough. 			
			%			First, we claim that for any $k\ge 0$
			%			\begin{equation*}
				%			\HH(G_k^k)\gtrsim\HH(F_k\cap\Gamma_k)\gtrsim_{\kappa,A}1.
				%			\end{equation*}
			Since $x\in 3Q$ was arbitrary, we get that $3Q\cap E\subset G_m$, and so
			\begin{multline*}
				\HH(G_m)\gtrsim \sum_{Q\in\cF_k}\HH(3Q\cap E)\gtrsim A^{-1}\sum_{Q\in\cF_k}2^{-k}\gtrsim A^{-1}\HH(F_k\cap\Gamma_k)\\
				\overset{\eqref{eq:Fklarge}}{\gtrsim}A^{-2}\psi(C^{-1}A^{-1}\kappa, CA^2)\cdot \HH(E).
			\end{multline*}
			%			
			%			 Let $Q\in \cF_k$. As noticed earlier, for any $x\in Q\cap E$ we have $x\in G_k^k$. Since $k\ge k_0\ge m$, we have $x\in G_k^k\subset G_m^k.$ 
			%			
			%			Thus, $x\in G_m^\infty$. It follows that
			%			\begin{equation*}
				%			\HH(G_m^\infty)\ge \HH(G_k^k)\gtrsim_{\kappa,A} 1.
				%			\end{equation*}
			Since $G_m$ is a decreasing sequence of sets, we finally get that 
			\begin{equation*}
				\HH(\Gamma\cap E) = \HH\bigg(\bigcap_{m\ge 0}G_m\bigg)=\lim_{m\to\infty}\HH(G_m)\gtrsim A^{-2}\psi(C^{-1}A^{-1}\kappa, CA^2)\cdot \HH(E).
			\end{equation*}
		\end{proof}
		
		\section{Constant dependence in \thmref{thm:MOmainprop}}\label{app:const}
		In this section we explain briefly how to obtain \thmref{thm:MOmainprop}, which is a version of {\cite[Proposition 1.12]{martikainen2018characterising}} with explicit parameter dependence. We repeat the statement below.
		\begin{theorem}\label{thm:MOmainprop2}
			Assume that $\rho=1/2$, $E\subset B(0,1)$ is an Ahlfors regular set with constant $A$, $\diam(E)=1$, and $F\subset E$ is $\HH$-measurable with $\HH(F)\ge\tau$. Suppose that $J\subset\TT$ is an interval with $\HH(J)=\alpha\le c_J$ for some absolute $0<c_J<1$, and that for every $x\in F$ we have
			\begin{equation*}
			\#\Bad_{F}(x,J)\le M_0.
			\end{equation*}
			Then, there exists a Lipschitz graph $\Gamma$ with $\lip(\Gamma)\sim 2^{M_0}\alpha^{-1}$ and
			\begin{equation}\label{eq:MOquant2}
			\HH(F\cap \Gamma)\ge (c\alpha A^{-2}\tau)^{2^{M_0}},
			\end{equation}
			where $0<c<1$ is absolute.
		\end{theorem}
	
		The key technical proposition used in the proof of \cite[Proposition 1.12]{martikainen2018characterising} is \cite[Proposition 3.2]{martikainen2018characterising} stated below.
		\begin{prop}\label{prop:MOmainprop}
			Suppose that $F\subset E$ and $J\subset\TT$ are as in \thmref{thm:MOmainprop2}. Assume that we have
			\begin{equation*}
			\#\Bad_{F}(x,J)\le M\quad\text{for all $x\in F$.}
			\end{equation*}
			Then, there exists a compact set $K\subset F$ satisfying
			\begin{equation*}
			\#\Bad_{K}(x,0.5 J)\le M-1 \quad\text{for all $x\in K$}
			\end{equation*}
			and
			\begin{equation}\label{eq:MOquant3}
			\HH(K)\gtrsim_{\tau,\alpha,A} 1.
			\end{equation}
%			for some implicit $f:(0,1)\times (0,1)\times\R_+\to (0,1)$.
		\end{prop}	
		\begin{claim}
		Careful inspection of the proof of \cite[Proposition 3.2]{martikainen2018characterising} shows that in \eqref{eq:MOquant3} we actually have
		\begin{equation}\label{eq:fdef}
		\HH(K)\ge c_0 \alpha A^{-2}\tau^2
		\end{equation}
		for some absolute $0<c_0<1$.
		\end{claim}
		\begin{proof}
				The relevant argument is contained on pp. 1065--1068 of \cite{martikainen2018characterising}. The set $K$ is obtained by an iterative algorithm summarized on p. 1065, and which involves Stopping Conditions 3.3 and 3.4. Over the course of the algorithm sequences of disjoint sets $D^0,\dots, D^k$ and $S^0,\dots,S^k$ are defined.
			
			If the algorithm stops due to the Stopping Condition 3.4, then it is shown that
			\begin{equation*}
			\HH(K)\ge\HH(F)/4\ge\tau/4,
			\end{equation*}
			which is much better than \eqref{eq:fdef}.
			
			The alternative is that the algorithm stops due to the Stopping Condition 3.3, which claims that the sets $D^0,\dots, D^k$ and $S^0,\dots,S^k$ defined over the course of the iteration satisfy either
			\begin{equation}\label{eq:stop34}
			\sum_{i=0}^k\HH(D^i)\ge \HH(F)/2\quad\text{or}\quad \sum_{i=0}^k\HH(S^i)\ge \HH(F)/2.
			\end{equation}
			In both cases, we set $K=\bigcup_{i=0}^k S^i$. If the second estimate of \eqref{eq:stop34} holds, then using that $S^i$ are disjoint we get 
			\begin{equation*}
			\HH(K)=\sum_{i=0}^k\HH(S^i)\ge \HH(F)/2\ge \tau/2,
			\end{equation*}
			which is again much better than \eqref{eq:fdef}.
			
			The remaining alternative is that $\sum_{i=0}^k\HH(D^i)\ge \HH(F)/2$. Luckily, the sets $S^i$ satisfy 
			\begin{equation}\label{eq:SiDi}
			\HH(S^i)\gtrsim_{\tau,\alpha,A}\HH(D^i)
			\end{equation}
			see the last displayed inequality on p. 1067, and so
			\begin{equation}\label{eq:Kcomp}
			\HH(K)=\sum_{i=0}^k\HH(S^i)\gtrsim_{\tau,\alpha,A} \sum_{i=0}^k\HH(D^i)\ge \HH(F)/2\ge \tau/2.
			\end{equation}
			It remains to make the estimate \eqref{eq:SiDi} explicit.
			
			The set $S^k$ is defined in the middle of p. 1067 as $S^k\coloneqq B(x_k, r_k)\cap F^{k,M}$. The sets $F^{k,M}$ are defined in (IV) on p. 1065. The point $x_k\in F^{k,M}$ satisfies 
			\begin{equation}\label{eq:Skest}
			\HH(S^k)=\HH(B(x_k, r_k)\cap F^{k,M})\ge\epsilon r_k
			\end{equation}
			for certain $\epsilon=\epsilon(\tau,A)>0$ fixed in the Stopping Condition 3.4. The radius $r_k$ is defined as $r_k=c 2^{-j_k}$ for certain $c=c(\alpha)$ fixed on p. 1068, where $j_k\in\mathbb{N}$ is one of the ``bad scales''. One can easily check that $c(\alpha)\sim \alpha$. 
			
			Concerning $\epsilon$, its definition in the Stopping Condition 3.4 invokes \cite[Lemma 2.1]{martikainen2018characterising}. Together with the (very short) proof of that lemma it is clear that one may take $\epsilon\sim A^{-1}\HH(F)$.\footnote{Literally using \cite[Lemma 2.1]{martikainen2018characterising} yields a worse bound, but a modified version of \cite[Lemma 2.1]{martikainen2018characterising}, where the assumption $\HH(E)\ge c>0$ is replaced by $\diam(E)=1$ (which is what we assume in \propref{prop:MOmainprop}), gives the claimed estimate. Modifying the proof is trivial.} Thus, \eqref{eq:Skest} becomes
			\begin{equation}\label{eq:Skest2}
			\HH(S^k)\gtrsim A^{-1}\alpha 2^{-j_k}\HH(F)\ge A^{-1}\alpha \tau 2^{-j_k}.
			\end{equation}
			On the other hand, the set $D^k\subset F\subset E$ is contained in a single ball of radius $\lesssim 2^{-j_k}$, see the bottom of p. 1067. Thus, $\HH(D^k)\lesssim A2^{-j_k}$. Together with \eqref{eq:Skest2} we get
			\begin{equation*}
			\HH(S^k)\gtrsim A^{-2}\alpha \tau\HH(D^k).
			\end{equation*}
			Recalling the computation \eqref{eq:Kcomp}, we finally arrive at 
			\begin{equation*}
			\HH(K)\gtrsim A^{-2}\alpha \tau^2,
			\end{equation*}
			which is what we claimed in \eqref{eq:fdef}.
		\end{proof}
		
		Having \propref{prop:MOmainprop} and \eqref{eq:fdef} at our disposal, it is easy to get \thmref{thm:MOmainprop2} by applying \propref{prop:MOmainprop} iteratively.
		\begin{proof}[Proof of \thmref{thm:MOmainprop2}]
			Set $K_0=F$. Applying \propref{prop:MOmainprop} with $K_0$ produces a set $K_1\subset K_0$ with
			\begin{equation*}
			\#\Bad_{K_1}(x,2^{-1} J)\le M_0-1 \quad\text{for all $x\in K_1$}
			\end{equation*}
			and
			\begin{equation}\label{eq:MOquant4}
			\HH(K_1)\ge \tau_1\coloneqq c_0\alpha A^{-2}\tau^2.
			\end{equation}
			More generally, after $1\le j\le M_0$ applications of \propref{prop:MOmainprop} we get a compact set $K_j\subset K_{j-1}\subset\dots\subset K_0$ with
			\begin{equation*}
			\#\Bad_{K_j}(x,2^{-j} J)\le M_0-j \quad\text{for all $x\in K_j$}
			\end{equation*}
			and
			\begin{align*}\label{eq:MOquant5}
			\HH(K_j)\ge \tau_j&\coloneqq c_02^{-j}\alpha A^{-2}\tau_{j-1}^2\\
			&=c_02^{-j}\alpha A^{-2}(c_02^{-j-1}\alpha A^{-2}\tau_{j-2}^2)^2\\
			&=(c_0\alpha A^{-2})^{1+2+\dots+2^{j-1}}2^{-j-2(j-1)-\dots-2^{j-1}\cdot 1}\tau^{2^{j}}\\
			&\ge (c_0\alpha A^{-2})^{2^j}2^{-2^{j+1}}\tau^{2^{j}} =  (c_0'\alpha A^{-2}\tau)^{2^{j}},
			\end{align*}
			where in the last inequality we used the fact that $\sum_{i=0}^{j-1} 2^i (j-i)\le 2^{j+1}$.
			
			Note that the set $G\coloneqq K_{M_0}\subset F$ satisfies 
			\begin{equation*}
			\#\Bad_{G}(x,2^{-M_0} J)=0 \quad\text{for all $x\in G$},
			\end{equation*}
			which means that $G$ is contained in a Lipschitz graph $\Gamma$ with $\lip(\Gamma)\lesssim 2^{M_0}\alpha^{-1}$. We also have
			\begin{equation*}
			\HH(F\cap \Gamma)\ge\HH(G)\ge (c_0'\alpha A^{-2}\tau)^{2^{M_0}}.
			\end{equation*}
		\end{proof}
		\begin{remark}\label{rem:expMO}
			If we followed the proof of \cite[Proposition 1.12]{martikainen2018characterising} more closely, we would obtain a Lipschitz graph $\Gamma$ with $\lip(\Gamma)\sim \alpha^{-1}$. However, this would entail applying \propref{prop:MOmainprop} roughly $2^{M_0}$ times, instead of $M_0$ times, and we would get another exponential loss in \eqref{eq:MOquant2}.
		\end{remark}

		\section{Intersections of lower content regular sets}\label{app:LCR}
		We say that $E\subset \R^2$ is lower content regular if $\HH_\infty(E\cap B(x,r))\gtrsim r$ for all $x\in E$, $0<r<\diam(E)$. It is easy to see that curves and Ahlfors regular sets are lower content regular.
			\begin{lemma}\label{lem:contentstuff}
				Suppose that $E\subset\R^2$ is lower content regular, and $\Gamma\subset\R^2$ is arbitrary. Then, for $0<\delta<\diam(E)$
				\begin{equation*}
				\HH_\infty(E\cap \Gamma(3\delta))\gtrsim \HH_\infty(E(\delta)\cap \Gamma).
				\end{equation*}
			\end{lemma}
			\begin{remark}\label{rem:contentstuff}
				In particular, if $E$ is Ahlfors regular and $\Gamma$ is a curve, then
				\begin{equation*}
				\HH_\infty(E\cap \Gamma(3\delta))\gtrsim \HH_\infty(E(\delta)\cap \Gamma)\gtrsim\HH_\infty(E\cap \Gamma(\delta/3)).
				\end{equation*}
%				This also shows that for lower content regular sets $E$ one can replace $\HH(E\cap\Gamma(\delta))$ by $\HH(E(\delta)\cap\Gamma)$ in \eqref{eq:ell}
			\end{remark}
			\begin{proof}
				Let $\gamma=E(\delta)\cap\Gamma$ and $F=E\cap \gamma(3\delta)\subset E\cap\Gamma(3\delta).$ We will show that $\HH_\infty(F)\gtrsim \HH_\infty(\gamma)$. 
				
				Without loss of generality assume $\HH_\infty(F)<\infty$, otherwise there is nothing to prove. Let $\{B_i=B(x_i,r_i)\}_{i\in I}$ be a covering of $F$ consisting of balls such that $\sum_{i\in I}r_i\lesssim \HH_\infty(F)$. By slightly increasing the radii, we may assume that $x_i\in F$ for all $i\in I$. Our goal is to modify this family of balls to obtain a nice covering of $\gamma$.
				
				The main observation is that we may tweak the covering $\{B_i\}_{i\in I}$ a little to get a covering $\cB''$ such that $\sum_{B\in\cB''} r(B)\lesssim\HH_\infty(F)$ and $r(B)\ge \delta$ for all $B\in\cB''$. To see that this is the case, let $I_{small}=\{i\in I : r_i<\delta\}$, $I_{big}=I\setminus I_{small}$, and
				\begin{equation*}
				I_{bad} = \{i\in I_{small} : B(x_i, 10\delta)\cap B_j=\varnothing\quad\text{for all } j\in I_{big}\}.
				\end{equation*}
				It follows from the triangle inequality that 
				\begin{equation*}
				\bigcup_{i\in I_{small}\setminus I_{bad}}B_i\subset \bigcup_{i\in I_{big}} 20 B_i.
				\end{equation*}
				
				Consider
				\begin{equation*}
				\cB=\{B(x_i,5\delta) : i\in I_{bad} \}\cup\{20 B_i : i\in I_{big} \},
				\end{equation*}
				so that $\cB$ is still a cover of $F$.
				
				By the 5$r$-covering lemma, we may find a family of disjoint balls $\cB'\subset \cB$ such that $\{5B\}_{B\in\cB'}$ covers $F$. We claim that
				\begin{equation*}
				\sum_{B\in\cB'}r(B)\lesssim \sum_{i\in I}r_i\lesssim \HH_\infty(F).
				\end{equation*}
				Let $I_0\subset I_{big}$ be such that $20B_i\in \cB'$ for $i\in I_0$, and $I_1\subset I_{bad}$ be such that $B(x_i, 5\delta)\in\cB'$ for $i\in I_1$. It is clear that $\sum_{i\in I_0}r(20B_i)\lesssim \sum_{i\in I}r_i$, so we only need to show that
				\begin{equation}\label{eq:rggr}
				\sum_{i\in I_1}\delta =\#I_1\cdot \delta \lesssim \sum_{i\in I}r_i.
				\end{equation}
				
				Recalling that $F=E\cap\gamma(3\delta)$ and $\gamma= E(\delta)\cap \Gamma$, and also that $x_i\in F$, we get that for every $i$ there exists $y_i\in \gamma$ such that $|x_i-y_i|\le 3\delta$, and  $z_i\in E$ such that $|y_i-z_i|\le \delta$. Then, for any $i\in I$ we have
				\begin{multline*}
				\HH_\infty(B(x_i, 10\delta)\cap F)\ge\HH_\infty(B(y_i, 5\delta)\cap F)=\HH_\infty(B(y_i, 5\delta)\cap E\cap \gamma(3\delta))\\
				\ge\HH_\infty(B(y_i, 3\delta)\cap E)
				\ge\HH_\infty(B(z_i,2\delta)\cap E)\gtrsim \delta,
				\end{multline*}
				where in the last line we used the lower content regularity of $E$.
				At the same time, since $\{B_j\}_{j\in I}$ is a cover of $F$, for $i\in I_{bad}$ we get
				%		since $\{B(x_i,10\delta)\}_{i\in I''}$ have finite overlap, we have
				\begin{equation*}
				\delta\lesssim\HH_\infty(B(x_i, 10\delta)\cap F)\le \sum_{j\in I : B_j\cap B(x_i, 10\delta)\neq\varnothing} r_j = \sum_{j\in I_{small} : B_j\cap B(x_i, 10\delta)\neq\varnothing} r_j,
				\end{equation*}
				where in the last equality we used the fact that $i\in I_{bad}$. 
				
				For every $i\in I_1\subset I_{bad}$ set $J(i)=\{j\in I_{small} : B_j\cap B(x_i, 10\delta)\neq\varnothing\}$. Recalling that for $j\in I_{small}$ we have $r(B_j)=r_j<\delta$, it follows that if $j\in J(i)$, then $B_j\subset B(x_i, 20\delta)$. Since the balls $\{B(x_i, 5\delta)\}_{i\in I_1}$ are disjoint, the balls $\{B(x_i, 20\delta)\}_{i\in I_1}$ have bounded overlap. Thus, every $j\in I_{small}$ belongs to at most a bounded number of families $J(i), i\in I_1$. This gives
				\begin{equation*}
				\sum_{i\in I_1}\delta\lesssim \sum_{i\in I_1}\sum_{j\in J(i)}r_j\lesssim \sum_{j\in I_{small}}r_j,
				\end{equation*}
				which is the desired estimate \eqref{eq:rggr}. Now, $\cB''\coloneqq\{5B\}_{B\in\cB'}$ is a cover of $F$ satisfying $\sum_{B\in\cB''}r(B)\lesssim \HH_\infty(F)$ and $r(B)\ge\delta$ for all $B\in\cB''$.
				
				To finish the proof, observe that $\{2B:B\in\cB''\}$ is a covering of $\gamma=E(\delta)\cap\Gamma$. Indeed, for every $x\in \gamma$ there exists $y\in E$ with $|x-y|\le \delta$. In particular, $y\in E\cap \gamma(\delta)\subset F$. Taking $B$ such that $y\in B$ we immediately get $x\in 2B$ because $r(B)\ge \delta$. Hence, by the definition of Hausdorff content
				\begin{equation*}
					\HH_\infty(\gamma)\le \sum_{B\in\cB''}r(2B)\lesssim\HH_\infty(F).
				\end{equation*}
			\end{proof}

	\newcommand{\etalchar}[1]{$^{#1}$}
	

\begin{thebibliography}{CDOV24}
		\expandafter\ifx\csname url\endcsname\relax
		\def\url#1{\texttt{#1}}\fi
		\expandafter\ifx\csname doi\endcsname\relax
		\def\doi#1{\burlalt{doi:#1}{http://dx.doi.org/#1}}\fi
		\expandafter\ifx\csname urlprefix\endcsname\relax\def\urlprefix{URL }\fi
		\expandafter\ifx\csname href\endcsname\relax
		\def\href#1#2{#2}\fi
		\expandafter\ifx\csname burlalt\endcsname\relax
		\def\burlalt#1#2{\href{#2}{#1}}\fi
		
		\bibitem[Bat20]{bate2020purely}
		D.~Bate.
		\newblock {Purely unrectifiable metric spaces and perturbations of Lipschitz
			functions}.
		\newblock {\em Acta Math.}, 224(1):1--65, 2020.
		\newblock \doi{10.4310/ACTA.2020.v224.n1.a1}.
		
		\bibitem[Bes39]{besicovitch1939on}
		A.~S. Besicovitch.
		\newblock On the fundamental geometrical properties of linearly measurable
		plane sets of points {(III)}.
		\newblock {\em Math. Ann.}, 116(1):349--357, Dec 1939.
		\newblock \doi{10.1007/BF01597361}.
		
		\bibitem[BHH{\etalchar{+}}22]{bortz2022coronizations}
		S.~Bortz, J.~Hoffman, S.~Hofmann, J.~L. Luna-Garcia, and
		K.~Nystr{\ifmmode\ddot{o}\else\"{o}\fi}m.
		\newblock {Coronizations and big pieces in metric spaces}.
		\newblock {\em Annales de l'Institut Fourier}, 72(5):2037--2078, 2022.
		\newblock \doi{10.5802/aif.3518}.
		
		\bibitem[B{\L}V14]{bond2014buffon}
		M.~Bond, I.~{\L}aba, and A.~Volberg.
		\newblock Buffon’s needle estimates for rational product {Cantor} sets.
		\newblock {\em Amer. J. Math.}, 136(2):357--391, 2014.
		\newblock \doi{10.1353/ajm.2014.0013}.
		
		\bibitem[BN21]{badger2020radon}
		M.~Badger and L.~Naples.
		\newblock {Radon} measures and {Lipschitz} graphs.
		\newblock {\em Bull. Lond. Math. Soc.}, 53(3):921--936, 2021.
		\newblock \doi{10.1112/blms.12473}.
		
		\bibitem[Bon19]{bongers2019geometric}
		R.~Bongers.
		\newblock {Geometric bounds for Favard length}.
		\newblock {\em Proc. Amer. Math. Soc.}, 147(4):1447--1452, 2019.
		\newblock \doi{10.1090/proc/14358}.
		
		\bibitem[BT23]{bongers2021transversal}
		R.~Bongers and K.~Taylor.
		\newblock Transversal families of nonlinear projections and generalizations of
		{Favard} length.
		\newblock {\em Anal. PDE}, 16(1):279--308, 2023.
		\newblock \doi{10.2140/apde.2023.16.279}.
		
		\bibitem[BV10a]{bateman2010estimate}
		M.~Bateman and A.~Volberg.
		\newblock An estimate from below for the {Buffon} needle probability of the
		four-corner {Cantor} set.
		\newblock {\em Math. Res. Lett.}, 17(5):959--967, 2010.
		\newblock \doi{10.4310/MRL.2010.v17.n5.a12}.
		
		\bibitem[BV10b]{bond2010buffon}
		M.~Bond and A.~Volberg.
		\newblock {Buffon} needle lands in $\epsilon$-neighborhood of a $1$-dimensional
		{Sierpinski} {Gasket} with probability at most $|\log\epsilon|^{-c}$.
		\newblock {\em C. R. Math.}, 348(11):653--656, 2010.
		\newblock \doi{10.1016/j.crma.2010.04.006}.
		
		\bibitem[BV11]{bond2011circular}
		M.~Bond and A.~Volberg.
		\newblock {Circular Favard Length of the Four-Corner Cantor Set}.
		\newblock {\em J. Geom. Anal.}, 21(1):40--55, 2011.
		\newblock \doi{10.1007/s12220-010-9141-4}.
		
		\bibitem[Cal77]{calderon1977cauchy}
		A.~P. Calder{\ifmmode\acute{o}\else\'{o}\fi}n.
		\newblock Cauchy integrals on {Lipschitz} curves and related operators.
		\newblock {\em Proc. Natl. Acad. Sci. U.S.A.}, 74(4):1324--1327, 1977.
		\newblock \doi{10.1073/pnas.74.4.1324}.
		
		\bibitem[CDOV24]{chang2022structure}
		A.~Chang, D.~D{\k{a}}browski, T.~Orponen, and M.~Villa.
		\newblock Structure of sets with nearly maximal {Favard} length.
		\newblock {\em Anal. PDE}, 17(4):1473--1500, 2024.
		\newblock \doi{10.2140/apde.2024.17.1473}.
		
		\bibitem[CDT20]{cladek2020upper}
		L.~Cladek, B.~Davey, and K.~Taylor.
		\newblock Upper and lower bounds on the rate of decay of the {Favard} curve
		length for the four-corner {Cantor} set.
		\newblock {\em To appear in Indiana Univ. Math. J.}, 2020.
		\newblock \doi{10.48550/arXiv.2003.03620}.
		
		\bibitem[Chr90]{christ1990tb}
		M.~Christ.
		\newblock A {$T(b)$} theorem with remarks on analytic capacity and the {Cauchy}
		integral.
		\newblock {\em Colloq. Math.}, 60:601--628, 1990.
		\newblock \doi{10.4064/cm-60-61-2-601-628}.
		
		\bibitem[CT20]{chang2017analytic}
		A.~Chang and X.~Tolsa.
		\newblock Analytic capacity and projections.
		\newblock {\em J. Eur. Math. Soc. (JEMS)}, 22(12):4121--4159, 2020.
		\newblock \doi{10.4171/JEMS/1004}.
		
		\bibitem[D{\k{a}}b22a]{dabrowski2020cones}
		D.~D{\k{a}}browski.
		\newblock Cones, rectifiability, and singular integral operators.
		\newblock {\em Rev. Mat. Iberoam.}, 38(4):1287--1334, 2022.
		\newblock \doi{10.4171/RMI/1301}.
		
		\bibitem[D{\k{a}}b22b]{dabrowski2022quantitative}
		D.~D{\k{a}}browski.
		\newblock Quantitative Besicovitch projection theorem for irregular sets of
		directions.
		\newblock {\em Preprint}, 2022,
		\burlalt{arXiv:2211.16911}{http://arxiv.org/abs/2211.16911}.
		
		\bibitem[D{\k{a}}b22c]{dabrowski2020two}
		D.~D{\k{a}}browski.
		\newblock Two examples related to conical energies.
		\newblock {\em Ann. Fenn. Math.}, 47(1):261--281, 2022.
		\newblock \doi{10.54330/afm.113378}.
		
		\bibitem[Dav98]{david1998unrectifiable}
		G.~David.
		\newblock Unrectifiable 1-sets have vanishing analytic capacity.
		\newblock {\em Rev. Mat. Iberoam.}, 14(2):369--479, 1998.
		\newblock \doi{10.4171/RMI/242}.
		
		\bibitem[DM00]{david2000removable}
		G.~David and P.~Mattila.
		\newblock Removable sets for {Lipschitz} harmonic functions in the plane.
		\newblock {\em Rev. Mat. Iberoam.}, 16(1):137--215, 2000.
		\newblock \doi{10.4171/RMI/272}.
		
		\bibitem[DS91]{david1991singular}
		G.~David and S.~Semmes.
		\newblock Singular integrals and rectifiable sets in {$\mathbb{R}^n$}:
		Au-del\`{a} des graphes lipschitziens.
		\newblock {\em Ast{\'e}risque}, 193, 1991.
		\newblock \doi{10.24033/ast.68}.
		
		\bibitem[DS93a]{david1993analysis}
		G.~David and S.~Semmes.
		\newblock {\em Analysis of and on Uniformly Rectifiable Sets}, volume~38 of
		{\em Math. Surveys Monogr.}
		\newblock Amer. Math. Soc., Providence, RI, 1993.
		
		\bibitem[DS93b]{david1993quantitative}
		G.~David and S.~Semmes.
		\newblock {Quantitative rectifiability and Lipschitz mappings}.
		\newblock {\em Trans. Amer. Math. Soc.}, 337(2):855--889, 1993.
		\newblock \doi{10.1090/S0002-9947-1993-1132876-8}.
		
		\bibitem[DT22]{davey2022quantification}
		B.~Davey and K.~Taylor.
		\newblock {A Quantification of a Besicovitch Non-linear Projection Theorem via
			Multiscale Analysis}.
		\newblock {\em J. Geom. Anal.}, 32(4):138--55, 2022.
		\newblock \doi{10.1007/s12220-021-00793-z}.
		
		\bibitem[DV22]{dabrowski2022analytic}
		D.~D\k{a}browski and M.~Villa.
		\newblock Analytic capacity and dimension of sets with plenty of big
		projections.
		\newblock {\em To appear in Trans. Amer. Math. Soc.}, 2022.
		\newblock \doi{10.48550/arXiv.2204.05804}.
		
		\bibitem[Fed47]{federer1947varphi}
		H.~Federer.
		\newblock The ($\varphi$, k) rectifiable subsets of $n$ space.
		\newblock {\em Trans. Amer. Math. Soc.}, 62(1):114--192, 1947.
		\newblock \doi{10.2307/1990632}.
		
		\bibitem[Gra14]{grafakos2014classical}
		L.~Grafakos.
		\newblock {\em Classical {Fourier} Analysis}, volume 249 of {\em Grad. Texts in
			Math.}
		\newblock Springer, New York, 3rd edition, 2014.
		\newblock \doi{10.1007/978-1-4939-1194-3}.
		
		\bibitem[HJJL12]{hovila2012besicovitch}
		R.~Hovila,
		E.~J{\ifmmode\ddot{a}\else\"{a}\fi}rvenp{\ifmmode\ddot{a}\else\"{a}\fi}{\ifmmode\ddot{a}\else\"{a}\fi},
		M.~J{\ifmmode\ddot{a}\else\"{a}\fi}rvenp{\ifmmode\ddot{a}\else\"{a}\fi}{\ifmmode\ddot{a}\else\"{a}\fi},
		and F.~Ledrappier.
		\newblock {Besicovitch-Federer projection theorem and geodesic flows on Riemann
			surfaces}.
		\newblock {\em Geom. Dedicata}, 161(1):51--61, 2012.
		\newblock \doi{10.1007/s10711-012-9693-5}.
		
		\bibitem[JM88]{jones1988positive}
		P.~W. Jones and T.~Murai.
		\newblock Positive analytic capacity but zero {Buffon} needle probability.
		\newblock {\em Pacific J. Math.}, 133(1):99--114, 1988.
		\newblock \doi{10.2140/pjm.1988.133.99}.
		
		\bibitem[Jon91]{jones1991traveling}
		P.~W. Jones.
		\newblock The Traveling Salesman Problem and Harmonic analysis.
		\newblock {\em Publ. Mat.}, 35(1):259--267, 1991.
		\newblock \doi{10.5565/PUBLMAT_35191_12}.
		
		\bibitem[{\L}ab14]{laba2014recent}
		I.~{\L}aba.
		\newblock {Recent Progress on Favard Length Estimates for Planar Cantor Sets}.
		\newblock In {\em {Operator-Related Function Theory and Time-Frequency
				Analysis}}, volume~9 of {\em Abel Symposia}, pages 117--145. Springer, Cham,
		2014.
		\newblock \doi{10.1007/978-3-319-08557-9_5}.
		
		\bibitem[{\L}M22]{laba2022vanishing}
		I.~{\L}aba and C.~Marshall.
		\newblock {Vanishing sums of roots of unity and the Favard length of
			self-similar product sets}.
		\newblock {\em Discrete Anal.}, 19, 2022.
		\newblock \doi{10.19086/da.57602}.
		
		\bibitem[{\L}Z10]{laba2010favard}
		I.~{\L}aba and K.~Zhai.
		\newblock {The Favard length of product Cantor sets}.
		\newblock {\em Bull. London Math. Soc.}, 42(6):997--1009, 2010.
		\newblock \doi{10.1112/blms/bdq059}.
		
		\bibitem[Mat86]{mattila1986smooth}
		P.~Mattila.
		\newblock {Smooth Maps, Null-Sets for Integralgeometric Measure and Analytic
			Capacity}.
		\newblock {\em Ann. Of Math.}, 123(2):303--309, 1986.
		\newblock \doi{10.2307/1971273}.
		
		\bibitem[Mat90]{mattila1990orthogonal}
		P.~Mattila.
		\newblock {Orthogonal Projections, Riesz Capacities, and Minkowski Content}.
		\newblock {\em Indiana Univ. Math. J.}, 39(1):185--198, 1990.
		\newblock \doi{10.1512/iumj.1990.39.39011}.
		
		\bibitem[Mat95]{mattila1999geometry}
		P.~Mattila.
		\newblock {\em Geometry of sets and measures in {Euclidean} spaces: fractals
			and rectifiability}, volume~44 of {\em Cambridge Stud. Adv. Math.}
		\newblock Cambridge Univ. Press, Cambridge, UK, 1995.
		\newblock \doi{10.1017/CBO9780511623813}.
		
		\bibitem[Mat04]{mattila2004hausdorff}
		P.~Mattila.
		\newblock Hausdorff dimension, projections, and the {Fourier} transform.
		\newblock {\em Publ. Mat.}, 48(1):3--48, 2004.
		\newblock \doi{10.5565/PUBLMAT_48104_01}.
		
		\bibitem[MO18]{martikainen2018characterising}
		H.~Martikainen and T.~Orponen.
		\newblock Characterising the big pieces of {Lipschitz} graphs property using
		projections.
		\newblock {\em J. Eur. Math. Soc. (JEMS)}, 20(5):1055--1073, 2018.
		\newblock \doi{10.4171/JEMS/782}.
		
		\bibitem[Mur90]{murai1990formula}
		T.~Murai.
		\newblock {A formula for analytic separation capacity}.
		\newblock {\em Kodai Math. J.}, 13(2):265--288, 1990.
		\newblock \doi{10.2996/kmj/1138039223}.
		
		\bibitem[NPV11]{nazarov2011power}
		F.~Nazarov, Y.~Peres, and A.~Volberg.
		\newblock The power law for the {Buffon} needle probability of the four-corner
		{Cantor} set.
		\newblock {\em St. Petersburg Math. J.}, 22(1):61--72, 2011.
		\newblock \doi{10.1090/S1061-0022-2010-01133-6}.
		
		\bibitem[Orp21]{orponen2020plenty}
		T.~Orponen.
		\newblock Plenty of big projections imply big pieces of {Lipschitz} graphs.
		\newblock {\em Invent. Math.}, 226(2):653--709, 2021.
		\newblock \doi{10.1007/s00222-021-01055-z}.
		
		\bibitem[Pom60]{pommerenke1960uber}
		C.~Pommerenke.
		\newblock {Uber die analytische Kapazit{\ifmmode\ddot{a}\else\"{a}\fi}t}.
		\newblock {\em Arch. Math. (Basel)}, 11(1):270--277, 1960.
		\newblock \doi{10.1007/BF01236943}.
		
		\bibitem[PS02]{peres2002howlikely}
		Y.~Peres and B.~Solomyak.
		\newblock {How likely is Buffon{'}s needle to fall near a planar Cantor set?}
		\newblock {\em Pacific J. Math.}, 204(2):473--496, 2002.
		\newblock \doi{10.2140/pjm.2002.204.473}.
		
		\bibitem[Ste70]{stein1970singular}
		E.~M. Stein.
		\newblock {\em Singular Integrals and Differentiability Properties of
			Functions}, volume~30 of {\em Princeton Math. Ser.}
		\newblock Princeton Univ. Press, Princeton, NJ, 1970.
		
		\bibitem[Tao09]{tao2009quantitative}
		T.~Tao.
		\newblock A quantitative version of the {Besicovitch} projection theorem via
		multiscale analysis.
		\newblock {\em Proc. London Math. Soc.}, 98(3):559--584, 2009.
		\newblock \doi{10.1112/plms/pdn037}.
		
		\bibitem[Tol03]{tolsa2003painleve}
		X.~Tolsa.
		\newblock Painlev{\'e}'s problem and the semiadditivity of analytic capacity.
		\newblock {\em Acta Math.}, 190(1):105--149, 2003.
		\newblock \doi{10.1007/BF02393237}.
		
		\bibitem[Tol14]{tolsa2014analytic}
		X.~Tolsa.
		\newblock {\em Analytic capacity, the {Cauchy} transform, and non-homogeneous
			{Calder{\'o}n}-{Zygmund} theory}, volume 307 of {\em Progr. Math.}
		\newblock Birkhäuser, Cham, 2014.
		\newblock \doi{10.1007/978-3-319-00596-6}.
		
		\bibitem[Ver21]{verdera2021birth}
		J.~Verdera.
		\newblock Birth and life of the $L^{2}$ boundedness of the {Cauchy} {Integral}
		on {Lipschitz} graphs.
		\newblock {\em To appear in "Selected works of Yves Meyer"}, 2021,
		\burlalt{arXiv:2109.06690}{http://arxiv.org/abs/2109.06690}.
		\newblock \urlprefix\url{https://arxiv.org/abs/2109.06690}.
		
		\bibitem[Vit67]{vitushkin1967analytic}
		A.~G. Vitushkin.
		\newblock The analytic capacity of sets in problems of approximation theory.
		\newblock {\em Russian Math. Surveys}, 22(6):139--200, 1967.
		\newblock \doi{10.1070/rm1967v022n06abeh003763}.
		
		\bibitem[VV24]{Vardakis2024buffon}
		D.~Vardakis and A.~Volberg.
		\newblock {The Buffon's needle problem for random planar disk-like Cantor
			sets}.
		\newblock {\em J. Math. Anal. Appl.}, 529(2):127622, 2024.
		\newblock \doi{10.1016/j.jmaa.2023.127622}.
		
		\bibitem[Whi98]{white1998newproof}
		B.~White.
		\newblock A new proof of {Federer's} structure theorem for $k$-dimensional
		subsets of {$\mathbf{R}^N$}.
		\newblock {\em J. Amer. Math. Soc.}, 11(3):693--701, 1998.
		\newblock \doi{10.1090/S0894-0347-98-00267-7}.
		
		\bibitem[Wil17]{wilson2017sets}
		B.~Wilson.
		\newblock {Sets with Arbitrarily Slow Favard Length Decay}.
		\newblock {\em Preprint}, 2017.
		\newblock \doi{10.48550/arXiv.1707.08137}.
		
	\end{thebibliography}
\end{document}